\documentclass[12pt,a4paper]{amsart}
\usepackage[utf8]{inputenc}
\usepackage[english]{babel}
\usepackage{amsmath}
\usepackage{amsfonts}
\usepackage{amssymb}
\usepackage[dvipsnames]{xcolor}
\usepackage{mathrsfs} 
\usepackage{tikz}
\usepackage{pgfplots}
\usepackage{float}
\floatstyle{boxed}
\restylefloat{figure}
\usepackage{amsthm}
\theoremstyle{theorem}
\newtheorem{theorem}{Theorem}[section]
\newtheorem{proposition}[theorem]{Proposition}
\newtheorem{lemma}[theorem]{Lemma}
\newtheorem{condition}[theorem]{Condition}
\newtheorem{remark}[theorem]{Remark}
\newtheorem{corollary}[theorem]{Corollary}

\usepackage[paperwidth=20.5cm,paperheight=29.7cm,width=16.5cm,height=25.5cm,top=24mm,headsep=11mm]{geometry}

\numberwithin{equation}{section}

\begin{document}
\title[Homogenization of hyperbolic systems]{On homogenization of the first initial-boundary value
problem for periodic hyperbolic systems}
\author{Yu. M. Meshkova}
\thanks{Research is supported by the Russian Science Foundation grant no.~14-21-00035.}
\keywords{Periodic differential operators, hyperbolic systems, homogenization, operator error estimates.}
\address{Chebyshev Laboratory, St. Petersburg State University, 14th Line V.O., 29B, St.~Petersburg, 199178, Russia}
\email{y.meshkova@spbu.ru,\quad juliavmeshke@yandex.ru}
\subjclass[2010]{Primary 35B27. Secondary 	35L53}

\begin{abstract}
Let $\mathcal{O}\subset\mathbb{R}^d$ be a bounded domain of class  $C^{3,1}$. 
In $L_2(\mathcal{O};\mathbb{C}^n)$, we consider a~self-adjoint matrix strongly elliptic second order 
differential operator  $B_{D,\varepsilon}$, $0<\varepsilon \leqslant 1$, with the Dirichlet boundary condition. 
The coefficients of the operator $B_{D,\varepsilon}$ are periodic and depend on $\mathbf{x}/\varepsilon$. 
We are interested in the behavior of the operators $\cos(tB_{D,\varepsilon}^{1/2})$ and 
$B_{D,\varepsilon} ^{-1/2}\sin (t B_{D,\varepsilon} ^{1/2})$, $t\in\mathbb{R}$, in the small period limit. 
For these operators, approximations in the norm of operators acting from a certain subspace $\mathcal{H}$ 
of the Sobolev space $H^4(\mathcal{O};\mathbb{C}^n)$ to $L_2(\mathcal{O};\mathbb{C}^n)$ are found. 
Moreover, for $B_{D,\varepsilon} ^{-1/2}\sin (t B_{D,\varepsilon} ^{1/2})$, the approximation 
with the corrector in the norm of operators acting from $\mathcal{H}\subset H^4(\mathcal{O};\mathbb{C}^n)$ 
to $H^1(\mathcal{O};\mathbb{C}^n)$ is obtained.  The results are applied to homogenization 
for the solution of the first initial-boundary value problem for the hyperbolic equation 
$\partial ^2_t \mathbf{u}_\varepsilon =-B_{D,\varepsilon} \mathbf{u}_\varepsilon $.
\end{abstract}

\maketitle

\tableofcontents

\section*{Introduction}

The paper concerns homogenization theory of periodic differential operators (DO's). A broad literature is devoted to homogenization theory, see, e.~g., the books \cite{BaPa,BeLP,ZhKO,Sa}.

\subsection{Problem setting} 

Let $\Gamma\subset\mathbb{R}^d$ be a lattice and let $\Omega$ be the cell of the lattice $\Gamma$. For $\Gamma$-periodic functions in $\mathbb{R}^d$, we use the following notation $f^\varepsilon (\mathbf{x}):=f(\varepsilon ^{-1}\mathbf{x})$, $\varepsilon >0$.  
Let $\mathcal{O}\subset\mathbb{R}^d$ be a bounded domain of class $C^{1,1}$. In $L_2(\mathcal{O};\mathbb{C}^n)$, we consider a self-adjoint matrix strongly elliptic second order differential operator ${B}_{D,\varepsilon}$, $0<\varepsilon\leqslant 1$,  with the Dirichlet boundary condition. The principal part of the operator ${B}_{D,\varepsilon}$ is given in a factorized form \break
$A_{\varepsilon}=b(\mathbf{D})^*g^\varepsilon (\mathbf{x})b(\mathbf{D})$, 
where $b(\mathbf{D})=\sum _{l=1}^d b_l D_l$ is a matrix first order DO and $g(\mathbf{x})$ is a $\Gamma$-periodic matrix-valued function in $\mathbb{R}^d$ such that $g,g^{-1}\in L_\infty$ and $g(\mathbf{x})>0$. (The precise assumptions on $b(\mathbf{D})$ and $g$ are given below in Subsec.~\ref{Subsection operatoer A_D,eps}.) The operator ${B}_{D,\varepsilon}$ is given by the differential expression
\begin{equation}
\label{B_D,eps in introduction}
{B}_{\varepsilon}=b(\mathbf{D})^* g^\varepsilon (\mathbf{x}) b(\mathbf{D})
+\sum_{j=1}^d\bigl(a_j^\varepsilon (\mathbf{x})D_j+D_ja_j^\varepsilon(\mathbf{x})^*\bigr)
+Q^\varepsilon (\mathbf{x}) +\lambda I
\end{equation}
with the Dirichlet condition on $\partial\mathcal{O}$. 
Here $\Gamma$-periodic matrix-valued functions $a_j$, \break$j=1,\dots,d$, and $Q$ 
belong to suitable $L_p(\Omega)$-spaces and the matrix $Q(\mathbf{x})$ is assumed to be Hermitian. The constant $\lambda$ is chosen so that the operator $B_{D,\varepsilon}$ is positive definite. (The explicit assump\-ti\-ons on the coefficients are given below in Subsec.~\ref{Subsection lower order terms}.) The precise definition of the operator ${B}_{D,\varepsilon}$ is given via the corresponding quadratic form on the Sobolev class  $H^1_0(\mathcal{O};\mathbb{C}^n)$.

The coefficients of the operator $B_{D,\varepsilon}$ oscillate rapidly for small $\varepsilon$. We are interested in the behavior of the solution of the following problem for small $\varepsilon$:	 
\begin{equation}
\label{introduction hyperbolic system}
\begin{cases}
\frac{\partial ^2\mathbf{u}_\varepsilon }{ \partial t^2}(\mathbf{x},t)=-(B_\varepsilon\mathbf{u}_\varepsilon )(\mathbf{x},t)+\mathbf{F}(\mathbf{x},t),\quad
\mathbf{u}_\varepsilon (\cdot,t)\vert _{\partial\mathcal{O}}=0,\\
\mathbf{u}_\varepsilon (\mathbf{x},0)=\boldsymbol{\varphi}(\mathbf{x}),\quad \frac{\partial \mathbf{u}_\varepsilon}{\partial t}(\mathbf{x},0)=\boldsymbol{\psi}(\mathbf{x}).
\end{cases}
\end{equation} 
For $\boldsymbol{\varphi}\in H^1_0(\mathcal{O};\mathbb{C}^n)$, $\boldsymbol{\psi}\in L_2(\mathcal{O};\mathbb{C}^n)$, and $\mathbf{F}\in L_{1,\mathrm{loc}}(\mathbb{R};L_2(\mathcal{O};\mathbb{C}^n))$, 
we have
\begin{equation*}
\mathbf{u}_\varepsilon (\cdot ,t)=\cos (tB_{D,\varepsilon}^{1/2})\boldsymbol{\varphi}
+B_{D,\varepsilon}^{-1/2}\sin (tB_{D,\varepsilon}^{1/2})\boldsymbol{\psi}+\int _0^t B_{D,\varepsilon}^{-1/2}\sin \left((t-\widetilde{t})B_{D,\varepsilon}^{1/2}\right)\mathbf{F}(\cdot ,\widetilde{t})\,d\widetilde{t}.
\end{equation*}
Thus, the question about the behavior of the solution of problem \eqref{introduction hyperbolic system} 
is reduced to suitable approximations of the operators 
$\cos (tB_{D,\varepsilon}^{1/2})$ and $B_{D,\varepsilon}^{-1/2}\sin (tB_{D,\varepsilon}^{1/2})$ 
for small $\varepsilon$.

\subsection{Main results} 

\textit{Our first main results} are:
\begin{align}
\label{intr Th cos B_D,eps}
\Bigl\Vert &
\left(
\cos (t B_{D,\varepsilon}^{1/2})-\cos (t (B_D^0)^{1/2})
\right)(B_D^0)^{-2}\Bigr\Vert _{L_2(\mathcal{O})\rightarrow L_2(\mathcal{O})}
\leqslant C \varepsilon \left(1+\vert t\vert ^5\right) ,
\\
\label{intr Th sin 1}
\Bigl\Vert & \left( B_{D,\varepsilon}^{-1/2}\sin (tB_{D,\varepsilon}^{1/2})-(B_D^0)^{-1/2}\sin (t(B_D^0)^{1/2})\right) (B_D^0)^{-2}\Bigr\Vert _{L_2(\mathcal{O})\rightarrow L_2(\mathcal{O})}
\leqslant C\varepsilon \vert t\vert (1+\vert t\vert ^5) .
\end{align}
These estimates are valid for $t\in\mathbb{R}$ and sufficiently small $\varepsilon$. Here $B_D^0$ is the \textit{effective operator }with constant coefficients. \textit{The second result} is the following approximation
\begin{equation}
\label{intr Th sin 2}
\begin{split}
\Bigl\Vert &
\bigl( B_{D,\varepsilon}^{-1/2}\sin (tB_{D,\varepsilon}^{1/2})-(B_D^0)^{-1/2}\sin (t(B_D^0)^{1/2})
-\varepsilon \mathcal{K}_D(\varepsilon ;t)\bigr)(B_D^0)^{-2}\Bigr\Vert _{L_2(\mathcal{O})\rightarrow H^1(\mathcal{O})}
\\
&\leqslant C\varepsilon ^{1/2} (1+t ^6).
\end{split}
\end{equation}
Here $\mathcal{K}_D(\varepsilon ;t)$ is the \textit{corrector.} 
It contains rapidly oscillating factors and so depends on~$\varepsilon$. 
In the general case, the corrector contains a smoothing operator. For $d\leqslant 8$, if the boundary of the domain is sufficiently smooth, one can remove the smoothing operator from the corrector. 
The constants in estimates \eqref{intr Th cos B_D,eps}--\eqref{intr Th sin 2} 
can be controlled explicitly in terms of the problem data. The results of such type are called 
\textit{operator error estimates in homogenization theory.} 
For $t$ fixed, estimates \eqref{intr Th cos B_D,eps} and \eqref{intr Th sin 1} are of the sharp order  $O(\varepsilon)$. The order $O(\varepsilon ^{1/2})$ of estimate \eqref{intr Th sin 2} is worse because of the boundary influence. 
It is impossible to obtain an analogue of estimate \eqref{intr Th sin 2} for the operator  
$\cos (t B_{D,\varepsilon}^{1/2})$. But the ,,smoothed'' cosine operator can be approximated:
\begin{equation}
\label{cos corrector introduction}
\begin{split}
\Bigl\Vert &\Bigl(\cos (tB_{D,\varepsilon}^{1/2})B_{D,\varepsilon}^{-1}-\cos (t(B_D^0)^{1/2})(B_D^0)^{-1}-\varepsilon\mathscr{K}_D(\varepsilon;t)\Bigr)(B_D^0)^{-1}\Bigr\Vert _{L_2(\mathcal{O})\rightarrow H^1(\mathcal{O})}
\\
&\leqslant
C\varepsilon ^{1/2}(1+\vert t\vert ^5).
\end{split}
\end{equation}
It is in accordance with the results of \cite{BrOFMu}, see discussion in Subsec.~\ref{Subsection Survey} below.

\subsection{Survey} 

\label{Subsection Survey}

At present, operator error estimates attract a great deal of attention. Interest to this topic was caused by the paper \cite{BSu} of M.~Sh.~Birman and T.~A.~Suslina. In \cite{BSu}, the operator $A_\varepsilon=b(\mathbf{D})^*g^\varepsilon (\mathbf{x})b(\mathbf{D})$ acting in $L_2(\mathbb{R}^d;\mathbb{C}^n)$ was considered. By means of the \textit{spectral approach} it was obtained that
\begin{equation}
\label{A_eps res L2 intr}
\Vert (A_\varepsilon +I)^{-1}-(A^0+I)^{-1}\Vert _{L_2(\mathbb{R}^d)\rightarrow L_2(\mathbb{R}^d)}
\leqslant C\varepsilon .
\end{equation}
Here $A^0=b(\mathbf{D})^* g ^0 b(\mathbf{D})$ is the effective operator and $g^0$ is the constant effective matrix. In \cite{BSu06}, the operator $(A_\varepsilon +I)^{-1}$ was approximated in the $(L_2\rightarrow H^1)$-operator norm:
\begin{equation}
\label{A_eps res H1 intr}
\Vert (A_\varepsilon +I)^{-1}-(A^0+I)^{-1}-\varepsilon K(\varepsilon)\Vert _{L_2(\mathbb{R}^d)\rightarrow H^1(\mathbb{R}^d)}
\leqslant C\varepsilon.
\end{equation}
The estimates \eqref{A_eps res L2 intr} and \eqref{A_eps res H1 intr} were later generalized to the operator $B_\varepsilon$ of the form \eqref{B_D,eps in introduction} by T.~A.~Suslina \cite{SuAA}. 

\textit{Another approach} to operator error estimates in homogenization theory was suggested by V.~V.~Zhikov  \cite{Zh1}. In \cite{Zh1,ZhPas2}, estimates of the form \eqref{A_eps res L2 intr}, \eqref{A_eps res H1 intr} were obtained for the acoustics operator and for the elasticity operator. The \textit{\glqq modified method of the first order approximation\grqq\,}
or the \textit{\glqq shift method\grqq\,}, in the terminology of the authors, was based on the analysis of the first order approximation to the solution and introducing the additional parameter.  Besides the problems in $\mathbb{R}^d$, in \cite{Zh1,ZhPas2}, the homogenization problems in a bounded domain $\mathcal{O}\subset \mathbb{R}^d$ with the Dirichlet or Neumann boundary condition were studied. Further results of V.~V.~Zhikov and S.~E.~Pastukhova can be found in the survey \cite{ZhPasUMN}.

The operator error estimates for homogenization of the Dirichlet and Neumann problems for second order elliptic equation in a bounded domain were studied by many authors, see \cite{ZhPas2,Gr1,Gr2,KeLiS,PSu,Su13,Su_SIAM}.  The detailed survey can be found in the introduction to the paper \cite{MSuPOMI}. In \cite{MSuPOMI}, approximations for the resolvent of the operator \eqref{B_D,eps in introduction} were obtained:
\begin{align}
\label{intr ell main result 1}
\Vert &(B_{D,\varepsilon}-\zeta I )^{-1}-(B_D^0-\zeta I)^{-1}\Vert _{L_2(\mathcal{O})\rightarrow L_2(\mathcal{O})}
\leqslant C(\phi)\varepsilon\vert\zeta\vert ^{-1/2},
\\
\label{intr ell main result 2}
\begin{split}
\Vert &(B_{D,\varepsilon}-\zeta I )^{-1}-(B_D^0-\zeta I)^{-1}-\varepsilon K_D(\varepsilon;\zeta)\Vert _{L_2(\mathcal{O})\rightarrow H^1(\mathcal{O})}
\leqslant C(\phi)\bigl(\varepsilon ^{1/2}\vert \zeta\vert ^{-1/4}+\varepsilon\bigr), 
\end{split}
\end{align}
$\zeta\in\mathbb{C}\setminus\mathbb{R}_+$, 
$\vert \zeta\vert \geqslant 1$. The values $C(\phi)$ are controlled explicitly in terms of the problem data and the angle  $\phi=\mathrm{arg}\,\zeta$. 
(For $\zeta$ fixed, close results were obtained by Q.~Xu \cite{Xu3}.) 

To parabolic problems in the whole space $\mathbb{R}^d$, spectral method was applied by T.~A.~Suslina \cite{Su04,Su_MMNP}. It was obtained that
\begin{align}
\label{est parabol new 1}
&\Vert e^{-tA_\varepsilon }-e^{-t A^0}\Vert _{L_2(\mathbb{R}^d)\rightarrow L_2(\mathbb{R}^d)}\leqslant C\varepsilon (t+\varepsilon ^2)^{-1/2},\quad t\geqslant 0,
\\
\label{est parabol new 2}
&\Vert e^{-tA_\varepsilon }-e^{-t A^0}-\varepsilon\mathfrak{K}(\varepsilon ;t)\Vert _{L_2(\mathbb{R}^d)\rightarrow H^1(\mathbb{R}^d)}
\leqslant C\varepsilon (t^{-1/2}+t^{-1}),\quad t\geqslant\varepsilon ^2.
\end{align}
By the shift method, these estimates were proven by V.~V.~Zhikov and S.~E.~Pastukhova \cite{ZhPAs_parabol}. Later, results \eqref{est parabol new 1} and \eqref{est parabol new 2} were transferred to the operator $B_\varepsilon$ by the author \cite{M13}.

Homogenization of the first initial-boundary value problem for a parabolic equation involving the operator $b(\mathbf{D})^*g^\varepsilon (\mathbf{x})b(\mathbf{D})$ or the operator  
 \eqref{B_D,eps in introduction} was studied by Yu.~M.~Meshkova and T.~A.~Sus\-li\-na in \cite{MSu15-2} and \cite{MSuAA17},  respectively. The method was based on the identity $$e^{-tB_{D,\varepsilon}}=-\frac{1}{2\pi i}\int _\gamma e^{-\zeta t}(B_{D,\varepsilon}-\zeta I)^{-1}\,d\zeta,$$
and estimates \eqref{intr ell main result 1} and \eqref{intr ell main result 2}. Here $\gamma\subset\mathbb{C}$ is a contour enclosing the spectrum of the operator $B_{D,\varepsilon}$ in the positive direction. Recall that, according to the classical Trotter-Kato theorem (see, e.~g.,  \cite[Chapter~X, Theorem 1.1]{Sa}), the strong convergence of semigroups follows from the strong convergence of the corresponding resolvents, while in \cite{MSuAA17} app\-ro\-xi\-ma\-tions in the uniform operator topology with explicit error estimates were obtained. Let us mention the recent work \cite{ChEl}, where the Trotter-Kato theorem was transferred to  weak and uniform operator topologies and the results were applied to homogenization of the parabolic equations (without operator error estimates). 

In  \cite{BSu08,M,DSu}, the spectral approach was applied to the hyperbolic systems. 
In \cite{BSu08}, for $t\in\mathbb{R}$ it was obtained that
\begin{align}
\label{Th BSu cos introduction}
&\Vert \cos (t {A}_\varepsilon ^{1/2})-\cos (t( {A}^0)^{1/2})\Vert _{H^2(\mathbb{R}^d)\rightarrow L_2(\mathbb{R}^d)}\leqslant C\varepsilon (1+\vert t\vert) ,
\\
\label{Th BSu}
\begin{split}
\Vert   {A}_\varepsilon^{-1/2}\sin( t  {A}^{1/2}_\varepsilon)-( {A}^0)^{-1/2}\sin ( t ( {A}^0)^{1/2})\Vert_{H^2(\mathbb{R}^d)\rightarrow L_2(\mathbb{R}^d)} \leqslant C\varepsilon (1+\vert t\vert )^2.
\end{split}
\end{align}
In \cite{M}, estimate \eqref{Th BSu} was refined with respect to the type of the operator norm:
\begin{equation}
\begin{split}
\label{Th sin M principal term introduction}
\Vert {A}_\varepsilon ^{-1/2}\sin( t  {A}^{1/2}_\varepsilon)-( {A}^0)^{-1/2}\sin ( t ( {A}^0)^{1/2})\Vert_{H^1(\mathbb{R}^d)\rightarrow L_2(\mathbb{R}^d)} \leqslant C\varepsilon (1+\vert t\vert ), 
\end{split}
\end{equation}
$t\in\mathbb{R}$, 
and approximation for the operator ${A}_\varepsilon^{-1/2}\sin( t  {A}^{1/2}_\varepsilon)$ in the $(H^2\rightarrow H^1)$-operator norm was obtained
\begin{align}
\label{main result 2}
\begin{split}
\Bigl\Vert & {A}_\varepsilon ^{-1/2}\sin ( t  {A}_\varepsilon ^{1/2})-( {A}^0)^{-1/2}\sin ( t ( {A}^0)^{1/2})-\varepsilon\mathrm{K}(\varepsilon ; t)\Bigr\Vert _{H^2(\mathbb{R}^d)\rightarrow H^1(\mathbb{R}^d)}
\\
&\leqslant C \varepsilon(1+\vert t\vert),\quad t\in\mathbb{R}.
\end{split}
\end{align}
Here $\mathrm{K}(\varepsilon ; t)$ is the corrector. 
In \cite{DSu}, the sharpness of estimates \eqref{Th BSu cos introduction} and \eqref{Th sin M principal term introduction} with respect to the type of the norm was proven in the general case. 

The correctors in estimates \eqref{A_eps res H1 intr}, \eqref{est parabol new 2}, and \eqref{main result 2} have a similar structure. 
So, it seems natural to expect that the cosine operator also can be approximated in the energy norm with a similar corrector. However, in \cite{BrOFMu} it was observed that even the fact of the norm convergence is true only for the very special choice of the initial data. The argument used in \cite{BrOFMu} is the following: the convergence of the energy does not occur in the general situation. But the solution can be splitted into two parts: the first one is designed so that the corresponding energy converges to the energy for the effective equation and the second part tends to zero $*$-weakly in $L_\infty ((0,T);H_0^1(\mathcal{O}))\cap W^{1,\infty}((0,T);L_2(\mathcal{O}))$. In our considerations, we deal only with the first part. 
This case corresponds to estimate \eqref{cos corrector introduction}. 
In the general case, some approximations with the corrector were obtained in \cite{BraLe,CaDCoCaMaMarG1}. 
Their results can not be written in the uniform operator topology. The corresponding corrector is non-local because of the dispersion of waves in the~inhomogeneous media. Dispersion effects for homogenization of the wave equation were discussed in
 \cite{ABriV,ConOrV,ConSaMaBalV}, but the operator error estimates were not studied.

Let us also mention a recent preprint \cite{CooSav}, where (independently of the present work)  the homogenization of the attractors of the quasi-linear damped wave equation was derived from the estimate of the form \eqref{intr ell main result 1} for the operator $B_{D,\varepsilon}=-\mathrm{div}\, g^\varepsilon (\mathbf{x})\nabla$ (and $\zeta =0$). But the results of \cite{CooSav} can not be written in the uniform operator topology. 
Thus, operator error estimates for homogenization of hyperbolic systems in a bounded domain were not previously known.

\subsection{Method}

The present work develops the method of \cite{MSu15-2,MSuAA17}. We deduce operator error estimates for non-stationary problem from the elliptic results via the inverse Laplace transform.  (Surely, the Laplace transform had been applied for homogenization of hyperbolic problems previously, see \cite[Chapter 2, Subsec. 3.9]{BeLP}, \cite[Chapter V, Sec.~6]{Sa}, and \cite{Pas,ZhPas1}. We also note that  the non-stationary Maxwell system was studied by using the Laplace transform in \cite[Chapter IV]{ZhKO}. But the operator error estimates were not discussed in the books and papers listed above.)

The method is based on the identity
\begin{equation}
\label{int cos tozd}
\cos \bigl(tB_{D,\varepsilon}^{1/2}\bigr)B_{D,\varepsilon}^{-2}
=-\frac{t^2}{2}B_{D,\varepsilon}^{-1}+B_{D,\varepsilon}^{-2}
+\frac{1}{2\pi i}\int _{\mathrm{Re}\,\lambda =\sqrt{c} }\lambda ^{-3}(B_{D,\varepsilon}+\lambda ^2)^{-1}e^{\lambda t}\,d\lambda,
\end{equation}
$c>0$, 
and on using the approximations of the operator
 $(B_{D,\varepsilon}-\zeta I)^{-1}$, $\zeta\in\mathbb{C}\setminus\mathbb{R}_+$, with the error estimates that are two-parametric with respect to $\varepsilon$ and $\zeta$. The required approximations were obtained in   \cite{MSuPOMI}. Combining \eqref{int cos tozd}, the similar identity for the effective operator, and estimate \eqref{intr ell main result 1}, we obtain inequality \eqref{intr Th cos B_D,eps}. 
To derive estimate \eqref{intr Th sin 1} from \eqref{intr Th cos B_D,eps}, we use the representations
\begin{align}
\label{int sin tozd}
&B_{D,\varepsilon}^{-1/2}\sin (tB_{D,\varepsilon }^{1/2})
=\int _0^t \cos (\tau B_{D,\varepsilon}^{1/2})\,d\tau,
\\
\label{int eff sin tozd}
&(B_D^0)^{-1/2}\sin \left(t(B_D^0)^{1/2}\right)
=\int _0^t \cos \left(\tau (B_D^0)^{1/2}\right)\,d\tau .
\end{align} 

The approximation with the corrector for the operator $B_{D,\varepsilon}^{-1}\cos(tB_{D,\varepsilon}^{1/2})$ follows 
from \eqref{int cos tozd}, the similar identity for the effective operator, and estimate \eqref{intr ell main result 2}. Using this appro\-xi\-ma\-tion, identities \eqref{int sin tozd}, \eqref{int eff sin tozd}, and homogenization results for the resolvent, we obtain inequality \eqref{intr Th sin 2}.

The presence of the operator $(B_D^0)^{-2}$ in estimates \eqref{intr Th cos B_D,eps}--\eqref{intr Th sin 2} is caused by the method of investigation. Because of this factor, the initial data and the right-hand side in \eqref{introduction hyperbolic system} are subject to the following restrictions:
\begin{equation}
\label{intr inital data and F}
\boldsymbol{\varphi},\boldsymbol{\psi}\in\mathrm{Dom}\,(B_D^0)^2, \quad \mathbf{F}\in L_{1,\mathrm{loc}}(\mathbb{R};\mathrm{Dom}\,(B_D^0)^{2}).
\end{equation}
If $\partial\mathcal{O}\in C^{3,1}$, then 
$\mathrm{Dom}\,(B_D^0)^{2}$ can be considered as the subspace of $H^4(\mathcal{O};\mathbb{C}^n)$. 
Thus, the initial data and the right-hand side of the equation \eqref{introduction hyperbolic system} are required to be more smooth compared with the data for the problems in the whole space. Apparently, the results of the present paper are not sharp with respect to the classes of smoothness for the initial data and the right-hand side of the equation. 
However, it seems that the applied technique does not allow to improve the results.

\subsection{Plan of the paper} The paper consists of three sections and introduction. In Section~\ref{Section 1}, the class of the operators $B_{D,\varepsilon}$ is described, the effective operator $B_D^0$ is defined and the approximations for the resolvent $(B_{D,\varepsilon}-\zeta I)^{-1}$ are formulated. Section~\ref{Section 2} contains the main results of the paper. Their proofs can be found in Section~\ref{Section 3}. 

\subsection{Notation} Let $\mathfrak{H}$ and $\mathfrak{H}_*$ be complex separable Hilbert spaces. The symbols $(\cdot ,\cdot)_\mathfrak{H}$ and $\Vert \cdot\Vert _\mathfrak{H}$ denote the inner product and the norm in   $\mathfrak{H}$, respectively; the symbol $\Vert \cdot\Vert _{\mathfrak{H}\rightarrow\mathfrak{H}_*}$ means the norm of the linear continuous operator from $\mathfrak{H}$ to $\mathfrak{H}_*$.

We use the notation $\mathbb{Z}_+$ for the set of non-negative integers and $\mathbb{R}_+$ for the positive half-line $[0,\infty)$.

The symbols $\langle \cdot ,\cdot\rangle$ and $\vert \cdot\vert$ stand for the inner product and the norm in   $\mathbb{C}^n$, respectively; $\mathbf{1}_n$ is the identity $(n\times n)$-matrix. If $a$ is $(m\times n)$-matrix, then the symbol $\vert a\vert$ denotes the norm of the matrix $a$ viewed as the operator from  $\mathbb{C}^n$ to $\mathbb{C}^m$. 
For $z\in\mathbb{C}$, by $z^*$ we denote the complex conjugate number. (We use such non-standard notation, because the upper line denotes the mean value of a periodic function over the cell of periodicity.) 
We use the notation $\mathbf{x}=(x_1,\dots , x_d)\in\mathbb{R}^d$, $iD_j=\partial _j =\partial /\partial x_j$, $j=1,\dots,d$, $\mathbf{D}=-i\nabla=(D_1,\dots ,D_d)$. The classes $L_p$ of $\mathbb{C}^n$-valued functions in a domain $\mathcal{O}\subset\mathbb{R}^d$ are denoted by $L_p(\mathcal{O};\mathbb{C}^n)$, $1\leqslant p\leqslant \infty$. The Sobolev spaces of $\mathbb{C}^n$-valued functions in a domain $\mathcal{O}\subset\mathbb{R}^d$ are denoted by $H^s(\mathcal{O};\mathbb{C}^n)$. 
By $H^1_0(\mathcal{O};\mathbb{C}^n)$ we denote the closure of the class  $C_0^\infty (\mathcal{O};\mathbb{C}^n)$ in the space $H^1(\mathcal{O};\mathbb{C}^n)$. For $n=1$, we simply write  $L_p(\mathcal{O})$, $H^s(\mathcal{O})$ and so on, but, sometimes, if this does not lead to confusion, we use such simple notation for the spaces of vector-valued or matrix-valued functions. The symbol $L_p((0,T);\mathfrak{H})$, $1\leqslant p\leqslant\infty$, means the $L_p$-space of $\mathfrak{H}$-valued functions on the interval $(0,T)$.

Various constants in estimates are denoted by $c$, $\mathfrak{c}$, $C$, $\mathcal{C}$, $\mathscr{C}$, $\mathfrak{C}$  
(possibly, with indices and marks).

\subsection*{Acknowledgement}

The author is deeply grateful to T.~A.~Suslina for her attention to this work.

\section{Homogenization results for the elliptic Dirichlet problem}
\label{Section 1}

\subsection{Lattices in $\mathbb{R}^d$} 
\label{Subsection lattices}
Let $\Gamma \subset \mathbb{R}^d$ be the lattice generated by the basis $\mathbf{a}_1,\dots ,\mathbf{a}_d \in \mathbb{R}^d$:
$$\Gamma =\Bigl\lbrace
\mathbf{a}\in \mathbb{R}^d : \mathbf{a}=\sum _{j=1}^d \nu _j \mathbf{a}_j, \nu _j\in \mathbb{Z}
\Bigr\rbrace ,$$
and let $\Omega$ be the elementary cell of the lattice~$\Gamma$:
$$\Omega =
\Bigl \lbrace
\mathbf{x}\in \mathbb{R}^d :\mathbf{x}=\sum _{j=1}^d \tau _j \mathbf{a}_j , -\frac{1}{2}<\tau _j<\frac{1}{2}
\Bigr\rbrace .$$
By $\vert \Omega \vert $ we denote the Lebesgue measure of the cell $\Omega$: $\vert \Omega \vert =\mathrm{meas}\,\Omega$. 
Set  
$2r_1:=\mathrm{diam}\,\Omega$.

The basis $\mathbf{b}_1,\dots,\mathbf{b}_d\in\mathbb{R}^d$, dual to $\mathbf{a}_1,\dots,\mathbf{a}_d$, is defined by the relations $\langle\mathbf{b}_j,\mathbf{a}_i\rangle=2\pi\delta _{ji}$. This basis generates the  
lattice $\widetilde{\Gamma}$ dual to $\Gamma$. Denote $2r_0:=\min _{0\neq \mathbf{b}\in\widetilde{\Gamma}}\vert \mathbf{b}\vert $.

Let $\widetilde{H}^1(\Omega)$ be the subspace of functions from  $H^1(\Omega)$ whose $\Gamma$-periodic extension to $\mathbb{R}^d$ belongs to $H^1_{\mathrm{loc}}(\mathbb{R}^d)$. If $\Phi (\mathbf{x})$~is a~$\Gamma$-periodic matrix-valued function in $\mathbb{R}^d$, we put
$\Phi ^\varepsilon (\mathbf{x}):=\Phi (\mathbf{x}/\varepsilon)$, $\varepsilon >0$; $
\overline{\Phi}:=\vert \Omega\vert ^{-1}\int _\Omega \Phi(\mathbf{x})\,d\mathbf{x}$, 
 $\underline{\Phi}:=\left(\vert \Omega\vert ^{-1}\int _\Omega \Phi(\mathbf{x})^{-1}\,d\mathbf{x}\right)^{-1}$.  
Here in the definition of 
$\overline{\Phi}$ it is assumed that $\Phi\in L_{1,\mathrm{loc}}(\mathbb{R}^d)$; in the definition of $\underline{\Phi}$ it is assumed that the matrix-valued function $\Phi$ is square
and non-degenerate, and $\Phi^{-1}\in L_{1,\mathrm{loc}}(\mathbb{R}^d)$. By $[\Phi^\varepsilon ]$ we denote the operator of multiplication by the matrix-valued function $\Phi^\varepsilon (\mathbf{x})$.

\subsection{The Steklov smoothing} The Steklov smoothing operator $S_\varepsilon^{(k)}$ acts in the space $L_2(\mathbb{R}^d;\mathbb{C}^k)$ (where $k\in\mathbb{N}$) and is defined by 
\begin{equation}
\label{S_eps}
\begin{split}
(S_\varepsilon^{(k)} \mathbf{u})(\mathbf{x})=\vert \Omega \vert ^{-1}\int _\Omega \mathbf{u}(\mathbf{x}-\varepsilon \mathbf{z})\,d\mathbf{z},\quad \mathbf{u}\in L_2(\mathbb{R}^d;\mathbb{C}^k)
.
\end{split}
\end{equation}
We will omit the index $k$ in the notation and write simply $S_\varepsilon$. 
Obviously, $S_\varepsilon \mathbf{D}^\alpha \mathbf{u}=\mathbf{D}^\alpha S_\varepsilon \mathbf{u}$ for $\mathbf{u}\in H^\sigma(\mathbb{R}^d;\mathbb{C}^k)$ 
and any multiindex $\alpha$ such that $\vert \alpha\vert \leqslant \sigma$. 
We need the following properties of the operator $S_\varepsilon$
(see \cite[Lemmas 1.1 and 1.2]{ZhPas2} or \cite[Propositions 3.1 and 3.2]{PSu}).

\begin{proposition}
\label{Proposition S__eps - I}
For any function $\mathbf{u}\in H^1(\mathbb{R}^d;\mathbb{C}^k)$ 
we have
\begin{equation*}
\Vert S_\varepsilon \mathbf{u}-\mathbf{u}\Vert _{L_2(\mathbb{R}^d)}\leqslant \varepsilon r_1\Vert \mathbf{D}\mathbf{u}\Vert _{L_2(\mathbb{R}^d)},
\end{equation*}
where $2r_1=\mathrm{diam}\,\Omega$.
\end{proposition}

\begin{proposition}
\label{Proposition f^eps S_eps}
Let $\Phi$ be a $\Gamma$-periodic function in $\mathbb{R}^d$ such that $\Phi\in L_2(\Omega)$. 
Then the operator $[\Phi ^\varepsilon ]S_\varepsilon $ is continuous in $L_2(\mathbb{R}^d)$ and
\begin{equation*}
\Vert [\Phi^\varepsilon]S_\varepsilon \Vert _{L_2(\mathbb{R}^d)\rightarrow L_2(\mathbb{R}^d)}\leqslant \vert \Omega \vert ^{-1/2}\Vert \Phi\Vert _{L_2(\Omega)}.
\end{equation*}
\end{proposition}

\subsection{The operator $A_{D,\varepsilon}$} 
\label{Subsection operatoer A_D,eps}
Let $\mathcal{O}\subset \mathbb{R}^d$ be a bounded domain of class $C^{1,1}$. 
In $L_2(\mathcal{O};\mathbb{C}^n)$, we consider the operator $A_{D,\varepsilon}$ formally given by the differential expression $A_\varepsilon = b(\mathbf{D})^*g^\varepsilon (\mathbf{x})b(\mathbf{D})$ with the Dirichlet condition on $\partial\mathcal{O}$. 
(We agree to mark  a differential operator with the Dirichlet condition and its quadratic form (but not a formal differential expression corresponding to the operator) by the lower index \glqq$D$\grqq.)   
Here $g(\mathbf{x})$ is a $\Gamma$-periodic Hermitian $(m\times m)$-matrix-valued function (in general, with complex entries). Assume that $g(\mathbf{x})>0$ and $g,g^{-1}\in L_\infty (\mathbb{R}^d)$. The differential operator $b(\mathbf{D})$ has the form $b(\mathbf{D})=\sum _{j=1}^d b_jD_j$, where $b_j$, $j=1,\dots ,d$, are constant matrices of the size $m\times n$ (in general, with complex entries). Assume that $m\geqslant n$ and that the symbol $b(\boldsymbol{\xi})=\sum _{j=1}^d b_j\xi_j$ of the operator $b(\mathbf{D})$ has maximal rank: 
$\mathrm{rank}\,b(\boldsymbol{\xi})=n$ for $0\neq \boldsymbol{\xi}\in\mathbb{R}^d$. 
This is equivalent to the existence of constants $\alpha _0$ and $\alpha _1$ such that
\begin{equation}
\label{<b^*b<}
\alpha _0\mathbf{1}_n \leqslant b(\boldsymbol{\theta})^*b(\boldsymbol{\theta}) 
\leqslant \alpha _1\mathbf{1}_n,\quad 
\boldsymbol{\theta}\in \mathbb{S}^{d-1};\quad 
0<\alpha _0\leqslant \alpha _1<\infty.
\end{equation}
By \eqref{<b^*b<},
\begin{equation}
\label{b_l <=}
\vert b_j\vert \leqslant \alpha _1^{1/2},\quad j=1,\dots ,d.
\end{equation}

The precise definition of the operator $A_{D,\varepsilon}$ is given via the quadratic form
\begin{equation}
\label{a_D,eps}
\mathfrak{a}_{D,\varepsilon} [\mathbf{u},\mathbf{u}]=\int _{\mathcal{O}}\langle g^\varepsilon (\mathbf{x})b(\mathbf{D})\mathbf{u},b(\mathbf{D})\mathbf{u}\rangle \,d\mathbf{x},\quad \mathbf{u}\in H^1_0(\mathcal{O};\mathbb{C}^n).
\end{equation}
Extending the function $\mathbf{u}\in H^1_0(\mathcal{O};\mathbb{C}^n)$ by zero onto $\mathbb{R}^d\setminus\mathcal{O}$ and taking \eqref{<b^*b<} into account, we obtain
\begin{equation}
\label{a_D,eps estimates}
\alpha _0\Vert g^{-1}\Vert ^{-1}_{L_\infty}\Vert\mathbf{D}\mathbf{u}\Vert ^2 _{L_2(\mathcal{O})}
\leqslant 
\mathfrak{a}_{D,\varepsilon}[\mathbf{u},\mathbf{u}]
\leqslant \alpha _1\Vert g\Vert _{L_\infty}\Vert \mathbf{D}\mathbf{u}\Vert ^2 _{L_2(\mathcal{O})},
\quad
\mathbf{u}\in H^1_0(\mathcal{O};\mathbb{C}^n).
\end{equation}

\subsection{Lower order terms. The operator $B_{D,\varepsilon}$} 
\label{Subsection lower order terms}
We study the self-adjoint operator $B_{D,\varepsilon}$ with the principal part $A_{\varepsilon}$. To define the lower order terms, let us introduce $\Gamma$-periodic $(n\times n)$-matrix-valued functions (in general, with complex entries) $a_j$, $j=1,\dots ,d$, such that
\begin{equation*}
a_j \in L_\rho (\Omega ), \quad\rho =2 \;\mbox{for}\;d=1,\quad\rho >d\;\mbox{for}\;d\geqslant 2,\quad j=1,\dots ,d.
\end{equation*}
Next, let $Q$  be the $\Gamma$-periodic Hermitian $(n\times n)$-matrix-valued function (with complex entries) such that
\begin{equation}
\label{Q condition}
Q\in L_s(\Omega ),\quad s=1 \;\mbox{for}\;d=1,\quad s >d/2\;\mbox{for}\;d\geqslant 2.
\end{equation}
By the Sobolev embedding theorem, conditions imposed on 
$\rho$ and $s$ guarantee that 
the lower terms of the operator $B_{D,\varepsilon}$ are strongly subordinate to its principal part $A_\varepsilon$.

For convenience of further references, the following set of variables is called the ,,problem data'':
\begin{equation}
\label{problem data}
\begin{split}
&d,\,m,\,n,\,\rho ,\,s ;\,\alpha _0,\, \alpha _1 ,\,\Vert g\Vert _{L_\infty},\, \Vert g^{-1}\Vert _{L_\infty},\, \Vert a_j\Vert _{L_\rho (\Omega)},\, j=1,\dots ,d;\,
\Vert Q\Vert _{L_s(\Omega)};\\
&\text{the parameters of the lattice }\Gamma ;\;\text{the domain }\mathcal{O}.
\end{split}
\end{equation}

In $L_2(\mathcal{O};\mathbb{C}^n)$, we consider the operator $B_{D,\varepsilon}$, $0<\varepsilon\leqslant 1$, formally given by the differential expression
\begin{equation}
\label{B_D,eps}
B_{\varepsilon}=b(\mathbf{D})^* g^\varepsilon (\mathbf{x})b(\mathbf{D})+\sum _{j=1}^d \left(
a_j^\varepsilon (\mathbf{x})D_j +D_j a_j^\varepsilon (\mathbf{x})^*
\right)
+Q^\varepsilon (\mathbf{x})+\lambda I
\end{equation}
with the Dirichet boundary condition. Here the constant $\lambda$ is chosen so that (see \eqref{lambda =} below) the operator $B_{D,\varepsilon}$ is positive definite. The precise definition of the operator $B_{D,\varepsilon}$ is given via the quadratic form
\begin{equation}
\label{b_D,eps}
\begin{split}
\mathfrak{b}_{D,\varepsilon }[\mathbf{u},\mathbf{u}]&=(g^\varepsilon b(\mathbf{D})\mathbf{u},b(\mathbf{D})\mathbf{u})_{L_2(\mathcal{O})}+2\mathrm{Re}\,\sum _{j=1}^d (a_j^\varepsilon D_j \mathbf{u},\mathbf{u})_{L_2(\mathcal{O})}\\
&+(Q^\varepsilon \mathbf{u},\mathbf{u})_{L_2(\mathcal{O})}+\lambda ( \mathbf{u},\mathbf{u})_{L_2(\mathcal{O})},\quad \mathbf{u}\in H^1_0(\mathcal{O};\mathbb{C}^n).
\end{split}
\end{equation}
Let us check that the form $\mathfrak{b}_{D,\varepsilon}$ is closed. By the H\"older inequality and the Sobolev embedding theorem, it can be shown (see  \cite[(5.11)--(5.14)]{SuAA}) that for any $\nu>0$ there exist constants $C_j(\nu)>0$ such that
\begin{equation*}
\begin{split}
\Vert a_j^*\mathbf{u}\Vert ^2 _{L_2(\mathbb{R}^d)}\leqslant \nu \Vert \mathbf{D}\mathbf{u}\Vert ^2_{L_2(\mathbb{R}^d)}+C_j(\nu )\Vert \mathbf{u}\Vert ^2 _{L_2(\mathbb{R}^d)},\quad \mathbf{u}\in H^1(\mathbb{R}^d;\mathbb{C}^n),\quad j=1,\dots ,d.
\end{split}
\end{equation*} 
By the change of variables $\mathbf{y}:=\varepsilon ^{-1}\mathbf{x}$ and $\mathbf{u}(\mathbf{x})=:\mathbf{v}(\mathbf{y})$, we deduce
\begin{equation*}
\begin{split}
\Vert   (a_j^\varepsilon )^*\mathbf{u}\Vert ^2_{L_2(\mathbb{R}^d)}
&=\int _{\mathbb{R}^d}\vert a_j(\varepsilon ^{-1}\mathbf{x})^*\mathbf{u}(\mathbf{x})\vert ^2\,d\mathbf{x}
=\varepsilon ^d\int _{\mathbb{R}^d}\vert a_j(\mathbf{y})^*\mathbf{v}(\mathbf{y})\vert ^2\,d\mathbf{y}\\
&\leqslant \varepsilon ^d\nu \int _{\mathbb{R}^d}\vert \mathbf{D}_{\mathbf{y}}\mathbf{v}(\mathbf{y})\vert ^2\,d\mathbf{y}
+\varepsilon ^d C_j(\nu)\int _{\mathbb{R}^d}\vert \mathbf{v}(\mathbf{y})\vert ^2\,d\mathbf{y}\\
&\leqslant \nu \Vert \mathbf{D}\mathbf{u}\Vert ^2_{L_2(\mathbb{R}^d)}+C_j(\nu)\Vert \mathbf{u}\Vert ^2_{L_2(\mathbb{R}^d)},\quad \mathbf{u}\in H^1(\mathbb{R}^d;\mathbb{C}^n),\quad 0<\varepsilon\leqslant 1.
\end{split}
\end{equation*}
Then, by \eqref{<b^*b<}, for any $\nu >0$ there exists a constant $C(\nu)>0$ such that
\begin{equation}
\label{sum a-j u}
\begin{split}
\sum _{j=1}^d \Vert (a_j^\varepsilon)^*\mathbf{u}\Vert ^2 _{L_2(\mathbb{R}^d)}
\leqslant \nu
\Vert (g^\varepsilon)^{1/2}b(\mathbf{D})\mathbf{u}\Vert ^2_{L_2(\mathbb{R}^d)}
+C(\nu)\Vert \mathbf{u}\Vert ^2_{L_2(\mathbb{R}^d)},\\\mathbf{u}\in H^1(\mathbb{R}^d;\mathbb{C}^n),\quad 0<\varepsilon\leqslant 1.
\end{split}
\end{equation}
For $\nu$ fixed, $C(\nu)$ depends only on $d$, $\rho$, $\alpha _0$, on  the norms $\Vert g^{-1}\Vert _{L_\infty}$, $\Vert a_j\Vert _{L_\rho (\Omega)}$, $j=1,\dots ,d$, and on the parameters of the lattice $\Gamma$.

By \eqref{<b^*b<}, for $\mathbf{u}\in H^1(\mathbb{R}^d;\mathbb{C}^n)$ we have
\begin{equation}
\label{Du <= c_1^2a}
\Vert \mathbf{D}\mathbf{u}\Vert ^2_{L_2(\mathbb{R}^d)}
\leqslant c_1^2\Vert (g^\varepsilon )^{1/2}b(\mathbf{D})\mathbf{u}\Vert ^2_{L_2(\mathbb{R}^d)},
\end{equation}
where $c_1:=\alpha _0 ^{-1/2}\Vert g^{-1}\Vert^{1/2}_{L_\infty}$.
Combining this with \eqref{sum a-j u}, we obtain
\begin{equation}
\label{2Re sum j <=}
\begin{split}
2\left\vert \mathrm{Re}\,\sum _{j=1}^d (D_j\mathbf{u},(a_j^\varepsilon)^*\mathbf{u})_{L_2(\mathbb{R}^d)}\right\vert 
\leqslant\frac{1}{4}\Vert (g^\varepsilon)^{1/2}b(\mathbf{D})\mathbf{u}\Vert ^2_{L_2(\mathbb{R}^d)} +c_2\Vert \mathbf{u}\Vert ^2_{L_2(\mathbb{R}^d)},\\
\mathbf{u}\in H^1(\mathbb{R}^d;\mathbb{C}^n),\quad 0<\varepsilon\leqslant 1,
\end{split}
\end{equation}
where $c_2:=8c_1^2C(\nu _0)$ for $\nu _0:=2^{-6}\alpha _0\Vert g^{-1}\Vert ^{-1}_{L_\infty}$.

Next, by condition \eqref{Q condition} for $Q$, for any $\nu >0$ there exists a constant $C_Q(\nu)>0$ such that
\begin{equation}
\label{(Qu,u)<=}
\begin{split}
\vert (Q^\varepsilon \mathbf{u},\mathbf{u})_{L_2(\mathbb{R}^d)}\vert \leqslant\nu \Vert \mathbf{D}\mathbf{u}\Vert ^2_{L_2(\mathbb{R}^d)}+C_Q(\nu)\Vert \mathbf{u}\Vert ^2_{L_2(\mathbb{R}^d)},\\
\mathbf{u}\in H^1(\mathbb{R}^d;\mathbb{C}^n),\quad 0<\varepsilon\leqslant 1 .
\end{split}
\end{equation}
For $\nu$ fixed, the constant $C_Q(\nu)$ is controlled in terms of $d$, $s$, $\Vert Q\Vert _{L_s(\Omega)}$, and the parameters of the lattice $\Gamma$.

We fix a constant $\lambda$ in \eqref{B_D,eps} as follows:
\begin{equation}
\label{lambda =}
\lambda := C_Q(\nu _*)+c_2\quad\text{for}\;\nu _* :=2^{-1}\alpha _0\Vert g^{-1}\Vert ^{-1}_{L_\infty}.
\end{equation}

Now, we return to the form \eqref{b_D,eps}. The function $\mathbf{u}\in H^1_0(\mathcal{O};\mathbb{C}^n)$ is extended by zero to $\mathbb{R}^d\setminus\mathcal{O}$. 
From \eqref{a_D,eps}, \eqref{Du <= c_1^2a}, \eqref{2Re sum j <=}, \eqref{(Qu,u)<=} with $\nu=\nu_*$, and \eqref{lambda =} we derive the lower estimate for the form \eqref{b_D,eps}:
\begin{equation}
\label{b_eps >=}
\mathfrak{b}_{D,\varepsilon}[\mathbf{u},\mathbf{u}]\geqslant \frac{1}{4}\mathfrak{a}_{D,\varepsilon} [\mathbf{u},\mathbf{u}]\geqslant c_*\Vert \mathbf{D}\mathbf{u}\Vert ^2_{L_2(\mathcal{O})},\quad\mathbf{u}\in H^1_0(\mathcal{O};\mathbb{C}^n);\quad c_*:=\frac{1}{4}\alpha _0\Vert g^{-1}\Vert ^{-1}_{L_\infty}.
\end{equation}
Next, by \eqref{a_D,eps estimates}, \eqref{2Re sum j <=}, and \eqref{(Qu,u)<=} with $\nu=1$,
\begin{equation}
\label{1.15a new}
\begin{split}
&\mathfrak{b}_{D,\varepsilon} [\mathbf{u},\mathbf{u}]\leqslant C_*\Vert \mathbf{u}\Vert ^2_{H^1(\mathcal{O})},\quad 
\mathbf{u}\in H^1_0(\mathcal{O};\mathbb{C}^n),
\end{split}
\end{equation}
where $C_*:=\max\lbrace\frac{5}{4}\alpha _1\Vert g\Vert _{L_\infty}+1;C_Q(1)+\lambda +c_2\rbrace$. Thus, the form $\mathfrak{b}_{D,\varepsilon}$ is closed. The corresponding self-adjoint operator in $L_2(\mathcal{O};\mathbb{C}^n)$ is denoted by $B_{D,\varepsilon}$.

By the Friedrichs inequality, from \eqref{b_eps >=} we deduce that
\begin{equation}
\label{b D,eps >= H1-norm}
\mathfrak{b}_{D,\varepsilon}[\mathbf{u},\mathbf{u}]\geqslant c_*(\mathrm{diam}\,\mathcal{O})^{-2}\Vert \mathbf{u}\Vert ^2_{L_2(\mathcal{O})},\quad \mathbf{u}\in H^1_0(\mathcal{O};\mathbb{C}^n).
\end{equation}
So, the operator $B_{D,\varepsilon}$ is positive definite. By \eqref{b_eps >=} and \eqref{b D,eps >= H1-norm},
\begin{align}
\label{H^1-norm <= BDeps^1/2}
\Vert \mathbf{u}\Vert _{H^1(\mathcal{O})}\leqslant c_3\Vert B_{D,\varepsilon}^{1/2}\mathbf{u}\Vert _{L_2(\mathcal{O})},\quad \mathbf{u}\in H^1_0(\mathcal{O};\mathbb{C}^n);\quad c_3:=c_*^{-1/2}\left(1+(\mathrm{diam}\,\mathcal{O})^{2}\right)^{1/2}.
\end{align}

We will need the following inequalities deduced from \eqref{b D,eps >= H1-norm} and \eqref{H^1-norm <= BDeps^1/2}:
\begin{align}
\label{B_D,eps L2 ->L2}
&\Vert B_{D,\varepsilon}^{-1}\Vert _{L_2(\mathcal{O})\rightarrow L_2(\mathcal{O})}
\leqslant c_*^{-1}(\mathrm{diam}\,\mathcal{O})^{2}=:\mathcal{C}_1,
\\
\label{B_D,eps L2 ->H^1}
&\Vert B_{D,\varepsilon}^{-1}\Vert _{L_2(\mathcal{O})\rightarrow H^1(\mathcal{O})}
\leqslant c_3 \Vert B_{D,\varepsilon}^{-1/2}\Vert _{L_2(\mathcal{O})\rightarrow L_2(\mathcal{O})}
\leqslant c_3 c_*^{-1/2}\mathrm{diam}\,\mathcal{O}=:\mathcal{C}_2.
\end{align}

\subsection{The effective matrix and its properties}  The effective operator for $A_{D,\varepsilon} $ is given by the differential expression $A^0=b(\mathbf{D})^*g^0b(\mathbf{D})$ with the Dirichlet condition on $\partial\mathcal{O}$. Here $g^0$ is the constant \textit{effective} matrix of the size $m\times m$. The matrix $g^0$ is defined in terms of the auxiliary problem on the cell. Let $\Gamma$-periodic $(n\times m)$-matrix-valued function $\Lambda (\mathbf{x})$ be the weak solution of the problem
\begin{equation}
\label{Lambda problem}
b(\mathbf{D})^*g(\mathbf{x})(b(\mathbf{D})\Lambda (\mathbf{x})+\mathbf{1}_m)=0,\quad \int _{\Omega }\Lambda (\mathbf{x})\,d\mathbf{x}=0.
\end{equation}
Denote
\begin{equation}
\label{tilde g}
\widetilde{g}(\mathbf{x}):=g(\mathbf{x})(b(\mathbf{D})\Lambda (\mathbf{x})+\mathbf{1}_m).
\end{equation}
Then the effective matrix is given by the expression
\begin{equation}
\label{g^0}
g^0:=\vert \Omega \vert ^{-1}\int _{\Omega} \widetilde{g}(\mathbf{x})\,d\mathbf{x}.
\end{equation}
It can be checked that the matrix $g^0$ is positive definite.

From \eqref{Lambda problem} it follows that
\begin{equation}
\label{b(D)Lambda <=}
\Vert b(\mathbf{D})\Lambda \Vert _{L_2(\Omega)}\leqslant \vert \Omega\vert ^{1/2}m^{1/2}\Vert g\Vert ^{1/2}_{L_\infty}\Vert g^{-1}\Vert _{L_\infty}^{1/2}.
\end{equation}

We also need the following estimates for the solution of problem \eqref{Lambda problem} proven in \cite[(6.28) and Subsec.~7.3]{BSu05}:
\begin{align}
\label{Lambda <=}
&\Vert \Lambda \Vert _{L_2(\Omega)}\leqslant \vert \Omega \vert ^{1/2}M_1,\quad M_1:=m^{1/2}(2r_0)^{-1}\alpha _0^{-1/2}\Vert g\Vert ^{1/2}_{L_\infty}\Vert g^{-1}\Vert ^{1/2}_{L_\infty},\\
\label{DLambda<=}
&\Vert \mathbf{D}\Lambda \Vert _{L_2(\Omega)}\leqslant \vert \Omega \vert ^{1/2}M_2,\quad M_2:=m^{1/2}\alpha _0^{-1/2}\Vert g\Vert ^{1/2}_{L_\infty}\Vert g^{-1}\Vert ^{1/2}_{L_\infty}.
\end{align}

The effective matrix satisfies the estimates known as the Voigt–Reuss bracketing
(see, e.~g., \cite[Chapter 3, Theorem 1.5]{BSu}).

\begin{proposition}
Let $g^0$ be the effective matrix \eqref{g^0}. Then
\begin{equation}
\label{Foigt-Reiss}
\underline{g}\leqslant g^0\leqslant \overline{g}.
\end{equation}
If $m=n$, then $g^0=\underline{g}$.
\end{proposition}

Inequalities \eqref{Foigt-Reiss} imply that 
\begin{equation}
\label{g0<=}
\vert g^0\vert \leqslant \Vert g\Vert _{L_\infty},\quad \vert (g^0)^{-1}\vert \leqslant \Vert g^{-1}\Vert _{L_\infty}.
\end{equation}

Now we distinguish the cases where one of the inequalities in \eqref{Foigt-Reiss}  becomes an identity, see  \cite[Chapter 3, Propositions 1.6 and 1.7]{BSu}. 

\begin{proposition}
The identity $g^0=\overline{g}$ is equivalent to the relations
\begin{equation}
\label{overline-g}
b(\mathbf{D})^* {\mathbf g}_k(\mathbf{x}) =0,\ \ k=1,\dots,m,
\end{equation}
where ${\mathbf g}_k(\mathbf{x})$, $k=1,\dots,m,$ are the columns of the matrix $g(\mathbf{x})$.
\end{proposition}

\begin{proposition} The identity $g^0 =\underline{g}$ is equivalent to the relations
\begin{equation}
\label{underline-g}
{\mathbf l}_k(\mathbf{x}) = {\mathbf l}_k^0 + b(\mathbf{D}) {\mathbf w}_k,\ \ {\mathbf l}_k^0\in \mathbb{C}^m,\ \
{\mathbf w}_k \in \widetilde{H}^1(\Omega;\mathbb{C}^m),\ \ k=1,\dots,m,
\end{equation}
where ${\mathbf l}_k(\mathbf{x})$, $k=1,\dots,m,$ are the columns of the matrix $g(\mathbf{x})^{-1}$.
\end{proposition}

\subsection{The effective operator} 
\label{Subsection Effective operator}
To describe homogenization procedure for the lower order terms of the operator $B_{D,\varepsilon}$, we need another cell problem. 
Let  $\widetilde{\Lambda}(\mathbf{x})$ be 
the $\Gamma$-periodic $(n\times n)$-matrix-valued solution of the problem
\begin{equation}
\label{tildeLambda_problem}
b(\mathbf{D})^*g(\mathbf{x})b(\mathbf{D})\widetilde{\Lambda }(\mathbf{x})+\sum \limits _{j=1}^dD_ja_j(\mathbf{x})^*=0,\quad \int _{\Omega }\widetilde{\Lambda }(\mathbf{x})\,d\mathbf{x}=0.
\end{equation}
(The equation is understood in the weak sense.) 
The following estimates were proven in \cite[(7.51), (7.52)]{SuAA}:
\begin{align}
\label{b(D) tilde Lambda <=}
&\Vert b(\mathbf{D})\widetilde{\Lambda}\Vert _{L_2(\Omega)}
\leqslant C_a n^{1/2}\alpha_0 ^{-1/2}\Vert g^{-1}\Vert _{L_\infty},\\
\label{tilde Lambda<=}
&\Vert \widetilde{\Lambda}\Vert _{L_2(\Omega)}\leqslant (2r_0)^{-1}C_an^{1/2}\alpha _0^{-1}\Vert g^{-1}\Vert _{L_\infty},
\\
\label{D tilde Lambda}
&\Vert \mathbf{D}\widetilde{\Lambda}\Vert _{L_2(\Omega)}\leqslant C_a n^{1/2}\alpha _0^{-1}\Vert g^{-1}\Vert _{L_\infty}.
\end{align}
Here $C_a^2=\sum _{j=1}^d \int _\Omega \vert a_j(\mathbf{x})\vert ^2\,d\mathbf{x}$. 

Next, we define the constant matrices $V$ and $W$ as follows:
\begin{align}
\label{V=}
&V:=\vert \Omega \vert ^{-1}\int _{\Omega}(b(\mathbf{D})\Lambda (\mathbf{x}))^*g(\mathbf{x})(b(\mathbf{D})\widetilde{\Lambda}(\mathbf{x}))\,d\mathbf{x},\\
\label{W=}
&W:=\vert \Omega \vert ^{-1}\int _{\Omega} (b(\mathbf{D})\widetilde{\Lambda}(\mathbf{x}))^*g(\mathbf{x})(b(\mathbf{D})\widetilde{\Lambda}(\mathbf{x}))\,d\mathbf{x}.
\end{align}

In $L_2(\mathcal{O};\mathbb{C}^n)$, consider the quadratic form
\begin{equation*}
\begin{split}
\mathfrak{b}_D^0[\mathbf{u},\mathbf{u}]&=(g^0b(\mathbf{D})\mathbf{u},b(\mathbf{D})\mathbf{u})_{L_2(\mathcal{O})}
+2\mathrm{Re}\,\sum _{j=1}^d (\overline{a_j}D_j\mathbf{u},\mathbf{u})_{L_2(\mathcal{O})}
-2\mathrm{Re}\,(V\mathbf{u},b(\mathbf{D})\mathbf{u})_{L_2(\mathcal{O})}\\
&-(W\mathbf{u},\mathbf{u})_{L_2(\mathcal{O})}
+(\overline{Q}\mathbf{u},\mathbf{u})_{L_2(\mathcal{O})}
+\lambda (\mathbf{u},\mathbf{u})_{L_2(\mathcal{O})},
\quad \mathbf{u}\in H^1_0(\mathcal{O};\mathbb{C}^n).
\end{split}
\end{equation*}
The following estimates were obtained in \cite[(2.22) and (2.23)]{MSuPOMI}:
\begin{align}
\label{b_D^0 ots }
&c_*\Vert \mathbf{D}\mathbf{u}\Vert ^2_{L_2(\mathcal{O})}\leqslant \mathfrak{b}_D^0[\mathbf{u},\mathbf{u}]\leqslant c_4\Vert \mathbf{u}\Vert ^2_{H^1(\mathcal{O})},\quad \mathbf{u}\in H^1_0(\mathcal{O};\mathbb{C}^n),
\\
\label{b_D^0 ots 2}
&\mathfrak{b}_D^0[\mathbf{u},\mathbf{u}]\geqslant c_*(\mathrm{diam}\,\mathcal{O})^{-2}\Vert \mathbf{u}\Vert ^2_{L_2(\mathcal{O})},\quad \mathbf{u}\in H^1_0(\mathcal{O};\mathbb{C}^n).
\end{align}
Here the constant $c_4$ depends only on the problem data \eqref{problem data}. 
By $B_D^0$ we denote the self-adjoint operator in $L_2(\mathcal{O};\mathbb{C}^n)$ corresponding to the form  $\mathfrak{b}_D^0$. Combining \eqref{b_D^0 ots } and \eqref{b_D^0 ots 2}, we obtain
\begin{equation}
\label{H^1-norm <= BD0^1/2}
\begin{split}
&\Vert \mathbf{u}\Vert _{H^1(\mathcal{O})}\leqslant c_3\Vert (B_D^0)^{1/2}\mathbf{u}\Vert _{L_2(\mathcal{O})},\quad \mathbf{u}\in H^1_0(\mathcal{O};\mathbb{C}^n),
\end{split}
\end{equation}
where $c_3$ is the constant from \eqref{H^1-norm <= BDeps^1/2}. By \eqref{b_D^0 ots 2} and \eqref{H^1-norm <= BD0^1/2},
\begin{align}
\label{B_D^0 L2 ->L2}
&\Vert (B_{D}^0)^{-1}\Vert _{L_2(\mathcal{O})\rightarrow L_2(\mathcal{O})}
\leqslant \mathcal{C}_1,
\\
&\Vert (B_{D}^0)^{-1}\Vert _{L_2(\mathcal{O})\rightarrow H^1(\mathcal{O})}
\leqslant \mathcal{C}_2.
\nonumber
\end{align}
Here the constants $\mathcal{C}_1$ and $\mathcal{C}_2$ are the same as in \eqref{B_D,eps L2 ->L2} and \eqref{B_D,eps L2 ->H^1}.

By the condition $\partial\mathcal{O}\in C^{1,1}$, the operator $B_D^0$ is defined by the differential expression
\begin{equation}
\label{B_D^0}
B^0=b(\mathbf{D})^*g^0b(\mathbf{D})-b(\mathbf{D})^*V-V^*b(\mathbf{D})
+\sum _{j=1}^d(\overline{a_j+a_j^*})D_j-W+\overline{Q}+\lambda I
\end{equation}
on the domain $H^2(\mathcal{O};\mathbb{C}^n)\cap H^1_0(\mathcal{O};\mathbb{C}^n)$, and 
\begin{equation}
\label{B_D^0 L2 ->H^2}
\Vert (B_D^0)^{-1}\Vert _{L_2(\mathcal{O})\rightarrow H^2(\mathcal{O})}\leqslant \mathcal{C}_3.
\end{equation}
Here the constant $\mathcal{C}_3$ depends only on the problem data \eqref{problem data}. To justify this fact, we refer to the theorems about regularity of solutions of the strongly elliptic systems (see \cite[Chapter~4]{McL}).

\begin{remark}
Instead of the condition $\partial\mathcal{O}\in C^{1,1}$, one could impose the following implicit condition\textnormal{:} a bounded Lipschitz domain $\mathcal{O}\subset \mathbb{R}^d$  
is such that estimate \eqref{B_D^0 L2 ->H^2} holds.  For such domain the main 
results of the paper in the operator terms \textnormal{(}see Theorems~\textnormal{\ref{Theorem cos New}, \ref{Theorem sin corrector}}, and \textnormal{\ref{Theorem cos corrector}}\textnormal{)} remain true. In the case of scalar elliptic operators, wide conditions on $\partial \mathcal{O}$ ensuring estimate \eqref{B_D^0 L2 ->H^2} can be found in \textnormal{\cite{KoE}} and \textnormal{\cite[Chapter 7]{MaSh} (}in particular,  it suffices to assume that  $\partial\mathcal{O}\in C^\alpha$, $\alpha >3/2${\rm)}.
\end{remark}

\begin{lemma}
\label{Lemma (B_D0)2 and H4-norm}
Let $B^0$ be the differential expression \eqref{B_D^0}. Then for $\boldsymbol{\Phi} \in  H^4(\mathcal{O};\mathbb{C}^n)$ we have
\begin{equation}
\label{BD0^2PHi<= Lm}
\Vert (B^0)^2\boldsymbol{\Phi} \Vert _{L_2(\mathcal{O})}\leqslant\mathfrak{C}\Vert \boldsymbol{\Phi} \Vert _{H^4(\mathcal{O})},
\end{equation}
where the constant $\mathfrak{C}$ depends only on the problem data \eqref{problem data}.
\end{lemma}

\begin{proof}
By \eqref{b_l <=}, \eqref{g0<=}, and \eqref{B_D^0}, for $\boldsymbol{\Psi}\in H^2(\mathcal{O};\mathbb{C}^n)$ we have
\begin{equation}
\label{BD0Phi<= start of the proof}
\begin{split}
\Vert B^0\boldsymbol{\Psi}\Vert _{L_2(\mathcal{O})}
&\leqslant
d\alpha _1 \Vert g\Vert _{L_\infty}\Vert \mathbf{D}^2 \boldsymbol{\Psi} \Vert _{L_2(\mathcal{O})}
+2\alpha _1^{1/2}d^{1/2}\vert V\vert \Vert \mathbf{D}\boldsymbol{\Psi} \Vert _{L_2(\mathcal{O})}
\\
&+2\Bigl(\sum _{j=1}^d\vert \overline{a_j}\vert ^2\Bigr)^{1/2}\Vert \mathbf{D}\boldsymbol{\Psi}\Vert _{L_2(\mathcal{O})}
+\left(\vert W\vert +\vert \overline{Q}\vert +\lambda\right)\Vert \boldsymbol{\Psi} \Vert _{L_2(\mathcal{O})}
.
\end{split}
\end{equation}

From \eqref{b(D)Lambda <=}, \eqref{b(D) tilde Lambda <=}, and \eqref{V=} it follows that
\begin{equation}
\label{V<=}
\vert V\vert\leqslant \vert \Omega\vert ^{-1}\Vert g\Vert _{L_\infty}\Vert b(\mathbf{D})\Lambda \Vert _{L_2(\Omega)}\Vert b(\mathbf{D})\widetilde{\Lambda}\Vert _{L_2(\Omega)}\leqslant C_V,
\end{equation}
where $C_V:=\vert \Omega\vert ^{-1/2}\alpha _0^{-1/2}C_a m^{1/2}n^{1/2}\Vert g\Vert ^{3/2}_{L_\infty}\Vert g^{-1}\Vert ^{3/2}_{L_\infty}$. 
By  \eqref{b(D) tilde Lambda <=} and \eqref{W=}, 
\begin{equation}
\label{W<=}
\vert W\vert \leqslant\vert \Omega\vert ^{-1}\Vert g\Vert _{L_\infty}\Vert b(\mathbf{D})\widetilde{\Lambda}\Vert ^2_{L_2(\Omega)}\leqslant C_W,
\end{equation}
where $C_W:=\vert \Omega\vert ^{-1}C_a^2n\alpha _0^{-1}\Vert g\Vert _{L_\infty}\Vert g^{-1}\Vert ^2_{L_\infty}$. 
Obviously,
\begin{equation}
\label{sum aj2<=}
\sum _{j=1}^d\vert \overline{a_j}\vert ^2\leqslant \vert \Omega\vert ^{-1}C_a^2,\quad
\vert \overline{Q}\vert \leqslant \vert \Omega\vert ^{-1/s}\Vert Q\Vert _{L_s(\Omega)}.
\end{equation}

Bringing \eqref{BD0Phi<= start of the proof}--\eqref{sum aj2<=} together, we conclude
\begin{equation}
\label{BDoPhi 5 st<=}
\begin{split}
\Vert B^0\boldsymbol{\Psi}\Vert _{L_2(\mathcal{O})}
\leqslant
C_B\left(\Vert \mathbf{D}^2 \boldsymbol{\Psi}\Vert _{L_2(\mathcal{O})}+\Vert \mathbf{D}\boldsymbol{\Psi}\Vert _{L_2(\mathcal{O})}+\Vert \boldsymbol{\Psi}\Vert _{L_2(\mathcal{O})}\right)
,\quad\boldsymbol{\Psi}\in H^2(\mathcal{O};\mathbb{C}^n).
\end{split}
\end{equation}
Here $C_B:=\max\lbrace d\alpha _1 \Vert g\Vert _{L_\infty};2(d\alpha _1)^{1/2}C_V+2C_a\vert \Omega\vert ^{-1/2};C_W+\vert \Omega\vert ^{-1/s}\Vert Q\Vert _{L_s(\Omega)}+\lambda\rbrace$. 
Below we will use \eqref{BDoPhi 5 st<=} with $\boldsymbol{\Psi}=B^0\boldsymbol{\Phi}$, $\boldsymbol{\Phi}\in H^4(\mathcal{O};\mathbb{C}^n)$.

By analogy with \eqref{BD0Phi<= start of the proof}, using \eqref{B_D^0} and \eqref{V<=}--\eqref{sum aj2<=}, we obtain
\begin{equation}
\begin{split}
\Vert &\mathbf{D}^2 B^0\boldsymbol{\Phi}\Vert _{L_2(\mathcal{O})}
\leqslant
\Vert \mathbf{D}^2 b(\mathbf{D})^*g^0b(\mathbf{D})\boldsymbol{\Phi}\Vert _{L_2(\mathcal{O})}
+\Vert \mathbf{D}^2b(\mathbf{D})^*V\boldsymbol{\Phi}\Vert _{L_2(\mathcal{O})}
+\Vert \mathbf{D}^2V^*b(\mathbf{D})\boldsymbol{\Phi}\Vert _{L_2(\mathcal{O})}
\\
&+\sum _{j=1}^d\Vert (\overline{a_j+a_j^*})\mathbf{D}^2D_j\boldsymbol{\Phi}\Vert _{L_2(\mathcal{O})}
+\Vert \mathbf{D}^2W\boldsymbol{\Phi}\Vert _{L_2(\mathcal{O})}
+\Vert \overline{Q}\mathbf{D}^2\boldsymbol{\Phi}\Vert _{L_2(\mathcal{O})}
+\lambda\Vert \mathbf{D}^2\boldsymbol{\Phi}\Vert _{L_2(\mathcal{O})}
\\
&\leqslant C_B\left(\Vert \mathbf{D}^4\boldsymbol{\Phi}\Vert _{L_2(\mathcal{O})}
+\Vert \mathbf{D}^3\boldsymbol{\Phi}\Vert _{L_2(\mathcal{O})}
+\Vert \mathbf{D}^2\boldsymbol{\Phi}\Vert _{L_2(\mathcal{O})}\right).
\end{split}
\end{equation}
Similarly,
\begin{align}
&\Vert \mathbf{D}B^0\boldsymbol{\Phi}\Vert _{L_2(\mathcal{O})}
\leqslant C_B\left(\Vert \mathbf{D}^3\boldsymbol{\Phi}\Vert _{L_2(\mathcal{O})}
+\Vert \mathbf{D}^2\boldsymbol{\Phi}\Vert _{L_2(\mathcal{O})}
+\Vert \mathbf{D}\boldsymbol{\Phi}\Vert _{L_2(\mathcal{O})}\right),
\\
\label{BDoPhi<= 8st}
&\Vert B^0\boldsymbol{\Phi}\Vert _{L_2(\mathcal{O})}
\leqslant C_B\left(\Vert \mathbf{D}^2\boldsymbol{\Phi}\Vert _{L_2(\mathcal{O})}
+\Vert \mathbf{D}\boldsymbol{\Phi}\Vert _{L_2(\mathcal{O})}
+\Vert \boldsymbol{\Phi}\Vert _{L_2(\mathcal{O})}\right).
\end{align}
Combining \eqref{BDoPhi 5 st<=}--\eqref{BDoPhi<= 8st}, we have
\begin{equation*}
\begin{split}
\Vert &(B^0)^2\boldsymbol{\Phi}\Vert_{L_2(\mathcal{O})}
\\
&\leqslant C_B^2\left(\Vert \mathbf{D}^4\boldsymbol{\Phi}\Vert _{L_2(\mathcal{O})}
+2\Vert \mathbf{D}^3\boldsymbol{\Phi}\Vert _{L_2(\mathcal{O})}
+3\Vert\mathbf{D}^2\boldsymbol{\Phi}\Vert _{L_2(\mathcal{O})}
+2\Vert \mathbf{D}\boldsymbol{\Phi}\Vert _{L_2(\mathcal{O})}
+\Vert \boldsymbol{\Phi}\Vert _{L_2(\mathcal{O})}\right)
\\
&\leqslant \sqrt{19}C_B^2\Vert \boldsymbol{\Phi}\Vert _{H^4(\mathcal{O})},\quad\boldsymbol{\Phi}\in H^4(\mathcal{O};\mathbb{C}^n).
\end{split}
\end{equation*}
We arrive at estimate \eqref{BD0^2PHi<= Lm} with the constant $\mathfrak{C}:=\sqrt{19}C_B^2$.
\end{proof}

\subsection{Approximation of the resolvent $(B_{D,\varepsilon}-\zeta I)^{-1}$}

Now we formulate the results of the paper \cite{MSuPOMI}, where the behavior of the 
resolvent $(B_{D,\varepsilon}-\zeta I)^{-1}$ was studied. See also the brief communication \cite{MSuFAA2017}.

We choose the numbers $\varepsilon _0$, $\varepsilon _1\in (0,1]$ according to the following condition.

\begin{condition}
\label{condition varepsilon}
Let $\mathcal{O}\subset \mathbb{R}^d$ be a bounded domain. Denote $$(\partial\mathcal{O})_{\varepsilon}: =\left\lbrace \mathbf{x}\in \mathbb{R}^d : \mathrm{dist}\,\lbrace \mathbf{x};\partial\mathcal{O}\rbrace <\varepsilon \right\rbrace .$$ 
Suppose that there exists a number $\varepsilon _0\in (0,1]$ such that the strip  $(\partial\mathcal{O})_{\varepsilon _0}$ can be covered by a
finite number of open sets admitting diffeomorphisms of class $C^{0,1}$ rectifying the boundary $\partial\mathcal{O}$. We set $\varepsilon _1:=\varepsilon _0 (1+r_1)^{-1}$, 
where $2r_1=\mathrm{diam}\,\Omega$.
\end{condition}
Obviously, the number $\varepsilon _1$ depends only on the domain $\mathcal{O}$ and the lattice $\Gamma$.

Note that Condition \ref{condition varepsilon} is ensured only by the assumption that $\partial\mathcal{O}$  is Lipschitz; we imposed a more restrictive condition  $\partial\mathcal{O}\in C^{1,1}$ in order to guarantee estimate \eqref{B_D^0 L2 ->H^2}.

The following result was obtained in \cite[Theorems 9.2 and 10.1]{MSuPOMI}.

\begin{theorem}
\label{Theorem resolvent}
Let $\mathcal{O}\subset
\mathbb{R}^d$ be a bounded domain of class $C^{1,1}$. Suppose that the assumptions of Subsec.~\textnormal{\ref{Subsection operatoer A_D,eps}--\ref{Subsection Effective operator}} are satisfied. Suppose that $\varepsilon _1$ is subject to Condition~\textnormal{\ref{condition varepsilon}}.

\noindent
$1^\circ$. Let $\zeta =\vert \zeta \vert e^{i\phi}\in \mathbb{C}\setminus \mathbb{R}_+$, $\vert \zeta \vert \geqslant 1$. Denote
\begin{equation*}
c(\phi):=\begin{cases}
\vert \sin \phi \vert ^{-1}, &\phi\in (0,\pi /2)\cup (3\pi /2 ,2\pi),\\
1, &\phi\in [\pi /2,3\pi /2].
\end{cases}
\end{equation*}
Then for $0<\varepsilon\leqslant\varepsilon _1$ and $\zeta\in\mathbb{C}\setminus\mathbb{R}_+$, $\vert\zeta\vert\geqslant 1$ we have
\begin{equation}
\label{sem'.a}
\Vert (B_{D,\varepsilon}-\zeta I)^{-1}-(B_D^0-\zeta I)^{-1}\Vert _{L_2(\mathcal{O})\rightarrow L_2(\mathcal{O})}
\leqslant C_1 c(\phi)^2\varepsilon\vert \zeta\vert ^{-1/2} .
\end{equation}

\noindent
$2^\circ$. 
Let $c_\flat$ be a common lower bound for the operators $B_D^0$ and $B_{D,\varepsilon}$ for $0<\varepsilon\leqslant \varepsilon _1$. 
Denote $\psi =\mathrm{arg}\,(\zeta -c_\flat)$, $0<\psi <2\pi$, and
\begin{equation}
\label{rho(zeta)}
\varrho _ \flat (\zeta):=\begin{cases}
c(\psi)^2\vert \zeta -c_\flat\vert ^{-2}, &\vert \zeta -c_\flat\vert <1,\\
c(\psi)^2, &\vert \zeta -c_\flat\vert \geqslant 1.
\end{cases}
\end{equation}
Then for $0<\varepsilon \leqslant \varepsilon _1$ and $\zeta\in\mathbb{C}\setminus[c_\flat,\infty)$ 
we have
\begin{align}
\label{Th dr appr 1}
\Vert &(B_{D,\varepsilon}-\zeta I )^{-1}-(B_D^0-\zeta I)^{-1}\Vert _{L_2(\mathcal{O})\rightarrow L_2(\mathcal{O})}\leqslant C_{2}\varrho _ \flat (\zeta )\varepsilon .
\end{align}

The constants $C_1$ and $C_2$ depend only on the problem data \eqref{problem data}.
\end{theorem}

The constant $c_\flat$ in Theorem \ref{Theorem resolvent}($2^\circ$) is any common lower bound for the operators  $B_D^0$ and $B_{D,\varepsilon}$. 
Taking into account inequalities \eqref{b D,eps >= H1-norm}, \eqref{b_D^0 ots 2}, and the expression for the constant $c_*$ (see \eqref{b_eps >=}), 
\textit{we choose}
\begin{equation}
\label{c_flat}
c_\flat :=4^{-1}\alpha _0\Vert g^{-1}\Vert ^{-1}_{L_\infty}(\mathrm{diam}\,\mathcal{O})^{-2}.
\end{equation}

Fix a linear continuous extension operator
\begin{equation}
\label{P_O H^1, H^2}
\begin{split}
P_\mathcal{O}: H^l(\mathcal{O};\mathbb{C}^n)\rightarrow H^l(\mathbb{R}^d;\mathbb{C}^n),\quad l\in\mathbb{Z}_+.
\end{split}
\end{equation}
Such a ,,universal'' extension operator exists for any Lipschitz bounded domain (see \cite{St} or \cite{R}).
We have
\begin{equation}
\label{PO}
\| P_{\mathcal O} \|_{H^l({\mathcal O}) \to H^l({\mathbb R}^d)} \leqslant C_{\mathcal O}^{(l)},\quad l\in\mathbb{Z}_+,
\end{equation}
where the constant $C_{\mathcal O}^{(l)}$ depends only on $l$ and the domain ${\mathcal O}$.
Let $R_\mathcal{O}$ be the operator of restriction of functions in $\mathbb{R}^d$ to the domain $\mathcal{O}$. 
Denote
\begin{equation}
\label{K_D(eps,zeta)}
K_D(\varepsilon ;\zeta ):=R_{\mathcal{O}}\bigl([\Lambda ^\varepsilon ] b(\mathbf{D})+[\widetilde{\Lambda}^\varepsilon ]\bigr)S_\varepsilon P_{\mathcal{O}}(B_D^0-\zeta I)^{-1}.
\end{equation}
The corrector \eqref{K_D(eps,zeta)} is a continuous operator acting from $L_2(\mathcal{O};\mathbb{C}^n)$ to $H^1(\mathcal{O};\mathbb{C}^n)$. This can be checked by using Proposition \ref{Proposition f^eps S_eps} and inclusions $\Lambda$, $\widetilde{\Lambda}\in\widetilde{H}^1(\Omega)$. Note that \break$\Vert \varepsilon K_D(\varepsilon ;\zeta)\Vert _{L_2(\mathcal{O})\rightarrow H^1(\mathcal{O})}=O(1)$ for small $\varepsilon$ and  $\zeta$ fixed. 

Let us formulate the results of \cite[Theorems 9.2 and 10.1]{MSuPOMI}.
\begin{theorem} 
\label{Theorem resolvent corrector}
Under the assumptions of Theorem \textnormal{\ref{Theorem resolvent}}, 
let $K_D(\varepsilon ;\zeta )$ be the operator \eqref{K_D(eps,zeta)}.

\noindent $1^\circ$. For $0<\varepsilon\leqslant\varepsilon _1$ and $\zeta\in\mathbb{C}\setminus\mathbb{R}_+$, $\vert\zeta\vert\geqslant 1$, we have
\begin{equation}
\label{28}
\Vert (B_{D,\varepsilon}-\zeta I)^{-1}-(B_D^0-\zeta I)^{-1}-\varepsilon K_D(\varepsilon ;\zeta )\Vert _{L_2(\mathcal{O})\rightarrow H^1(\mathcal{O})}
\leqslant C_3 c(\phi)^2(\varepsilon ^{1/2}\vert\zeta\vert ^{-1/4}+\varepsilon ).
\end{equation}

\noindent $2^\circ$.  Let $c_\flat$ be the constant \eqref{c_flat}. 
Then for $0<\varepsilon \leqslant \varepsilon _1$ and $\zeta\in\mathbb{C}\setminus[c_\flat,\infty)$ 
we have
\begin{equation}
\label{31}
\begin{split}
\Vert  &(B_{D,\varepsilon}-\zeta I)^{-1}-(B_D^0-\zeta I)^{-1}-\varepsilon K_D(\varepsilon ;\zeta )\Vert _{L_2(\mathcal{O})\rightarrow H^1(\mathcal{O})}\\
&\leqslant C_4(\varepsilon ^{1/2}\varrho _\flat (\zeta )^{1/2}+\varepsilon \vert 1+\zeta\vert ^{1/2}\varrho _\flat (\zeta)).
\end{split}
\end{equation}
The constants $C_3$ and $C_4$ depend only on the problem data \eqref{problem data}.
\end{theorem}

\begin{corollary}
\label{ellliptic corollary}
Under the assumptions of Theorem \textnormal{\ref{Theorem resolvent corrector},} 
for $0<\varepsilon\leqslant\varepsilon _1$ and \break$\zeta\in\mathbb{C}\setminus\mathbb{R}_+$, $\vert\zeta\vert\geqslant 1$, we have
\begin{equation}
\label{30}
\begin{split}
\Vert & (B_{D,\varepsilon}-\zeta I)^{-1}-(B_D^0-\zeta I)^{-1}-\varepsilon K_D(\varepsilon ;\zeta )\Vert _{L_2(\mathcal{O})\rightarrow H^1(\mathcal{O})}\\
&\leqslant  C_5\left( c(\phi)^2\varepsilon ^{1/2}\vert\zeta\vert ^{-1/4}+c(\phi)^{3/2}\varepsilon ^{1/2}\right).
\end{split}
\end{equation}
The constant $C_5$ depends only on the problem data \eqref{problem data}.
\end{corollary}

Corollary \ref{ellliptic corollary} follows from Theorem~\ref{Theorem resolvent corrector}($1^\circ$) 
and a rough estimate for the operators under the norm sign in \eqref{28}. We need the estimates for the resolvents $(B_{D,\varepsilon}-\zeta I)^{-1}$ and $(B_D^0-\zeta I)^{-1}$ (see \cite[Lemmas 2.1 and 2.3]{MSuPOMI}).

\begin{lemma}
\label{Lemma resolvents grubo ots}
For $0<\varepsilon\leqslant 1$ and $\zeta\in\mathbb{C}\setminus\mathbb{R}_+$ we have
\begin{align*}
&\Vert (B_{D,\varepsilon}-\zeta I)^{-1}\Vert _{L_2(\mathcal{O})\rightarrow L_2(\mathcal{O})}
\leqslant c(\phi)\vert \zeta\vert ^{-1},
\\
&\Vert \mathbf{D}(B_{D,\varepsilon}-\zeta I)^{-1}\Vert _{L_2(\mathcal{O})\rightarrow L_2(\mathcal{O})}
\leqslant \mathcal{C}_4 c(\phi)\vert \zeta\vert ^{-1/2},
\\
&\Vert (B_D^0-\zeta I)^{-1}\Vert _{L_2(\mathcal{O})\rightarrow L_2(\mathcal{O})}
\leqslant c(\phi)\vert \zeta\vert ^{-1},
\\
&\Vert \mathbf{D}(B_D^0-\zeta I)^{-1}\Vert _{L_2(\mathcal{O})\rightarrow L_2(\mathcal{O})}
\leqslant \mathcal{C}_4 c(\phi)\vert \zeta\vert ^{-1/2},
\\
&\Vert (B_D^0-\zeta I)^{-1}\Vert _{L_2(\mathcal{O})\rightarrow H^2(\mathcal{O})}
\leqslant\mathcal{C}_3c(\phi).
\end{align*}
Here $\mathcal{C}_4:=2^{3/2}\alpha _0^{-1/2}\Vert g^{-1}\Vert ^{1/2}_{L_\infty}$ and the constant $\mathcal{C}_3$ is the same as in \eqref{B_D^0 L2 ->H^2}.
\end{lemma}

\begin{proof}[Proof of Corollary \textnormal{\ref{ellliptic corollary}}.]
Let us estimate the operator \eqref{K_D(eps,zeta)}:
\begin{equation}
\label{ell corr proof start}
\begin{split}
\varepsilon &\Vert K_{D}(\varepsilon ;\zeta)\Vert _{L_2(\mathcal{O})\rightarrow H^1(\mathcal{O})}
\\
&\leqslant
\left(\varepsilon\Vert \Lambda ^\varepsilon S_\varepsilon \Vert _{L_2(\mathbb{R}^d)\rightarrow L_2(\mathbb{R}^d)}+\Vert (\mathbf{D}\Lambda )^\varepsilon S_\varepsilon\Vert _{L_2(\mathbb{R}^d)\rightarrow L_2(\mathbb{R}^d)}\right)
\\
&\times
\Vert b(\mathbf{D})P_\mathcal{O}(B_D^0-\zeta I)^{-1}\Vert _{L_2(\mathcal{O})\rightarrow L_2(\mathbb{R}^d)}
\\
&+\left(\varepsilon\Vert \widetilde{\Lambda}^\varepsilon S_\varepsilon \Vert _{L_2(\mathbb{R}^d)\rightarrow L_2(\mathbb{R}^d)}+\Vert (\mathbf{D}\widetilde{\Lambda})^\varepsilon S_\varepsilon\Vert _{L_2(\mathbb{R}^d)\rightarrow L_2(\mathbb{R}^d)}\right)
\\
&\times
\Vert P_\mathcal{O} (B_D^0-\zeta I)^{-1}\Vert _{L_2(\mathcal{O})\rightarrow L_2(\mathbb{R}^d)}
\\
&+\varepsilon \Vert \Lambda^\varepsilon S_\varepsilon\Vert _{L_2(\mathbb{R}^d)\rightarrow L_2(\mathbb{R}^d)}\Vert \mathbf{D} b(\mathbf{D})P_\mathcal{O}(B_D^0-\zeta I)^{-1}\Vert_{L_2(\mathcal{O})\rightarrow L_2(\mathbb{R}^d)}
\\
&+\varepsilon\Vert \widetilde{\Lambda}^\varepsilon S_\varepsilon \Vert _{L_2(\mathbb{R}^d)\rightarrow L_2(\mathbb{R}^d)}\Vert \mathbf{D}P_\mathcal{O}(B_D^0-\zeta I)^{-1}\Vert_{L_2(\mathcal{O})\rightarrow L_2(\mathbb{R}^d)}.
\end{split}
\end{equation}
By Proposition~\ref{Proposition f^eps S_eps} and inequalities \eqref{Lambda <=}, \eqref{DLambda<=}, \eqref{tilde Lambda<=}, and \eqref{D tilde Lambda},
\begin{align}
\label{Lambda S_eps<=}
&\Vert \Lambda ^\varepsilon S_\varepsilon \Vert _{L_2(\mathbb{R}^d)\rightarrow L_2(\mathbb{R}^d)}\leqslant M_1,
\\
&\Vert (\mathbf{D}\Lambda )^\varepsilon S_\varepsilon\Vert _{L_2(\mathbb{R}^d)\rightarrow L_2(\mathbb{R}^d)}\leqslant M_2,
\nonumber
\\
\label{tildeLambda S_eps<=}
&\Vert \widetilde{\Lambda}^\varepsilon S_\varepsilon \Vert _{L_2(\mathbb{R}^d)\rightarrow L_2(\mathbb{R}^d)}\leqslant  \vert\Omega\vert ^{-1/2}(2r_0)^{-1}C_an^{1/2}\alpha _0^{-1}\Vert g^{-1}\Vert _{L_\infty}=:\widetilde{M}_1,
\\
&\Vert (\mathbf{D}\widetilde{\Lambda} )^\varepsilon S_\varepsilon \Vert _{L_2(\mathbb{R}^d)\rightarrow L_2(\mathbb{R}^d)}\leqslant\vert\Omega\vert ^{-1/2}C_a n^{1/2}\alpha _0^{-1}\Vert g^{-1}\Vert _{L_\infty}=:\widetilde{M}_2.
\nonumber
\end{align}
Combining these estimates, Lemma~\ref{Lemma resolvents grubo ots}, and \eqref{<b^*b<}, \eqref{PO}, \eqref{ell corr proof start}, for $0<\varepsilon\leqslant 1$, $\zeta\in\mathbb{C}\setminus\mathbb{R}_+$, $\vert\zeta\vert\geqslant 1$ we have
\begin{align}
\label{29}
&\Vert (B_{D,\varepsilon}-\zeta I)^{-1}-(B_D^0-\zeta I)^{-1}-\varepsilon K_D(\varepsilon ;\zeta )\Vert _{L_2(\mathcal{O})\rightarrow H^1(\mathcal{O})}
\leqslant \widehat{C}_5 c(\phi),
\\ 
&\widehat{C}_5:=\left(2+(M_1+M_2)\alpha_1^{1/2}C_\mathcal{O}^{(1)}+\widetilde{M}_1C_\mathcal{O}^{(1)}\right)
(1+\mathcal{C}_4)
+(\widetilde{M}_1+\widetilde{M}_2)C_\mathcal{O}^{(0)}+M_1\alpha_1^{1/2}C_\mathcal{O}^{(2)}\mathcal{C}_3
.\nonumber
\end{align}

Combining \eqref{28} and \eqref{29}, for $0<\varepsilon\leqslant\varepsilon _1$ and $\zeta\in\mathbb{C}\setminus\mathbb{R}_+$, $\vert\zeta\vert \geqslant 1$, we obtain
\begin{equation*}
\begin{split}
\Vert & (B_{D,\varepsilon}-\zeta I)^{-1}-(B_D^0-\zeta I)^{-1}-\varepsilon K_D(\varepsilon ;\zeta )\Vert _{L_2(\mathcal{O})\rightarrow H^1(\mathcal{O})}\\
&\leqslant \min\lbrace  C_3 c(\phi)^2(\varepsilon ^{1/2}\vert\zeta\vert ^{-1/4}+\varepsilon );\widehat{C}_5 c(\phi)\rbrace
\\
&\leqslant   C_3 c(\phi)^2\varepsilon ^{1/2}\vert\zeta\vert ^{-1/4}+\min\lbrace  C_3 c(\phi)^2\varepsilon ; \widehat{C}_5 c(\phi)\rbrace
\\
&\leqslant  C_3 c(\phi)^2\varepsilon ^{1/2}\vert\zeta\vert ^{-1/4}+(C_3\widehat{C}_5)^{1/2}c(\phi)^{3/2}\varepsilon ^{1/2}.
\end{split}
\end{equation*}
We arrive at estimate \eqref{30} with the constant $C_5:=\max\lbrace C_3;(C_3\widehat{C}_5)^{1/2}\rbrace$.
\end{proof}

\subsection{Approximation of the operator $B_{D,\varepsilon}^{-1/2}$}

The following result is a consequence of Theorem \ref{Theorem resolvent}.

\begin{lemma}
\label{Lemma square root of resolvent}
Under the conditions of Theorem \textnormal{\ref{Theorem resolvent},} for $0<\varepsilon\leqslant\varepsilon _1$ we have
\begin{equation}
\label{lemma square root of resolvent}
\Vert B_{D,\varepsilon}^{-1/2}-(B_D^0)^{-1/2}\Vert _{L_2(\mathcal{O})\rightarrow L_2(\mathcal{O})}
\leqslant {C}_6\varepsilon ^{1/2}.
\end{equation}
The constant ${C}_6$ depends only on the problem data \eqref{problem data}.
\end{lemma}
\begin{proof}
We have 
$
B_{D,\varepsilon}^{-1/2}=\pi ^{-1}\int _0^\infty \nu ^{-1/2}(B_{D,\varepsilon}+\nu I)^{-1}\,d\nu $. 
See, e.~g., \cite[Chapter III, \S 3, Subsec. 4]{ViGKo}. For $(B_D^0)^{-1/2}$ we have the similar identity. Hence,
\begin{equation*}
\begin{split}
\Vert & B_{D,\varepsilon}^{-1/2}-(B_D^0)^{-1/2}\Vert _{L_2(\mathcal{O})\rightarrow L_2(\mathcal{O})}
\\
&\leqslant \pi ^{-1}\int _0^\infty \nu ^{-1/2}\Vert (B_{D,\varepsilon}+\nu I)^{-1}-(B_D^0 +\nu I)^{-1}\Vert _{L_2(\mathcal{O})\rightarrow L_2(\mathcal{O})}\,d\nu .
\end{split}
\end{equation*}
Since $c_\flat $ is a common lower bound for the operators $B_{D,\varepsilon}$ and $B_D^0$, 
\begin{equation*}
\Vert (B_{D,\varepsilon}+\nu I)^{-1}-(B_D^0+\nu I)^{-1}\Vert _{L_2(\mathcal{O})\rightarrow L_2(\mathcal{O})}
\leqslant 2(\nu +c_\flat )^{-1},\quad \nu\in\mathbb{R}_+ .
\end{equation*}
Thus, 
\begin{equation*}
\begin{split}
\Vert & B_{D,\varepsilon}^{-1/2}-(B_D^0)^{-1/2}\Vert _{L_2(\mathcal{O})\rightarrow L_2(\mathcal{O})}
\\
&\leqslant 2^{1/2} \pi ^{-1}\int _0^\infty \nu ^{-1/2}(\nu +c_\flat )^{-1/2}\Vert (B_{D,\varepsilon}+\nu I)^{-1}-(B_D^0 +\nu I)^{-1}\Vert ^{1/2} _{L_2(\mathcal{O})\rightarrow L_2(\mathcal{O})}\,d\nu .
\end{split}
\end{equation*}
For $\nu \in [0,1]$, we use \eqref{Th dr appr 1}:
\begin{equation*}
\Vert (B_{D,\varepsilon}+\nu I)^{-1}-(B_D^0 +\nu I)^{-1}\Vert  _{L_2(\mathcal{O})\rightarrow L_2(\mathcal{O})}
\leqslant C_2\varepsilon \max \lbrace 1; (c_\flat +\nu)^{-2}\rbrace
\leqslant C_2\varepsilon \max \lbrace 1;c_\flat ^{-2}\rbrace .
\end{equation*}
For $\nu >1$, we apply estimate \eqref{sem'.a}:
\begin{equation*}
\Vert (B_{D,\varepsilon}+\nu I)^{-1}-(B_D^0 +\nu I)^{-1}\Vert  _{L_2(\mathcal{O})\rightarrow L_2(\mathcal{O})}
\leqslant C_1\varepsilon \nu ^{-1/2},\quad \nu >1.
\end{equation*}
It follows that
\begin{equation*}
\begin{split}
\Vert & B_{D,\varepsilon}^{-1/2}-(B_D^0)^{-1/2}\Vert _{L_2(\mathcal{O})\rightarrow L_2(\mathcal{O})}
\\
&\leqslant 2^{1/2} \pi ^{-1}C_2^{1/2}\max\lbrace 1; c_\flat ^{-1}\rbrace \varepsilon ^{1/2}
\int _0 ^1\nu ^{-1/2}(\nu+c_\flat )^{-1/2}\,d\nu
\\
&+2^{1/2}\pi ^{-1}C_1^{1/2}\varepsilon ^{1/2}\int _1^\infty \nu ^{-1/2}(\nu +c_\flat )^{-1/2}\nu ^{-1/4}\,d\nu .
\end{split}
\end{equation*}
Evaluating these integrals, we arrive at estimate \eqref{lemma square root of resolvent} with the constant
$$ {C}_6:=2^{3/2}\pi ^{-1}C_2^{1/2}c_\flat ^{-1/2}\max\lbrace 1;c_\flat ^{-1}\rbrace+2^{5/2}\pi ^{-1}C_1^{1/2}.$$
\end{proof}

\section{Problem setting. Main results}
\label{Section 2}

\subsection{The first initial-boundary value problem for hyperbolic systems}
\textit{Our goal }is to study the behavior of the solution of the following problem for small $\varepsilon$:
\begin{equation}
\label{hyperbolic system}
\begin{cases}
\frac{\partial ^2\mathbf{u}_\varepsilon }{ \partial t^2}(\mathbf{x},t)=-(B_\varepsilon\mathbf{u}_\varepsilon )(\mathbf{x},t)+\mathbf{F}(\mathbf{x},t),\\
\mathbf{u}_\varepsilon (\cdot,t)\vert _{\partial\mathcal{O}}=0,\\
\mathbf{u}_\varepsilon (\mathbf{x},0)=\boldsymbol{\varphi}(\mathbf{x}),\quad \frac{\partial \mathbf{u}_\varepsilon}{\partial t}(\mathbf{x},0)=\boldsymbol{\psi}(\mathbf{x}).
\end{cases}
\end{equation}
Here $\boldsymbol{\varphi},\boldsymbol{\psi}\in\mathrm{Dom}\,(B_D^0)^{2}$, $\mathbf{F}\in L_{1,\mathrm{loc}}(\mathbb{R};\mathrm{Dom}\,(B_D^0)^{2})$. (The imposed restrictions are caused by the technique used in the present work.) We have
\begin{equation}
\label{u_eps=}
\mathbf{u}_\varepsilon (\cdot ,t)=\cos (tB_{D,\varepsilon}^{1/2})\boldsymbol{\varphi}
+B_{D,\varepsilon}^{-1/2}\sin (tB_{D,\varepsilon}^{1/2})\boldsymbol{\psi}+\int _0^t B_{D,\varepsilon}^{-1/2}\sin \left((t-\widetilde{t})B_{D,\varepsilon}^{1/2}\right)\mathbf{F}(\cdot ,\widetilde{t})\,d\widetilde{t}.
\end{equation}
So, to study the behavior of $\mathbf{u}_\varepsilon (\cdot,t)$ 
it suffices to obtain approximations for the operators $\cos (tB_{D,\varepsilon}^{1/2})$ and $B_{D,\varepsilon}^{-1/2}\sin (tB_{D,\varepsilon}^{1/2})$.

The effective problem is given by
\begin{equation}
\label{effective hyperbolic system}
\begin{cases}
\frac{\partial ^2\mathbf{u}_0 }{ \partial t^2}(\mathbf{x},t)=-(B^0 \mathbf{u}_0 )(\mathbf{x},t)+\mathbf{F}(\mathbf{x},t),\\
\mathbf{u}_0(\cdot,t)\vert _{\partial\mathcal{O}}=0,\\
\mathbf{u}_0 (\mathbf{x},0)=\boldsymbol{\varphi}(\mathbf{x}),\quad \frac{\partial \mathbf{u}_0}{\partial t}(\mathbf{x},0)=\boldsymbol{\psi}(\mathbf{x}).
\end{cases}
\end{equation}
Then
\begin{equation}
\label{u_0=}
\begin{split}
\mathbf{u}_0 (\cdot ,t)&=\cos \left(t(B_D^0)^{1/2}\right)\boldsymbol{\varphi}
+(B_D^0)^{-1/2}\sin \left(t(B_D^0)^{1/2}\right)\boldsymbol{\psi}
\\
&+\int _0^t (B_D^0)^{-1/2}\sin \left((t-\widetilde{t})(B_D^0)^{1/2}\right)\mathbf{F}(\cdot ,\widetilde{t})\,d\widetilde{t}.
\end{split}
\end{equation}

\subsection{Main results in the operator terms}
\label{Subsection main results}

\begin{theorem}
\label{Theorem cos New}
Let $\mathcal{O}\subset
\mathbb{R}^d$ be a bounded domain of class $C^{1,1}$. Suppose that the assumptions of Subsec.~\textnormal{\ref{Subsection operatoer A_D,eps}--\ref{Subsection Effective operator}} are satisfied. Let $\varepsilon _1$ be subject to Condition~\textnormal{\ref{condition varepsilon}}.  
Then for $t\in\mathbb{R}$ and $0<\varepsilon\leqslant\varepsilon _1$ we have
\begin{align}
\label{Th cos B_D,eps}
\Bigl\Vert &
\left(
\cos (t B_{D,\varepsilon}^{1/2})-\cos (t (B_D^0)^{1/2})
\right)(B_D^0)^{-2}\Bigr\Vert _{L_2(\mathcal{O})\rightarrow L_2(\mathcal{O})}
\leqslant C_7 \varepsilon \left(1+\vert t\vert ^5\right) ,
\\
\label{Th sin 1}
\Bigl\Vert & \left( B_{D,\varepsilon}^{-1/2}\sin (tB_{D,\varepsilon}^{1/2})-(B_D^0)^{-1/2}\sin (t(B_D^0)^{1/2})\right) (B_D^0)^{-2}\Bigr\Vert _{L_2(\mathcal{O})\rightarrow L_2(\mathcal{O})}
\leqslant C_7\varepsilon \vert t\vert (1+\vert t\vert ^5) .
\end{align}
The constant $C_7$ depends only on the problem data \eqref{problem data}.
\end{theorem}

It seems natural to expect that, for hyperbolic systems, the analog of Theorem~\ref{Theorem resolvent corrector} holds true. However, according to the results of \cite{BrOFMu}, it is impossible to approximate the operator $\cos (tB_{D,\varepsilon}^{1/2})$ in the energy norm, while the operators $B_{D,\varepsilon}^{-1}\cos (tB_{D,\varepsilon}^{1/2})$ and $B_{D,\varepsilon}^{-1/2}\sin (tB_{D,\varepsilon}^{1/2})$ can be approximated (see Theorems~\ref{Theorem sin corrector} and \ref{Theorem cos corrector} below). We also approximate the operator $g^\varepsilon b(\mathbf{D})B_{D,\varepsilon}^{-1/2}\sin (tB_{D,\varepsilon}^{1/2})$ which corresponds to the ,,flux.''

\begin{theorem}
\label{Theorem sin corrector}
Suppose that the assumptions of Theorem \textnormal{\ref{Theorem cos New}} are satisfied. Let matrix-valued functions $\Lambda (\mathbf{x})$ and $\widetilde{\Lambda}(\mathbf{x})$ be the $\Gamma$-periodic solutions of problems \eqref{Lambda problem} and \eqref{tildeLambda_problem}, respectively. Let $S_\varepsilon $ be the Steklov smoothing operator \eqref{S_eps} and let $P_\mathcal{O}$ be the linear extension operator \eqref{P_O H^1, H^2}. 
Then for $t\in\mathbb{R}$ and $0<\varepsilon\leqslant\varepsilon_1$ we have
\begin{equation}
\label{Th sin 2}
\begin{split}
\Bigl\Vert &
\bigl( B_{D,\varepsilon}^{-1/2}\sin (tB_{D,\varepsilon}^{1/2})-(B_D^0)^{-1/2}\sin (t(B_D^0)^{1/2})
\\
&-\varepsilon \bigl(\Lambda ^\varepsilon b(\mathbf{D})+\widetilde{\Lambda}^\varepsilon\bigr)S_\varepsilon P_\mathcal{O}(B_D^0)^{-1/2}\sin (t(B_D^0)^{1/2})\bigr)(B_D^0)^{-2}\Bigr\Vert _{L_2(\mathcal{O})\rightarrow H^1(\mathcal{O})}
\\
&\leqslant C_8\varepsilon ^{1/2} (1+ t ^6).
\end{split}
\end{equation}
Let $\widetilde{g}(\mathbf{x})$ be the matrix-valued function \eqref{tilde g}. 
Denote
\begin{equation*}
G_D(\varepsilon ;t):=\left(\widetilde{g}^\varepsilon S_\varepsilon b(\mathbf{D})+g^\varepsilon (b(\mathbf{D})\widetilde{\Lambda})^\varepsilon S_\varepsilon \right)
P_\mathcal{O}(B_D^0)^{-1/2}\sin (t(B_D^0)^{1/2}).
\end{equation*}
Then for $t\in\mathbb{R}$ and $0<\varepsilon\leqslant\varepsilon_1$ we have
\begin{equation}
\label{Th fluxes operator terms}
\begin{split}
\bigl\Vert \bigl(
g^\varepsilon b(\mathbf{D})B_{D,\varepsilon}^{-1/2}\sin (tB_{D,\varepsilon}^{1/2})
-
G_D(\varepsilon ;t)
\bigr)(B_D^0)^{-2}\bigr\Vert _{L_2(\mathcal{O})\rightarrow L_2(\mathcal{O})}
\leqslant
C_9 \varepsilon ^{1/2}  (1+ t ^6).
\end{split}
\end{equation}
Here the constants $C_8$ and $C_9$ depend only on the problem data \eqref{problem data}.
\end{theorem}

The proofs of Theorems~\ref{Theorem cos New} and \ref{Theorem sin corrector} are given below in Section~\ref{Section 3}.

\subsection{On approximation of the operator $\cos (tB_{D,\varepsilon}^{1/2})B_{D,\varepsilon}^{-1}$ in the energy norm}

\begin{theorem}
\label{Theorem cos corrector}
Under the assumptions of Theorem~\textnormal{\ref{Theorem sin corrector},} for $t\in\mathbb{R}\setminus \lbrace 0\rbrace$ and $0<\varepsilon\leqslant\varepsilon _1$ 
we have
\begin{equation}
\label{Th cos with correction term}
\begin{split}
\Bigl\Vert &\Bigl(\cos (tB_{D,\varepsilon}^{1/2})B_{D,\varepsilon}^{-1}-\cos \bigl(t (B_D^0)^{1/2}\bigr)(B_D^0)^{-1}
\\
&-\varepsilon \bigl(\Lambda ^\varepsilon b(\mathbf{D})+\widetilde{\Lambda}^\varepsilon\bigr)S_\varepsilon P_\mathcal{O}\cos \bigl(t (B_D^0)^{1/2}\bigr)(B_D^0)^{-1}\Bigr)(B_D^0)^{-1}\Bigr\Vert _{L_2(\mathcal{O})\rightarrow H^1(\mathcal{O})}
\\
&\leqslant C_{10}\varepsilon ^{1/2}(1+\vert t\vert ^5).
\end{split}
\end{equation}
The constant $C_{10}$ depends only on the problem data \eqref{problem data}.
\end{theorem}

The proof of Theorem~\textnormal{\ref{Theorem cos corrector}} is given below in Section~\ref{Section 3}.

Theorem~\textnormal{\ref{Theorem cos corrector}} allows us to obtain approximation in the energy class for the solution of the hyperbolic problem with the special choice of the initial data:
\begin{equation*}
\partial _t^2\mathbf{u}_\varepsilon =-B_{\varepsilon}\mathbf{u}_\varepsilon,
\quad\mathbf{u}_\varepsilon (\cdot ,t)\vert _{\partial\mathcal{O}}=0,
\quad\mathbf{u}_\varepsilon \vert _{t=0}=B_{D,\varepsilon}^{-1}\boldsymbol{\varphi},\quad (\partial _t\mathbf{u}_\varepsilon )\vert _{t=0}=0.
\end{equation*}
Here $\boldsymbol{\varphi}\in\mathrm{Dom}\,B_D^0= H^2(\mathcal{O};\mathbb{C}^n)\cap H^1_0(\mathcal{O};\mathbb{C}^n)$. In this case, the effective problem has the form
\begin{equation*}
\partial _t^2\mathbf{u}_0 =-B^0\mathbf{u}_0,
\quad\mathbf{u}_0(\cdot,t)\vert _{\partial\mathcal{O}}=0,
\quad\mathbf{u}_0 \vert _{t=0}=(B_{D}^0)^{-1}\boldsymbol{\varphi},\quad (\partial _t\mathbf{u}_0 )\vert _{t=0}=0.
\end{equation*}
From \eqref{BDoPhi 5 st<=} and \eqref{Th cos with correction term} it follows that
\begin{equation*}
\Vert \mathbf{u}_\varepsilon (\cdot ,t)-\mathbf{u}_0(\cdot ,t)-\varepsilon (\Lambda ^\varepsilon b(\mathbf{D})+\widetilde{\Lambda}^\varepsilon)S_\varepsilon \widetilde{\mathbf{u}}_0(\cdot,t)\Vert _{H^1(\mathcal{O})}
\leqslant 3^{1/2}C_BC_{10}\varepsilon ^{1/2}(1+\vert t\vert ^5)\Vert \boldsymbol{\varphi}\Vert _{H^2(\mathcal{O})}.
\end{equation*}
For such choice of the initial data, the possibility to approximate the solution in the energy class is in accordance with the results of \cite{BrOFMu}. 

Note that Lemma~\ref{Lemma square root of resolvent}, Theorem~\ref{Theorem cos corrector}, and estimates \eqref{H^1-norm <= BDeps^1/2}, \eqref{B_D^0 L2 ->L2} allow us to approximate the operator $\cos(tB_{D,\varepsilon}^{1/2})B_{D,\varepsilon}^{-1/2}$. We have
\begin{equation*}
\begin{split}
\Vert &\cos (t B_{D,\varepsilon}^{1/2})B_{D,\varepsilon}^{-1/2}(B_{D,\varepsilon}^{-1/2}-(B_D^0)^{-1/2})(B_D^0)^{-1}\Vert _{L_2(\mathcal{O})\rightarrow H^1(\mathcal{O})}
\\
&\leqslant
c_3\Vert B_{D,\varepsilon}^{-1/2}-(B_D^0)^{-1/2}\Vert _{L_2(\mathcal{O})\rightarrow L_2(\mathcal{O})}\Vert (B_D^0)^{-1}\Vert _{L_2(\mathcal{O})\rightarrow L_2(\mathcal{O})}
\leqslant
c_3C_6\mathcal{C}_1\varepsilon^{1/2}.
\end{split}
\end{equation*}
Combining this with \eqref{Th cos with correction term}, we obtain
\begin{equation*}
\begin{split}
\Bigl\Vert & \bigl(\cos (tB_{D,\varepsilon}^{1/2})B_{D,\varepsilon}^{-1/2}-(I+\varepsilon (\Lambda ^\varepsilon b(\mathbf{D})+\widetilde{\Lambda}^\varepsilon)S_\varepsilon P_\mathcal{O})
\cos(t(B_D^0)^{1/2})(B_D^0)^{-1/2}\Bigr)
\\
&\times
(B_D^0)^{-3/2}\Bigr\Vert _{L_2(\mathcal{O})\rightarrow H^1(\mathcal{O})}
\leqslant(C_{10}+c_3C_6\mathcal{C}_1)\varepsilon ^{1/2}(1+\vert t\vert ^5).
\end{split}
\end{equation*}

\subsection{Removal of the smoothing operator from the corrector}
\label{Subsection main results no S_eps}

It turns out that the smoothing operator can be removed from the corrector if the matrix-valued functions  $\Lambda (\mathbf{x})$ and $\widetilde{\Lambda}(\mathbf{x})$ are subject to some additional assumptions. 

\begin{condition}
\label{Condition Lambda in L infty}
Assume that the $\Gamma$-periodic solution $\Lambda (\mathbf{x})$ of problem \eqref{Lambda problem} is bounded, i.~e., $\Lambda\in L_\infty (\mathbb{R}^d)$.
\end{condition}

Some cases where Condition~\ref{Condition Lambda in L infty} is fulfilled automatically  were distinguished in \cite[Lemma~8.7]{BSu06}.

\begin{proposition}
\label{Proposition Lambda in L infty <=}
Suppose that at least one of the following assumptions is satisfied\textnormal{:}

\noindent
$1^\circ )$ $d\leqslant 2${\rm ;}

\noindent
$2^\circ )$ the dimension $d\geqslant 1$ is arbitrary, and the differential expression $A_\varepsilon$ is given by $A_\varepsilon =\mathbf{D}^* g^\varepsilon (\mathbf{x})\mathbf{D}$, where $g(\mathbf{x})$ is a symmetric matrix with real entries{\rm ;}

\noindent
$3^\circ )$ the dimension $d$ is arbitrary, and $g^0=\underline{g}$, i.~e., relations \eqref{underline-g} are satisfied.

\noindent
Then Condition~\textnormal{\ref{Condition Lambda in L infty}}  holds.
\end{proposition}

In order to remove $S_\varepsilon$ from the term involving $\widetilde{\Lambda}^\varepsilon$, it suffices to impose the following condition.

\begin{condition}
\label{Condition tilde Lambda in Lp}
Assume that the $\Gamma$-periodic solution $\widetilde{\Lambda}(\mathbf{x})$ of problem \eqref{tildeLambda_problem} is such that
\begin{equation*}
\widetilde{\Lambda}\in L_p(\Omega),\quad p=2 \;\mbox{for}\;d=1,\quad p>2\;\mbox{for}\;d=2, \quad p=d \;\mbox{for}\;d\geqslant 3.
\end{equation*}
\end{condition}

The following result was obtained in \cite[Proposition 8.11]{SuAA}.

\begin{proposition}
\label{Proposition tilde Lambda in Lp if}
Condition \textnormal{\ref{Condition tilde Lambda in Lp}} is fulfilled, if at least one of the following assumptions is satisfied\textnormal{:}

\noindent
$1^\circ )$ $d\leqslant 4${\rm ;}

\noindent
$2^\circ )$ the dimension $d$ is arbitrary, and the differential expression $A_{\varepsilon}$ has the form \break $A_{\varepsilon} =\mathbf{D}^* g^\varepsilon (\mathbf{x})\mathbf{D}$, where $g(\mathbf{x})$ is a~symmetric matrix-valued function with real entries. 
\end{proposition}

\begin{remark}
\label{Remark scalar problem}
If $A_{\varepsilon} =\mathbf{D}^* g^\varepsilon (\mathbf{x})\mathbf{D}$, where $g(\mathbf{x})$ is a~symmetric matrix-valued function with real entries, from \textnormal{\cite[Chapter {\rm III}, Theorem {\rm 13.1}]{LaU}} it follows that $\Lambda,\widetilde{\Lambda}\in L_\infty$ and the norm $\Vert\Lambda\Vert _{L_\infty}$ is controlled in terms of $d$, $\Vert g\Vert _{L_\infty}$, $\Vert g^{-1}\Vert _{L_\infty}$, and $\Omega$\textnormal{;} the norm $\Vert \widetilde{\Lambda}\Vert _{L_\infty }$ does not exceed a constant depending on $d$, $\rho$, $\Vert g\Vert _{L_\infty}$, $\Vert g^{-1}\Vert _{L_\infty}$, $\Vert a_j\Vert _{L_\rho (\Omega)}$, $j=1,\dots ,d$, and $\Omega$. In this case, Conditions \textnormal{\ref{Condition Lambda in L infty}} and \textnormal{\ref{Condition tilde Lambda in Lp}} are fulfilled simultaneously.
\end{remark}

In this subsection, our goal is to prove the following statement.

\begin{theorem}
\label{Theorem no S-eps} 
Suppose that the assumptions of Theorem~\textnormal{\ref{Theorem sin corrector}} are satisfied. Assume that the matrix-valued function $\Lambda(\mathbf{x})$ is subject to Condition \textnormal{\ref{Condition Lambda in L infty}} and the matrix-valued function $\widetilde{\Lambda}(\mathbf{x})$ satisfies Condition \textnormal{\ref{Condition tilde Lambda in Lp}}. Denote
\begin{equation}
\label{GD0}
G_D^0(\varepsilon ;t):=\left(\widetilde{g}^\varepsilon b(\mathbf{D})+g^\varepsilon (b(\mathbf{D})\widetilde{\Lambda})^\varepsilon  \right)
(B_D^0)^{-1/2}\sin (t(B_D^0)^{1/2}).
\end{equation}
Then for $t\in\mathbb{R}$ and $0<\varepsilon\leqslant\varepsilon _1$ we have
\begin{align}
\label{Th sin corr no S_eps-1}
\begin{split}
\Bigl\Vert &
\bigl( B_{D,\varepsilon}^{-1/2}\sin (tB_{D,\varepsilon}^{1/2})-
 \bigl(I+\varepsilon\Lambda ^\varepsilon b(\mathbf{D})+\varepsilon\widetilde{\Lambda}^\varepsilon\bigr)
 (B_D^0)^{-1/2}\sin (t(B_D^0)^{1/2})\bigr)
 \\
 &\times
 (B_D^0)^{-2}\Bigr\Vert _{L_2(\mathcal{O})\rightarrow H^1(\mathcal{O})}
\leqslant C_{11} \varepsilon ^{1/2}  (1+ t ^6),
\end{split}
\\
\label{Th sin corr no S_eps flux-1}
\begin{split}
\bigl\Vert & \bigl(
g^\varepsilon b(\mathbf{D})B_{D,\varepsilon}^{-1/2}\sin (tB_{D,\varepsilon}^{1/2})
-
G_D^0(\varepsilon ;t)
\bigr)(B_D^0)^{-2}\bigr\Vert _{L_2(\mathcal{O})\rightarrow L_2(\mathcal{O})}
\leqslant
C_{12} \varepsilon ^{1/2}  (1+ t ^6).
\end{split}
\end{align}
The constants $C_{11}$ and $C_{12}$ depend only on the problem data \eqref{problem data}, on $p$, and on the norms $\Vert \Lambda\Vert _{L_\infty}$, $\Vert \widetilde{\Lambda}\Vert _{L_p(\Omega)}$.
\end{theorem}

To prove Theorem \ref{Theorem no S-eps}, we need the following results obtained in  \cite[Lemmas~7.7 and 7.8]{MSuPOMI}.

\begin{lemma}
\label{Lemma Lambda (S-I)}
Let $\Gamma$-periodic matrix-valued solution $\Lambda (\mathbf{x})$ of problem \eqref{Lambda problem} satisfy Condition~\textnormal{\ref{Condition Lambda in L infty}}. Let $S_\varepsilon$ be the Steklov smoothing operator \eqref{S_eps}. Then for $0<\varepsilon\leqslant 1$ 
\begin{equation*}
\Vert [\Lambda ^\varepsilon ]b(\mathbf{D})(S_\varepsilon -I)\Vert _{H^2(\mathbb{R}^d)\rightarrow H^1(\mathbb{R}^d)}\leqslant\mathfrak{C}_\Lambda .
\end{equation*}
The constant $\mathfrak{C}_\Lambda$ depends only on $m$, $d$, $\alpha _0$, $\alpha _1$, $\Vert g\Vert _{L_\infty}$, $\Vert g^{-1}\Vert _{L_\infty}$, on the parameters of the lattice $\Gamma$, and on the norm $\Vert \Lambda\Vert _{L_\infty}$.
\end{lemma}

\begin{lemma}
\label{Lemma tilde Lambda(S-I)}
Let matrix-valued $\Gamma$-periodic solution $\widetilde{\Lambda }(\mathbf{x})$ of problem \eqref{tildeLambda_problem} satisfy Condition \textnormal{\ref{Condition tilde Lambda in Lp}}. Let $S_\varepsilon$ be the Steklov smoothing operator \eqref{S_eps}. Then for $0<\varepsilon\leqslant 1$
\begin{equation*}
\Vert [\widetilde{\Lambda}^\varepsilon ](S_\varepsilon -I)\Vert _{H^2(\mathbb{R}^d)\rightarrow H^1(\mathbb{R}^d)}\leqslant \mathfrak{C}_{\widetilde{\Lambda}}.
\end{equation*}
The constant $\mathfrak{C}_{\widetilde{\Lambda}}$ is controlled in terms of $n$, $d$, $\alpha _0$, $\alpha _1$, $\rho$, $\Vert g\Vert _{L_\infty}$, $\Vert g^{-1}\Vert _{L_\infty}$, the norms $\Vert a_j\Vert _{L_\rho (\Omega)}$, $j=1,\dots,d$, $p$, $\Vert \widetilde{\Lambda}\Vert _{L_p(\Omega)}$, and the parameters of the lattice $\Gamma$.
\end{lemma}

The following assertion can be easily checked by using the H\"older inequality and the Sobolev embedding theorem (cf. \cite[Lemma 3.5]{MSu15}).

\begin{lemma}
\label{Lemma tilde Lambda norm}
Assume that the matrix-valued function $\widetilde{\Lambda}(\mathbf{x})$ satisfies Condition~\textnormal{\ref{Condition tilde Lambda in Lp}}. Then for $0<\varepsilon\leqslant 1$ the operator $[\widetilde{\Lambda}^\varepsilon]$ is a continuous mapping from $H^1(\mathbb{R}^d;\mathbb{C}^n)$ to $L_2(\mathbb{R}^d;\mathbb{C}^n)$ and
\begin{equation*}
\Vert [\widetilde{\Lambda}^\varepsilon ]\Vert _{H^1(\mathbb{R}^d)\rightarrow L_2(\mathbb{R}^d)}
\leqslant\Vert \widetilde{\Lambda}\Vert _{L_p(\Omega)}C_\Omega (p),
\end{equation*}
where $C_\Omega (p)$ is the norm of the embedding operator $H^1(\Omega)\hookrightarrow L_{2(p/2)'}(\Omega)$. Here $(p/2)'=\infty$ for $d=1$ and $(p/2)'=p/(p-2)$ for $d\geqslant 2$.
\end{lemma}

\begin{proof}[Proof of Theorem~\textnormal{\ref{Theorem no S-eps}}.]
The result of Theorem~\ref{Theorem no S-eps} can be derived from Theorem~\ref{Theorem sin corrector} with the help of Lemmas \ref{Lemma Lambda (S-I)}, \ref{Lemma tilde Lambda(S-I)}, and \ref{Lemma tilde Lambda norm}.

By Lemma \ref{Lemma Lambda (S-I)} and \eqref{PO}, 
\begin{equation}
\label{proof no S_eps-1}
\begin{split}
\Vert &\Lambda ^\varepsilon b(\mathbf{D})(S_\varepsilon -I)P_\mathcal{O}(B_D^0)^{-5/2}\sin (t(B_D^0)^{1/2})\Vert _{L_2(\mathcal{O})\rightarrow H^1(\mathbb{R}^d)}
\\
&\leqslant \mathfrak{C}_\Lambda C_\mathcal{O}^{(2)}\Vert (B_D^0)^{-5/2}\sin (t(B_D^0)^{1/2})\Vert _{L_2(\mathcal{O})\rightarrow H^2(\mathcal{O})},\quad t\in\mathbb{R},\quad 0<\varepsilon\leqslant 1.
\end{split}
\end{equation}
We have
\begin{equation}
\label{proof no S_eps-2-a}
\begin{split}
\Vert & (B_D^0)^{-5/2}\sin (t(B_D^0)^{1/2})\Vert _{L_2(\mathcal{O})\rightarrow H^2(\mathcal{O})}
\\
&\leqslant \Vert (B_D^0)^{-1}\Vert _{L_2(\mathcal{O})\rightarrow H^2(\mathcal{O})}
\Vert (B_D^0)^{-1}\Vert _{L_2(\mathcal{O})\rightarrow L_2(\mathcal{O})}
\Vert (B_D^0)^{-1/2}\sin (t(B_D^0)^{1/2})\Vert _{L_2(\mathcal{O})\rightarrow L_2(\mathcal{O})}.
\end{split}
\end{equation}
By the spectral theorem and the elementary inequality $\vert \sin\mu\vert /\vert \mu\vert \leqslant 1$, $\mu\in\mathbb{R}$, 
\begin{equation*}
\Vert (B_D^0)^{-1/2}\sin (t(B_D^0)^{1/2})\Vert _{L_2(\mathcal{O})\rightarrow L_2(\mathcal{O})}
\leqslant\vert t\vert.
\end{equation*}
Combining this with \eqref{B_D^0 L2 ->L2}, \eqref{B_D^0 L2 ->H^2}, and \eqref{proof no S_eps-2-a}, we obtain
\begin{equation}
\label{proof no S_eps-2}
\begin{split}
\Vert  (B_D^0)^{-5/2}\sin (t(B_D^0)^{1/2})\Vert _{L_2(\mathcal{O})\rightarrow H^2(\mathcal{O})}
\leqslant\mathcal{C}_1\mathcal{C}_3\vert t\vert ,\quad t\in\mathbb{R}.
\end{split}
\end{equation}

From Lemma \ref{Lemma tilde Lambda(S-I)} and \eqref{PO}, \eqref{proof no S_eps-2} it follows that
\begin{equation}
\label{proof no S_eps-3}
\begin{split}
\Vert &\widetilde{\Lambda}^\varepsilon (S_\varepsilon -I)P_\mathcal{O}(B_D^0)^{-5/2}\sin (t(B_D^0)^{1/2})\Vert _{L_2(\mathcal{O})\rightarrow H^1(\mathbb{R}^d)}
\\
&\leqslant
\mathfrak{C}_{\widetilde{\Lambda}}C_\mathcal{O}^{(2)}\Vert (B_D^0)^{-5/2}\sin (t(B_D^0)^{1/2})\Vert _{L_2(\mathcal{O})\rightarrow H^2(\mathcal{O})}
\\
&\leqslant \mathfrak{C}_{\widetilde{\Lambda}}C_\mathcal{O}^{(2)}\mathcal{C}_1\mathcal{C}_3\vert t\vert 
,\quad t\in\mathbb{R},\quad 0<\varepsilon\leqslant 1.
\end{split}
\end{equation}

Bringing together \eqref{Th sin 2}, \eqref{proof no S_eps-1}, \eqref{proof no S_eps-2}, and \eqref{proof no S_eps-3}, we arrive at estimate \eqref{Th sin corr no S_eps-1} with the constant $C_{11}:=C_8+
(\mathfrak{C}_\Lambda +\mathfrak{C}_{\widetilde{\Lambda}})C_\mathcal{O}^{(2)}\mathcal{C}_1\mathcal{C}_3 $. Here the inequality $\vert t\vert\leqslant (1+ t ^6)$, $t\in\mathbb{R}$, is taken into account.

Now we proceed to the proof of inequality \eqref{Th sin corr no S_eps flux-1}. By \eqref{b_l <=} and \eqref{Th sin corr no S_eps-1}, 
\begin{equation}
\label{proof no S_eps-5.a}
\begin{split}
\Bigl\Vert & \Bigl(
g^\varepsilon b(\mathbf{D})B_{D,\varepsilon}^{-1/2}\sin (t B_{D,\varepsilon}^{1/2})
-g^\varepsilon b(\mathbf{D})\bigl(I+\varepsilon \Lambda ^\varepsilon b(\mathbf{D})+\varepsilon\widetilde{\Lambda}^\varepsilon\bigr)
(B_D^0)^{-1/2}\sin (t(B_D^0)^{1/2})
\Bigr)
\\
&\times
(B_D^0)^{-2}
\Bigr\Vert _{L_2(\mathcal{O})\rightarrow L_2(\mathcal{O})}
\leqslant (d\alpha _1)^{1/2}\Vert g\Vert _{L_\infty}C_{11}\varepsilon ^{1/2} (1+ t ^6),\quad t\in\mathbb{R},\quad 0<\varepsilon\leqslant\varepsilon _1.
\end{split}
\end{equation}
We have
\begin{equation}
\label{proof no S_eps-6}
\begin{split}
g^\varepsilon & b(\mathbf{D})\bigl(I+\varepsilon \Lambda ^\varepsilon b(\mathbf{D})+\varepsilon\widetilde{\Lambda}^\varepsilon\bigr)
(B_D^0)^{-5/2}\sin (t(B_D^0)^{1/2})
\\
&=
g^\varepsilon b(\mathbf{D})(B_D^0)^{-5/2}\sin (t(B_D^0)^{1/2})
+g^\varepsilon (b(\mathbf{D})\Lambda)^\varepsilon b(\mathbf{D})(B_D^0)^{-5/2}\sin (t(B_D^0)^{1/2})
\\
&+ g^\varepsilon \bigl(b(\mathbf{D})\widetilde{\Lambda}\bigr)^\varepsilon (B_D^0)^{-5/2}\sin (t(B_D^0)^{1/2})
\\
&+ \varepsilon \sum _{l=1}^d g^\varepsilon b_l \bigl(\Lambda ^\varepsilon b(\mathbf{D})+\widetilde{\Lambda}^\varepsilon \bigr) D_l (B_D^0)^{-5/2}\sin (t(B_D^0)^{1/2}).
\end{split}
\end{equation}
To estimate the fourth term in the right-hand side of \eqref{proof no S_eps-6}, we use Conditions \ref{Condition Lambda in L infty} and \ref{Condition tilde Lambda in Lp}, Lemma \ref{Lemma tilde Lambda norm}, and inequality \eqref{b_l <=}:
\begin{equation*}
\begin{split}
\Biggl\Vert & \varepsilon \sum _{l=1}^d g^\varepsilon b_l \bigl(\Lambda ^\varepsilon b(\mathbf{D})+\widetilde{\Lambda}^\varepsilon \bigr) D_l (B_D^0)^{-5/2}\sin (t(B_D^0)^{1/2})
\Biggr\Vert _{L_2(\mathcal{O})\rightarrow L_2(\mathcal{O})}
\\
&\leqslant
\varepsilon (d\alpha _1)^{1/2}\Vert g\Vert _{L_\infty}\Vert \Lambda\Vert _{L_\infty}\Vert b(\mathbf{D})\mathbf{D}(B_D^0)^{-5/2}\sin (t(B_D^0)^{1/2})\Vert _{L_2(\mathcal{O})\rightarrow L_2(\mathcal{O})}
\\
&+\varepsilon (d\alpha _1)^{1/2}\Vert g\Vert _{L_\infty}\Vert \widetilde{\Lambda}\Vert _{L_p(\Omega)}C_\Omega (p)\Vert P_\mathcal{O}\mathbf{D} (B_D^0)^{-5/2}\sin (t(B_D^0)^{1/2})\Vert _{L_2(\mathcal{O})\rightarrow H^1(\mathbb{R}^d)}.
\end{split}
\end{equation*}
Together with \eqref{b_l <=}, \eqref{PO}, and \eqref{proof no S_eps-2}, this implies
\begin{equation}
\label{proof no S_eps-7.a}
\begin{split}
\Biggl\Vert  \varepsilon \sum _{l=1}^d g^\varepsilon b_l \bigl(\Lambda ^\varepsilon b(\mathbf{D})+\widetilde{\Lambda}^\varepsilon \bigr) D_l (B_D^0)^{-5/2}\sin (t(B_D^0)^{1/2})
\Biggr\Vert _{L_2(\mathcal{O})\rightarrow L_2(\mathcal{O})}
\leqslant \varepsilon \vert t\vert \widehat{C}_{12},\quad t\in\mathbb{R},
\end{split}
\end{equation}
where $\widehat{C}_{12}:=(d\alpha _1)^{1/2}\mathcal{C}_1\mathcal{C}_3\Vert g\Vert _{L_\infty}\bigl(
(d\alpha _1)^{1/2}\Vert \Lambda\Vert _{L_\infty}
+C_\Omega (p)C_\mathcal{O}^{(1)}\Vert \widetilde{\Lambda}\Vert _{L_p(\Omega)}\bigr)$. 

From \eqref{tilde g} and \eqref{proof no S_eps-5.a}--\eqref{proof no S_eps-7.a} we derive estimate \eqref{Th sin corr no S_eps flux-1} with the constant $$C_{12}:=(d\alpha_1)^{1/2}\Vert g\Vert _{L_\infty}C_{11}+\widehat{C}_{12}.$$
\end{proof}

\subsection{Removal of the smoothing operator from the corrector for $3\leqslant d\leqslant 8$}
\label{Subsection main results no S_eps, 3<=d=<8}

If $d\leqslant 2$, then, according to Pro\-po\-si\-tions~\ref{Proposition Lambda in L infty <=} and \ref{Proposition tilde Lambda in Lp if}, Theorem~\ref{Theorem no S-eps} is applicable. So, let $d\geqslant 3$. Now we are interested in the possibility to remove the smoothing operator from the corrector without any additional assumptions on the matrix-valued functions $\Lambda$ and $\widetilde{\Lambda}$. 

If $3\leqslant d\leqslant 8$ and the boundary $\partial\mathcal{O}$ is sufficiently smooth, it turns out that the smoothing operator $S_\varepsilon$ can be eliminated from the both terms of the corrector. To do this, we use the properties of the matrix-valued functions $\Lambda ^\varepsilon$ and $\widetilde{\Lambda}^\varepsilon$ as multipliers. The following result was obtained in  \cite[Lemmas 6.3 and 6.5, Corollaries 6.4 and 6.6]{MSuAA17}.

\begin{lemma}
\label{Lemma Lambda multiplicator properties}
Let the~matrix-valued function $\Lambda({\mathbf x})$ be the $\Gamma$-periodic solution of problem 
\textnormal{\eqref{Lambda problem}}. Assume that $d\geqslant 3$ and put $l=d/2$. 

\noindent $1^\circ$. For $0< \varepsilon \leqslant 1$,
the operator $[\Lambda^\varepsilon]$ is a continuous mapping from  $H^{l-1}({\mathcal O};{\mathbb C}^m)$ to $L_2({\mathcal O};{\mathbb C}^n)$ 
and
\begin{equation*}
\|[\Lambda^\varepsilon] \|_{ H^{l-1}({\mathcal O}) \to L_2({\mathcal O})}
\leqslant C^{(0)} .
\end{equation*}

\noindent $2^\circ$. Let $0< \varepsilon \leqslant 1$. Then for the function
 ${\mathbf u} \in H^{l}({\mathbb R}^d; {\mathbb C}^m)$ we have the inclusion
$\Lambda^\varepsilon {\mathbf u} \in H^1({\mathbb R}^d;{\mathbb C}^n)$ and the estimate
\begin{equation*}
\| \Lambda^\varepsilon {\mathbf u}\|_{H^1({\mathbb R}^d)} \leqslant  C^{(1)} \varepsilon^{-1}
\| {\mathbf u}\|_{L_2({\mathbb R}^d)} + C^{(2)}  \|{\mathbf u}\|_{H^l({\mathbb R}^d)}.
\end{equation*}
The constants $C^{(0)}$, $C^{(1)}$, and $C^{(2)}$ depend on $m$, $d$, $\alpha_0$, $\alpha_1$,
$\|g\|_{L_\infty}$, $\|g^{-1}\|_{L_\infty}$, and the parameters of the lattice $\Gamma$. 
\end{lemma}

\begin{lemma}
\label{Lemma tilde Lambda multiplicator properties}
Let the matrix-valued function $\widetilde{\Lambda}(\mathbf{x})$ be the $\Gamma$-periodic solution of problem \eqref{tildeLambda_problem}. Assume that $d\geqslant 3$ and put $l=d/2$.

\noindent
$1^\circ$. For $0<\varepsilon\leqslant 1$, the operator $[\widetilde{\Lambda}^\varepsilon]$ is a continuous mapping from $H^{l-1}(\mathcal{O};\mathbb{C}^n)$ to $L_2(\mathcal{O};\mathbb{C}^n)$ and
\begin{equation*}
\Vert [\widetilde{\Lambda}^\varepsilon ]\Vert _{H^{l-1}(\mathcal{O})\rightarrow L_2(\mathcal{O})}
\leqslant \widetilde{C}^{(0)}.
\end{equation*}

\noindent
$2^\circ$. Let $0<\varepsilon\leqslant 1$. Then for $\mathbf{u}\in H^l(\mathbb{R}^d;\mathbb{C}^n)$ we have the inclusion $\widetilde{\Lambda}^\varepsilon\mathbf{u}\in H^1(\mathbb{R}^d;\mathbb{C}^n)$ and the estimate
\begin{equation*}
\Vert \widetilde{\Lambda}^\varepsilon\mathbf{u}\Vert _{H^1(\mathbb{R}^d)}
\leqslant \widetilde{C}^{(1)}\varepsilon ^{-1}\Vert \mathbf{u}\Vert _{H^1(\mathbb{R}^d)}
+\widetilde{C}^{(2)}\Vert \mathbf{u}\Vert _{H^l(\mathbb{R}^d)}.
\end{equation*}
The constants   
$
\widetilde{C}^{(0)}$, $ \widetilde{C}^{(1)}$, and $ \widetilde{C}^{(2)}$ 
depend only on the problem data \eqref{problem data}.
\end{lemma}

According to theorems about regularity of solutions of strongly elliptic systems (see, e.~g., \cite[Theorem~4.18]{McL}), the following assertion holds true.

\begin{lemma}
\label{Lemma with conditions on boundary}
Let $3\leqslant d\leqslant 8$. Assume that $\partial\mathcal{O}\in C^{d/2,1}$ if $d$ is even and $\partial\mathcal{O}\in C^{(d+1)/2,1}$ if $d$ is odd. Then the operator $(B_D^0)^{-5/2}$ is a continuous mapping from $L_2(\mathcal{O};\mathbb{C}^n)$ to $H^{d/2+1}(\mathcal{O};\mathbb{C}^n)$ and
\begin{equation}
\label{BD0 -5/2 to Hl+1}
\Vert (B_D^0)^{-5/2}\Vert _{L_2(\mathcal{O})\rightarrow H^{l+1}(\mathcal{O})}\leqslant\mathscr{C}_l,\quad l=d/2.
\end{equation}
\end{lemma}

Note that for any $d\geqslant 1$ and $\partial\mathcal{O}\in C^{4,1}$ the operator $(B_D^0)^{-5/2}:L_2(\mathcal{O};\mathbb{C}^n)\rightarrow H^5(\mathcal{O};\mathbb{C}^n)$ is continuous and
\begin{equation}
\label{BD0 -5/2 to H5}
\Vert (B_D^0)^{-5/2}\Vert _{L_2(\mathcal{O})\rightarrow H^5(\mathcal{O})}\leqslant \mathcal{C}_5.
\end{equation}

\begin{theorem}
\label{Theorem d<=6 no S_eps}
Suppose that the assumptions of Theorem~\textnormal{\ref{Theorem sin corrector}} are satisfied. Let $3\leqslant d\leqslant 8$ and let $\partial\mathcal{O}$ be subject to conditions of Lemma~\textnormal{\ref{Lemma with conditions on boundary}}. Let 
$G_D^0(\varepsilon ;t)$ be the operator \eqref{GD0}. 
Then for $t\in\mathbb{R}$ and $0<\varepsilon\leqslant\varepsilon _1$ we have
\begin{align}
\label{Th sin corr no S_eps}
\begin{split}
\Bigl\Vert &
\bigl( B_{D,\varepsilon}^{-1/2}\sin (tB_{D,\varepsilon}^{1/2})-
 \bigl(I+\varepsilon\Lambda ^\varepsilon b(\mathbf{D})+\varepsilon\widetilde{\Lambda}^\varepsilon\bigr)(B_D^0)^{-1/2}\sin (t(B_D^0)^{1/2})\bigr)
 \\
 &\times
 (B_D^0)^{-2}\Bigr\Vert _{L_2(\mathcal{O})\rightarrow H^1(\mathcal{O})}
\leqslant C_{13} \varepsilon ^{1/2}  (1+ t ^6),
\end{split}
\\
\label{Th sin corr no S_eps flux}
\begin{split}
\bigl\Vert & \bigl(
g^\varepsilon b(\mathbf{D})B_{D,\varepsilon}^{-1/2}\sin (tB_{D,\varepsilon}^{1/2})
-
G_D^0(\varepsilon ;t)
\bigr)(B_D^0)^{-2}\bigr\Vert _{L_2(\mathcal{O})\rightarrow L_2(\mathcal{O})}
\leqslant
C_{14} \varepsilon ^{1/2}  (1+ t ^6).
\end{split}
\end{align}
The constants $C_{13}$ and $C_{14}$ depend only on the problem data \eqref{problem data}.
\end{theorem}

\begin{proof}
By Proposition~\ref{Proposition S__eps - I}, Lemma~\ref{Lemma Lambda multiplicator properties}($2^\circ$), and \eqref{<b^*b<},
\begin{equation*}
\begin{split}
\varepsilon\Vert &[\Lambda ^\varepsilon ](S_\varepsilon -I)b(\mathbf{D})P_\mathcal{O}(B_D^0)^{-1/2}\sin (t(B_D^0)^{1/2})(B_D^0)^{-2}\Vert _{L_2(\mathcal{O})\rightarrow H^1(\mathcal{O})}
\\
&\leqslant
C^{(1)}\Vert (S_\varepsilon -I)b(\mathbf{D})P_\mathcal{O}(B_D^0)^{-5/2}\Vert _{L_2(\mathcal{O})\rightarrow L_2(\mathbb{R}^d)}
\\
&+\varepsilon C^{(2)}\Vert (S_\varepsilon -I)b(\mathbf{D})P_\mathcal{O}(B_D^0)^{-5/2}\Vert _{L_2(\mathcal{O})\rightarrow H^l(\mathbb{R}^d)}
\\
&\leqslant
r_1\alpha _1^{1/2}C^{(1)}\varepsilon\Vert \mathbf{D}^2P_\mathcal{O}(B_D^0)^{-5/2}\Vert _{L_2(\mathcal{O})\rightarrow L_2(\mathbb{R}^d)}
\\
&+
2\alpha _1^{1/2}C^{(2)}\varepsilon \Vert \mathbf{D}P_\mathcal{O}(B_D^0)^{-5/2}\Vert _{L_2(\mathcal{O})\rightarrow H^l(\mathbb{R}^d)}.
\end{split}
\end{equation*}
Together with \eqref{PO} and \eqref{BD0 -5/2 to Hl+1}, this implies
\begin{equation}
\label{term with Lambda}
\begin{split}
&\varepsilon\Vert [\Lambda ^\varepsilon ](S_\varepsilon -I)b(\mathbf{D})P_\mathcal{O}(B_D^0)^{-1/2}\sin (t(B_D^0)^{1/2})(B_D^0)^{-2}\Vert _{L_2(\mathcal{O})\rightarrow H^1(\mathcal{O})}
\\
&\leqslant
\varepsilon\alpha _1^{1/2}\left(r_1C^{(1)}C_\mathcal{O}^{(2)}\Vert (B_D^0)^{-5/2}\Vert _{L_2(\mathcal{O})\rightarrow H^2(\mathcal{O})}+2C^{(2)}C_\mathcal{O}^{(l+1)}\Vert (B_D^0)^{-5/2}\Vert _{L_2(\mathcal{O})\rightarrow H^{l+1}(\mathcal{O})}\right)
\\
&\leqslant
\varepsilon\alpha _1^{1/2}\left(r_1C^{(1)}C_\mathcal{O}^{(2)}
+2C^{(2)}C_\mathcal{O}^{(l+1)}\right)\mathscr{C}_l,\quad 3\leqslant d=2l\leqslant 8.
\end{split}
\end{equation}

Similarly, using Lemma~\ref{Lemma tilde Lambda multiplicator properties}($2^\circ$), we obtain
\begin{equation}
\label{term with tilde Lambda}
\begin{split}
\varepsilon \Vert & [\widetilde{\Lambda}^\varepsilon ](S_\varepsilon -I)P_\mathcal{O}(B_D^0)^{-1/2}\sin (t(B_D^0)^{1/2})(B_D^0)^{-2}\Vert _{L_2(\mathcal{O})\rightarrow H^1(\mathcal{O})}
\\
&\leqslant
\varepsilon\left(r_1\widetilde{C}^{(1)}C_\mathcal{O}^{(2)}+2\widetilde{C}^{(2)}C_\mathcal{O}^{(l)}\right)
\mathscr{C}_l,\quad 3\leqslant d=2l\leqslant 8.
\end{split}
\end{equation}

Combining \eqref{Th sin 2}, \eqref{term with Lambda}, and \eqref{term with tilde Lambda}, we arrive at estimate \eqref{Th sin corr no S_eps} with the constant \break$C_{13}:=C_8+\alpha _1^{1/2}\left(r_1C^{(1)}C_\mathcal{O}^{(2)}
+2C^{(2)}C_\mathcal{O}^{(l+1)}\right)\mathscr{C}_l+\left(r_1\widetilde{C}^{(1)}C_\mathcal{O}^{(2)}+2\widetilde{C}^{(2)}C_\mathcal{O}^{(l)}\right)
\mathscr{C}_l$. 

Now we proceed to the proof of inequality \eqref{Th sin corr no S_eps flux}. By \eqref{b_l <=} and \eqref{Th sin corr no S_eps},
\begin{equation}
\label{proof no S_eps-5}
\begin{split}
\Bigl\Vert & \Bigl(
g^\varepsilon b(\mathbf{D})B_{D,\varepsilon}^{-1/2}\sin (t B_{D,\varepsilon}^{1/2})
-g^\varepsilon b(\mathbf{D})\bigl(I+\varepsilon \Lambda ^\varepsilon b(\mathbf{D})+\varepsilon\widetilde{\Lambda}^\varepsilon\bigr)
(B_D^0)^{-1/2}\sin (t(B_D^0)^{1/2})
\Bigr)
\\
&\times
(B_D^0)^{-2}
\Bigr\Vert _{L_2(\mathcal{O})\rightarrow L_2(\mathcal{O})}
\leqslant (d\alpha _1)^{1/2}\Vert g\Vert _{L_\infty}C_{13}\varepsilon ^{1/2} (1+ t ^6),\quad t\in\mathbb{R},\quad 0<\varepsilon\leqslant\varepsilon _1.
\end{split}
\end{equation}
Identity \eqref{proof no S_eps-6} holds true. To estimate the fourth term in the right-hand side of 
\eqref{proof no S_eps-6}, we apply Lemmas~\ref{Lemma Lambda multiplicator properties}($1^\circ$) and \ref{Lemma tilde Lambda multiplicator properties}($1^\circ$) and inequalities  \eqref{b_l <=} and \eqref{BD0 -5/2 to Hl+1}:
\begin{equation}
\label{proof no S_eps-7}
\begin{split}
&\Biggl\Vert  \varepsilon \sum _{l=1}^d g^\varepsilon b_l \bigl(\Lambda ^\varepsilon b(\mathbf{D})+\widetilde{\Lambda}^\varepsilon \bigr) D_l (B_D^0)^{-5/2}\sin (t(B_D^0)^{1/2})
\Biggr\Vert _{L_2(\mathcal{O})\rightarrow L_2(\mathcal{O})}
\\
&\leqslant
\varepsilon \Vert g\Vert _{L_\infty}(d\alpha _1)^{1/2}
\left\Vert \left(\Lambda ^\varepsilon b(\mathbf{D})+\widetilde{\Lambda}^\varepsilon\right)
\mathbf{D}(B_D^0)^{-5/2}\right\Vert _{L_2(\mathcal{O})\rightarrow L_2(\mathcal{O})}
\\
&\leqslant
\varepsilon \Vert g\Vert _{L_\infty}(d\alpha _1)^{1/2}
\\
&\times\left(
C^{(0)}(d\alpha _1)^{1/2}\Vert \mathbf{D}^2(B_D^0)^{-5/2}\Vert _{L_2(\mathcal{O})\rightarrow H^{l-1}(\mathcal{O})}+\widetilde{C}^{(0)}\Vert \mathbf{D}(B_D^0)^{-5/2}\Vert _{L_2(\mathcal{O})\rightarrow H^{l-1}(\mathcal{O})}\right)
\\
&\leqslant
\varepsilon\widehat{C}_{14},\quad d\leqslant 8,\quad
\widehat{C}_{14}:=\Vert g\Vert _{L_\infty}(d\alpha _1)^{1/2}
\left(C^{(0)}(d\alpha _1)^{1/2}+\widetilde{C}^{(0)}\right)\mathscr{C}_l.
\end{split}
\end{equation}
Relations \eqref{proof no S_eps-5} and \eqref{proof no S_eps-7} imply the required estimate \eqref{Th sin corr no S_eps flux} with the constant $$C_{14}:=(d\alpha_1)^{1/2}\Vert g\Vert _{L_\infty}C_{13}+\widehat{C}_{14}.$$
\end{proof}

\begin{remark}
If $\partial\mathcal{O}\in C^{4,1}$ and $d=9,10$, it is possible to remove the smoothing operator $S_\varepsilon$ only from the term of the corrector containing $\widetilde{\Lambda}^\varepsilon$. To do this,  we use estimate \eqref{BD0 -5/2 to H5} instead of Lemma~\textnormal{\ref{Lemma with conditions on boundary}}.
\end{remark}

\subsection{Homogenization for the solution of the first initial-boundary value problem}

Now we apply the results of Subsec.~\ref{Subsection main results} and \ref{Subsection main results no S_eps} to 
homogenization for the solution of the first initial-boundary value problem \eqref{hyperbolic system}. Note that, if $\boldsymbol{\Phi}\in\mathrm{Dom}\,(B_D^0)^{2}$, then the function $\boldsymbol{\Phi}$ can be represented as $\boldsymbol{\Phi} =(B_D^0)^{-2}\boldsymbol{\check{\Phi}}$, where $\boldsymbol{\check{\Phi}}\in L_2(\mathcal{O};\mathbb{C}^n)$. 
By the theorems about regularity of solutions of the strongly elliptic systems (see \cite[Chapter~4]{McL}), if 
 $\partial\mathcal{O}\in C^{3,1}$, then $\mathrm{Dom}\,(B_D^0)^{2}\subset H^4(\mathcal{O};\mathbb{C}^n)$. So, in this case, by Lemma~\ref{Lemma (B_D0)2 and H4-norm},
\begin{equation}
\label{2.27a}
\Vert \boldsymbol{\check{\Phi}}\Vert _{L_2(\mathcal{O})}
=\Vert (B_D^0)^2\boldsymbol{\Phi}\Vert _{L_2(\mathcal{O})}
\leqslant \mathfrak{C}\Vert \boldsymbol{\Phi}\Vert _{H^4(\mathcal{O})}.
\end{equation}
Applying these considerations to the functions $\boldsymbol{\varphi}$, $\boldsymbol{\psi}$, and $\mathbf{F}(\cdot,t)$, using identities \eqref{u_eps=}, \eqref{u_0=}, and Theorem~\ref{Theorem cos New}, we obtain the following result.

\begin{theorem}
\label{Theorem solutions L2}
Let $\mathcal{O}\subset\mathbb{R}^d$ be a bounded domain of class $C^{3,1}$. Suppose that the assumptions of Subsec.~\textnormal{\ref{Subsection operatoer A_D,eps}--\ref{Subsection Effective operator}} are satisfied. 
Let $\mathbf{u}_\varepsilon $ be the solution of problem \eqref{hyperbolic system} and let $\mathbf{u}_0$ be the solution of the effective problem \eqref{effective hyperbolic system}, where $\boldsymbol{\varphi}$, $\boldsymbol{\psi}\in \mathrm{Dom}\,(B_D^0)^2$, and $\mathbf{F}\in L_{1,\mathrm{loc}}(\mathbb{R};\mathrm{Dom}\,(B_D^0)^2)$. Then for $t\in\mathbb{R}$ and $0<\varepsilon\leqslant\varepsilon _1$ we have
\begin{equation*}
\begin{split}
\Vert \mathbf{u}_\varepsilon (\cdot ,t)-\mathbf{u}_0(\cdot,t)\Vert _{L_2(\mathcal{O})}
&\leqslant \mathfrak{C} C_7
\varepsilon (1+\vert t\vert ^5)
\\
&\times
\left(
\Vert \boldsymbol{\varphi}\Vert _{H^4(\mathcal{O})}+\vert t\vert \Vert \boldsymbol{\psi}\Vert _{H^4(\mathcal{O})}+\vert t\vert \Vert \mathbf{F}\Vert _{L_1((0,t);H^4(\mathcal{O}))}\right).
\end{split}
\end{equation*}
The constants $\mathfrak{C}$ and $C_7$ depend only on the problem data \eqref{problem data}.
\end{theorem}

Using Theorems~\ref{Theorem cos New} and~\ref{Theorem sin corrector}, we obtain  approximation in the energy norm for the solution  $\mathbf{u}_\varepsilon$ of problem \eqref{hyperbolic system} with $\boldsymbol{\varphi}=0$.

\begin{theorem}
\label{Theorem solutions corrector}
Under the assumptions of Theorem~\textnormal{\ref{Theorem solutions L2}}, let $\boldsymbol{\varphi}=0$. Then for $t\in\mathbb{R}$ and $0<\varepsilon\leqslant\varepsilon _1$ we have
\begin{equation}
\label{Th sol energy 1}
\left\Vert \frac{\partial \mathbf{u}_\varepsilon}{\partial t}(\cdot ,t)-\frac{\partial \mathbf{u}_0}{\partial t}(\cdot ,t)\right\Vert _{L_2(\mathcal{O})}
\leqslant \mathfrak{C}C_7\varepsilon  (1+\vert t\vert ^5)
\left(\Vert \boldsymbol{\psi}\Vert _{H^4(\mathcal{O})}+\Vert \mathbf{F}\Vert _{L_1((0,t);H^4(\mathcal{O}))}\right).
\end{equation}
Let $\Lambda(\mathbf{x})$ and $\widetilde{\Lambda}(\mathbf{x})$ be the $\Gamma$-periodic solutions of problems \eqref{Lambda problem} and \eqref{tildeLambda_problem}, respectively. Let $P_\mathcal{O}$ be the linear continuous extension operator \eqref{P_O H^1, H^2} and let $S_\varepsilon$ be the Steklov smoothing operator \eqref{S_eps}. Put $\widetilde{\mathbf{u}}_0(\cdot ,t):=P_\mathcal{O}\mathbf{u}_0(\cdot,t)$. By $\mathbf{v}_\varepsilon (\cdot ,t)$ we denote the first order approximation for the solution  $\mathbf{u}_\varepsilon (\cdot ,t)$\textnormal{:}
\begin{equation*}
\widetilde{\mathbf{v}}_\varepsilon (\cdot ,t):=\widetilde{\mathbf{u}}_0(\cdot,t)+\varepsilon \Lambda^\varepsilon S_\varepsilon b(\mathbf{D})\widetilde{\mathbf{u}}_0(\cdot,t)
+\varepsilon \widetilde{\Lambda}^\varepsilon S_\varepsilon \widetilde{\mathbf{u}}_0(\cdot,t),
\quad
\mathbf{v}_\varepsilon (\cdot ,t):=\widetilde{\mathbf{v}}_\varepsilon (\cdot ,t)\vert _{\mathcal{O}}.
\end{equation*}
Then for $t\in\mathbb{R}$ and $0<\varepsilon\leqslant\varepsilon _1$ we have
\begin{equation}
\label{Th sol energy 2}
\Vert \mathbf{u}_\varepsilon (\cdot,t)-\mathbf{v}_\varepsilon (\cdot,t)\Vert _{H^1(\mathcal{O})}
\leqslant
\mathfrak{C} C_8\varepsilon ^{1/2}(1+ t ^6)
\left(\Vert \boldsymbol{\psi}\Vert _{H^4(\mathcal{O})}+\Vert \mathbf{F}\Vert _{L_1((0,t);H^4(\mathcal{O}))}\right).
\end{equation}
Let $\widetilde{g}(\mathbf{x})$ be the matrix-valued function \eqref{tilde g}. Let $\mathbf{p}_\varepsilon(\cdot ,t):=g^\varepsilon b(\mathbf{D})\mathbf{u}_\varepsilon (\cdot,t)$. Then for $t\in\mathbb{R}$ and $0<\varepsilon\leqslant\varepsilon _1$ we have
\begin{equation}
\label{Th sol energy 3}
\begin{split}
\Vert &\mathbf{p}_\varepsilon (\cdot ,t)-\widetilde{g}^\varepsilon S_\varepsilon b(\mathbf{D})\widetilde{\mathbf{u}}_0(\cdot,t)
-g^\varepsilon (b(\mathbf{D})\widetilde{\Lambda})^\varepsilon S_\varepsilon \widetilde{\mathbf{u}}_0(\cdot,t)\Vert _{L_2(\mathcal{O})}
\\
&\leqslant 
\mathfrak{C} C_9 \varepsilon ^{1/2}  (1+ t ^6)
\left(\Vert \boldsymbol{\psi}\Vert _{H^4(\mathcal{O})}+\Vert \mathbf{F}\Vert _{L_1((0,t);H^4(\mathcal{O}))}\right).
\end{split}
\end{equation}
The constants $\mathfrak{C}$, $C_8$, and $C_9$ depend only on the problem data \eqref{problem data}.
\end{theorem}

\begin{proof}
Estimates \eqref{Th sol energy 2} and \eqref{Th sol energy 3} follow from Lemma~\ref{Lemma (B_D0)2 and H4-norm}, Theorem~\ref{Theorem sin corrector},  and relations \eqref{u_eps=}, \eqref{u_0=}. 

Let us discuss the proof of inequality \eqref{Th sol energy 1}. We set $\boldsymbol{\varphi}=0$ in \eqref{u_eps=} and differentiate the obtained identity with respect to $t$. Then
\begin{equation*}
\frac{\partial \mathbf{u}_\varepsilon}{\partial t}(\cdot ,t)
=\cos (tB_{D,\varepsilon}^{1/2})\boldsymbol{\psi}
+\int _0^t \cos \left((t-\widetilde{t})B_{D,\varepsilon}^{1/2}\right)\mathbf{F}(\cdot ,\widetilde{t})\,d\widetilde{t}.
\end{equation*}
The similar identity holds for the solution of the effective problem. 
Together with Lemma~\ref{Lemma (B_D0)2 and H4-norm} and Theorem~\ref{Theorem cos New}, this implies estimate \eqref{Th sol energy 1}.
\end{proof}

From Theorem~\ref{Theorem d<=6 no S_eps} we derive the following result.

\begin{theorem}
\label{Theorem no_S_eps solutions}
Under the assumptions of Theorem~\textnormal{\ref{Theorem solutions corrector}}, let $d\leqslant8$. If $d=7,8$, we additionally assume that $\partial\mathcal{O}\in C^{4,1}$. Denote 
$
\check{\mathbf{v}}_\varepsilon (\cdot  ,t):=\mathbf{u}_0(\cdot,t)+\varepsilon (\Lambda ^\varepsilon b(\mathbf{D})+\widetilde{\Lambda}^\varepsilon )\mathbf{u}_0(\cdot ,t)$. 
Then for $t\in\mathbb{R}$ and $0<\varepsilon\leqslant\varepsilon _1$ we have
\begin{align*}
\Vert &\mathbf{u}_\varepsilon (\cdot ,t)-\check{\mathbf{v}}_\varepsilon (\cdot,t)\Vert _{H^1(\mathcal{O})}
\leqslant \mathfrak{C} C_{13}\varepsilon  ^{1/2} (1+t ^6)
\left(\Vert \boldsymbol{\psi}\Vert _{H^4(\mathcal{O})}+\Vert \mathbf{F}\Vert _{L_1((0,t);H^4(\mathcal{O}))}\right),
\\
\begin{split}
\Vert &\mathbf{p}_\varepsilon (\cdot ,t)-\widetilde{g}^\varepsilon b(\mathbf{D})\mathbf{u}_0(\cdot,t)-g^\varepsilon (b(\mathbf{D})\widetilde{\Lambda})^\varepsilon\mathbf{u}_0(\cdot ,t)\Vert _{L_2(\mathcal{O})}
\\
&\leqslant \mathfrak{C} C_{14} \varepsilon  ^{1/2}  (1+t ^6)
\left(\Vert \boldsymbol{\psi}\Vert _{H^4(\mathcal{O})}+\Vert \mathbf{F}\Vert _{L_1((0,t);H^4(\mathcal{O}))}\right).
\end{split}
\end{align*}
The constants $\mathfrak{C}$, $C_{13}$, and $C_{14}$ depend only on the problem data \eqref{problem data}.
\end{theorem}

\begin{remark}
If $\partial\mathcal{O}\in C^{1,1}$, the results of Theorems~\textnormal{\ref{Theorem solutions L2}, \ref{Theorem solutions corrector},} and \textnormal{\ref{Theorem no_S_eps solutions}} remain true with the norms $\Vert (B_D^0)^2\boldsymbol{\varphi}\Vert _{L_2(\mathcal{O})}$, $\Vert (B_D^0)^2\boldsymbol{\psi}\Vert _{L_2(\mathcal{O})}$, and $\Vert (B_D^0)^2\mathbf{F}\Vert _{L_1((0,t);L_2(\mathcal{O}))}$ instead of $\Vert \boldsymbol{\varphi}\Vert _{H^4(\mathcal{O})}$, $\Vert \boldsymbol{\psi}\Vert _{H^4(\mathcal{O})}$, and $\Vert \mathbf{F}\Vert _{L_1((0,t);H^4(\mathcal{O}))}$, respectively, in the error estimates.
\end{remark}

\subsection{The special case} 
Assume that $g^0=\underline{g}$, i.~e., relations \eqref{underline-g} are satisfied. Then, by Proposition~\ref{Proposition Lambda in L infty <=}($3^\circ$), Condition~\ref{Condition Lambda in L infty} holds. Herewith, according to \cite[Remark 3.5]{BSu05}, the matrix-valued function \eqref{tilde g} is constant and coincides with $g^0$, i.~e., $\widetilde{g}(\mathbf{x})=g^0=\underline{g}$. Thus, $\widetilde{g}^\varepsilon b(\mathbf{D})\mathbf{u}_0(\cdot,t)=g^0b(\mathbf{D})\mathbf{u}_0(\cdot,t)$.

In addition, suppose that
\begin{equation}
\label{sum Dj aj =0}
\sum _{j=1}^d D_j a_j(\mathbf{x})^* =0.
\end{equation}
Then the $\Gamma$-periodic solution of problem \eqref{tildeLambda_problem} is also equal to zero:  $\widetilde{\Lambda}(\mathbf{x})=0$. 
 So, Theorem~\ref{Theorem no S-eps} implies the following result.

\begin{proposition}
Under the assumptions of Theorem~\textnormal{\ref{Theorem solutions corrector}}, suppose that relations \eqref{underline-g} and \eqref{sum Dj aj =0} hold. Then for $t\in\mathbb{R}$ and $0<\varepsilon \leqslant \varepsilon _1$ we have
\begin{equation*}
\begin{split}
\Vert &\mathbf{p}_\varepsilon (\cdot ,t)-g^0 b(\mathbf{D})\mathbf{u}_0(\cdot,t)\Vert _{L_2(\mathcal{O})}
\leqslant \mathfrak{C} C_{12}\varepsilon  ^{1/2}  (1+ t ^6)
\left(\Vert \boldsymbol{\psi}\Vert _{H^4(\mathcal{O})}+\Vert \mathbf{F}\Vert _{L_1((0,t);H^4(\mathcal{O}))}\right).
\end{split}
\end{equation*}
\end{proposition}

\subsection{The case where the corrector is equal to zero} 
Assume that $g^0=\overline{g}$, i.~e., relations \eqref{overline-g} are satisfied. 
Assume that condition \eqref{sum Dj aj =0} holds. 
Then the $\Gamma$-periodic solutions of problems \eqref{Lambda problem} and~\eqref{tildeLambda_problem} are equal to zero: $\Lambda (\mathbf{x})=0$ and $\widetilde{\Lambda}(\mathbf{x})=0$. By Theorems~\ref{Theorem sin corrector} and \ref{Theorem cos corrector}, for $t\in\mathbb{R}$ and $0<\varepsilon\leqslant\varepsilon _1$ we have
\begin{align}
\label{sin with K=0}
\Bigl\Vert &\Bigl(
B_{D,\varepsilon}^{-1/2}\sin(tB_{D,\varepsilon}^{1/2})-(B_D^0)^{-1/2}\sin (t(B_D^0)^{1/2})
\Bigr)(B_D^0)^{-2}\Bigr\Vert
_{L_2(\mathcal{O})\rightarrow H^1(\mathcal{O})}
\leqslant 
C_8\varepsilon ^{1/2}(1+t ^6),
\\
\label{2.33 corr =0}
\Bigl\Vert &\Bigl(
\cos(tB_{D,\varepsilon}^{1/2})B_{D,\varepsilon}^{-1}-\cos(t(B_D^0)^{1/2})(B_D^0)^{-1}\Bigr)
(B_D^0)^{-1}\Bigr\Vert
 _{L_2(\mathcal{O})\rightarrow H^1(\mathcal{O})}
\leqslant
C_{10}\varepsilon ^{1/2}(1+\vert t\vert ^5).
\end{align}
In the case under consideration, Theorem~\ref{Theorem resolvent corrector}($2^\circ$) implies that
\begin{equation}
\Vert B_{D,\varepsilon}^{-1}-(B_D^0)^{-1}\Vert  _{L_2(\mathcal{O})\rightarrow H^1(\mathcal{O})}
\leqslant
C_4\max\lbrace 2;c_\flat ^{-1}+c_\flat ^{-2}\rbrace\varepsilon ^{1/2},\quad 0<\varepsilon\leqslant\varepsilon _1.
\end{equation}
Applying \eqref{H^1-norm <= BDeps^1/2} and \eqref{1.15a new} consistently, we obtain
\begin{equation}
\label{2.35 corr =0}
\begin{split}
\Vert &\cos(tB_{D,\varepsilon}^{1/2})(B_{D,\varepsilon}^{-1}-(B_D^0)^{-1})(B_D^0)^{-1}\Vert _{L_2(\mathcal{O})\rightarrow H^1(\mathcal{O})}
\\
&\leqslant c_3\Vert B_{D,\varepsilon}^{1/2}\cos(tB_{D,\varepsilon}^{1/2})(B_{D,\varepsilon}^{-1}-(B_D^0)^{-1})(B_D^0)^{-1}\Vert _{L_2(\mathcal{O})\rightarrow L_2(\mathcal{O})}
\\
&\leqslant c_3C_*^{1/2}\Vert B_{D,\varepsilon}^{-1}-(B_D^0)^{-1}\Vert _{L_2(\mathcal{O})\rightarrow H^1(\mathcal{O})}\Vert (B_D^0)^{-1}\Vert _{L_2(\mathcal{O})\rightarrow L_2(\mathcal{O})}.
\end{split}
\end{equation}
Combining \eqref{B_D^0 L2 ->L2} and \eqref{2.33 corr =0}--\eqref{2.35 corr =0}, for $t\in\mathbb{R}$ and $0<\varepsilon\leqslant\varepsilon _1$ we have
\begin{equation}
\label{cos with K =0}
\Bigl\Vert \Bigl(
\cos(tB_{D,\varepsilon}^{1/2})-\cos(t(B_D^0)^{1/2})\Bigr)(B_D^0)^{-2}\Bigr\Vert _{L_2(\mathcal{O})\rightarrow H^1(\mathcal{O})}
\leqslant
C_{15}\varepsilon ^{1/2}(1+\vert t\vert ^5).
\end{equation}
Here $C_{15}:=C_{10}+c_3C_*^{1/2}C_4\max\lbrace 2;c_\flat ^{-1}+c_\flat ^{-2}\rbrace\mathcal{C}_1$.

Bringing together \eqref{2.27a}, \eqref{sin with K=0}, and \eqref{cos with K =0}, we arrive at approximation in the Sobolev class $H^1(\mathcal{O};\mathbb{C}^n)$ for the solution \eqref{u_eps=} of the problem \eqref{hyperbolic system}.

\begin{proposition}
Let $\mathbf{u}_\varepsilon $ and $\mathbf{u}_0$ be solutions of problems \eqref{hyperbolic system} and \eqref{effective hyperbolic system}, respectively, for $\boldsymbol{\varphi}$, $\boldsymbol{\psi}\in \mathrm{Dom}\,(B_D^0)^2$, and $\mathbf{F}\in L_{1,\mathrm{loc}}(\mathbb{R};\mathrm{Dom}\,(B_D^0)^2)$. Assume that relations \eqref{overline-g} and \eqref{sum Dj aj =0} hold true. Then for $0<\varepsilon\leqslant\varepsilon _1$ and $t\in\mathbb{R}$ we have
\begin{equation*}
\begin{split}
\Vert \mathbf{u}_\varepsilon (\cdot ,t)-\mathbf{u}_0(\cdot ,t)\Vert _{H^1(\mathcal{O})}
&\leqslant \mathfrak{C}C_{15}\varepsilon ^{1/2}(1+\vert t\vert ^5)\Vert \boldsymbol{\varphi}\Vert _{H^4(\mathcal{O})}
\\
&+\mathfrak{C}C_8\varepsilon ^{1/2}(1+t ^6)\left(
\Vert \boldsymbol{\psi}\Vert _{H^4(\mathcal{O})}+ \Vert \mathbf{F}\Vert _{L_1((0,t);H^4(\mathcal{O}))}\right).
\end{split}
\end{equation*}
\end{proposition}

\section{Proof of Theorems \ref{Theorem cos New} and \ref{Theorem sin corrector}} 
\label{Section 3}

\subsection{Proof of Theorem \ref{Theorem cos New}}

To prove estimate \eqref{Th cos B_D,eps}, we use the inverse Laplace transform and Theorem~\ref{Theorem resolvent}. To guarantee the convergence of the corresponding integrals, we consider the function 
$\left(\cos (ta^{1/2})-1+at^2/2\right)a^{-2}$  
instead of the cosine. The reason is that the inverse Laplace transform of this function decreases faster than the inverse Laplace transform of the cosine (see, e.~g., \cite[Section 17.13]{Grad}).

\begin{proof}[Proof of Theorem \textnormal{\ref{Theorem cos New}}.]

For $t=0$, the result \eqref{Th cos B_D,eps} 
is trivial: $$\cos (t B_{D,\varepsilon}^{1/2})\vert _{t=0}=\cos (t(B_D^0)^{1/2})\vert _{t=0}=I.$$ Therefore, since the cosine is an even function, without loss of generality, we will further assume that $t>0$.

By \eqref{rho(zeta)} and \eqref{Th dr appr 1},
\begin{equation}
\label{AD,eps -1 -...}
\Vert B_{D,\varepsilon}^{-1}-(B_D^0)^{-1}\Vert _{L_2(\mathcal{O})\rightarrow L_2(\mathcal{O})}\leqslant C_{16}\varepsilon , \quad C_{16}:=\max \lbrace 1; c_\flat ^{-2}\rbrace C_2.
\end{equation}
So, by using the identity
\begin{equation*}
\begin{split}
B_{D,\varepsilon}^{-2}-(B_D^0)^{-2}
&=
\frac{1}{2}(B_{D,\varepsilon}^{-1}-(B_D^0)^{-1})(B_{D,\varepsilon}^{-1}+(B_D^0)^{-1})
\\
&+\frac{1}{2}(B_{D,\varepsilon}^{-1}+(B_D^0)^{-1})(B_{D,\varepsilon}^{-1}-(B_D^0)^{-1})
\end{split}
\end{equation*}
and estimates \eqref{B_D,eps L2 ->L2}, \eqref{B_D^0 L2 ->L2}, we obtain 
\begin{equation}
\label{raznost' kvadratov resolvent}
\Vert B_{D,\varepsilon}^{-2}-(B_D^0)^{-2}\Vert _{L_2(\mathcal{O})\rightarrow L_2(\mathcal{O})}
\leqslant 2\mathcal{C}_1C_{16}\varepsilon .
\end{equation}
Thus,
\begin{equation}
\label{cos-cos Res AD0}
\begin{split}
\Bigl\Vert & \left(\cos (tB_{D,\varepsilon}^{1/2})-\cos (t(B_D^0)^{1/2})\right)(B_D^0)^{-2}\Bigr\Vert _{L_2(\mathcal{O})\rightarrow L_2(\mathcal{O})}\\
&\leqslant \Vert \cos (tB_{D,\varepsilon}^{1/2})\left(B_{D,\varepsilon}^{-2}-(B_D^0)^{-2}\right)\Vert _{L_2(\mathcal{O})\rightarrow L_2(\mathcal{O})}\\
&+\Vert \cos (tB_{D,\varepsilon}^{1/2})B_{D,\varepsilon}^{-2}-\cos(t(B_D^0)^{1/2})(B_D^0)^{-2}\Vert _{L_2(\mathcal{O})\rightarrow L_2(\mathcal{O})}\\
&\leqslant 2\mathcal{C}_1C_{16}\varepsilon +\Vert \cos (tB_{D,\varepsilon}^{1/2})B_{D,\varepsilon}^{-2}-\cos(t(B_D^0)^{1/2})(B_D^0)^{-2}\Vert _{L_2(\mathcal{O})\rightarrow L_2(\mathcal{O})}.
\end{split}
\end{equation}

Let $a>0$ be a parameter. Then, by the residue theorem,
\begin{equation}
\label{basic identity for functions}
\left(
\frac{1}{2}a t^2-1+\cos (t\sqrt{a})
\right)a^{-2}
=\frac{1}{2\pi i}\int _{\mathrm{Re}\,\lambda = c} \lambda ^{-3}(a+\lambda ^2)^{-1}e^{\lambda t}\,d\lambda,\quad
c>0.
\end{equation}

Assume that the constant $c$ in \eqref{basic identity for functions} is equal to $\sqrt{c_\flat}/t$. With the help of the spectral theorem, from \eqref{basic identity for functions} we derive
\begin{equation}
\label{basic identity}
\left(
\frac{1}{2}B_{D,\varepsilon} t^2-I+\cos (t B_{D,\varepsilon}^{1/2})
\right)B_{D,\varepsilon}^{-2}
=\frac{1}{2\pi i}\int _{\mathrm{Re}\,\lambda = \sqrt{c_\flat}/t} \lambda ^{-3}(B_{D,\varepsilon}+\lambda ^2 I)^{-1}e^{\lambda t}\,d\lambda .
\end{equation}
The similar identity holds for the effective operator. So,
\begin{equation}
\label{14}
\begin{split}
\cos (&t B_{D,\varepsilon}^{1/2})B_{D,\varepsilon}^{-2}-\cos (t(B_D^0)^{1/2})(B_D^0)^{-2}
=
-\frac{t^2}{2}\left(B_{D,\varepsilon}^{-1}-(B_D^0)^{-1}\right)
+\left( B_{D,\varepsilon}^{-2}-(B_D^0)^{-2}\right)\\
&+\frac{1}{2\pi i}\int _{\mathrm{Re}\,\lambda =\sqrt{c_\flat}/t}\lambda ^{-3}
\left(
(B_{D,\varepsilon}+\lambda ^2I)^{-1}-(B_D^0+\lambda ^2 I)^{-1}\right)e^{\lambda t}\,d\lambda .
\end{split}
\end{equation}
Combining this with \eqref{AD,eps -1 -...} and \eqref{raznost' kvadratov resolvent}, we conclude that
\begin{equation}
\label{cos-cos=something + int norm estimate}
\begin{split}
\Vert & \cos (t B_{D,\varepsilon}^{1/2})B_{D,\varepsilon}^{-2}-\cos (t(B_D^0)^{1/2})(B_D^0)^{-2}\Vert _{L_2(\mathcal{O})\rightarrow L_2(\mathcal{O})}
\leqslant 2^{-1}C_{16}\varepsilon t^2+2\mathcal{C}_1C_{16}\varepsilon
\\
&+\left\Vert \frac{1}{2\pi i}\int _{\mathrm{Re}\,\lambda =\sqrt{c_\flat}/t}\lambda ^{-3}
\left(
(B_{D,\varepsilon}+\lambda ^2I)^{-1}-(B_D^0+\lambda ^2 I)^{-1}\right)e^{\lambda t}\,d\lambda \right\Vert _{L_2(\mathcal{O})\rightarrow L_2(\mathcal{O})}.
\end{split}
\end{equation}

Now we proceed to estimation of the integral in the right-hand side of \eqref{cos-cos=something + int norm estimate}. By the change of variables $\lambda t=\mu$,
\begin{equation}
\label{mathfrak I=}
\begin{split}
&\frac{1}{2\pi i}\int _{\mathrm{Re}\,\lambda =\sqrt{c_\flat}/t}\lambda ^{-3}
\left(
(B_{D,\varepsilon}+\lambda ^2I)^{-1}-(B_D^0+\lambda ^2 I)^{-1}\right)e^{\lambda t}\,d\lambda \\
&=\frac{t^2}{2\pi i}\int _{\mathrm{Re}\,\mu =\sqrt{c_\flat}}e^\mu \mu ^{-3}
\left(
(B_{D,\varepsilon}+\frac{\mu ^2}{t^2} I)^{-1}-(B_D^0+\frac{\mu ^2}{t ^2} I)^{-1}\right)\,d\mu=:\mathfrak{I}(\varepsilon ;t) .
\end{split}
\end{equation}

Substitute $\mu = \sqrt{c_\flat} +i\beta$, $\beta\in\mathbb{R}$. Denote $\zeta _t :=-\mu ^2/t^{2}$. Let us understand how the set of values of this variable looks like. We have
\begin{equation}
\label{3.8a}
\mu ^2 = c_\flat-\beta ^2 +i2\beta\sqrt{c_\flat}=:x+iy.
\end{equation}
Then
$$
\begin{cases}
x=c_\flat -\beta ^2,\\
y=2\beta\sqrt{c_\flat}.
\end{cases}
$$
So,
$$
\begin{cases}
x=c_\flat - (4c_\flat)^{-1}y^2,
\\
\beta =2^{-1}c_\flat ^{-1/2}y.
\end{cases}
$$
Thus, the values of $\zeta _t$ belong to the parabola $\Pi _t$:
\begin{equation}
\label{Pi}
\Pi _t:=\left\lbrace\zeta _t\in\mathbb{C} : \mathrm{Re}\,\zeta _t=-\frac{c_\flat}{t^2}+\frac{t^2}{4 c_\flat}(\mathrm{Im}\,\zeta _t)^2\right\rbrace .
\end{equation}

For $\zeta _t\in\Pi _t$ with $\mathrm{Re}\,\zeta _t<c_\flat +1$, we use approximation \eqref{Th dr appr 1} for the resolvent $(B_{D,\varepsilon}-\zeta _t I)^{-1}$. Let us estimate $\varrho _\flat (\zeta _t)$ for $\zeta _t\in\Pi _t$ under consideration. We have
\begin{equation}
\label{zeta_t(beta)=}
\zeta _t =-\frac{\mu ^2}{t^2}=\frac{\beta ^2 - c_\flat}{t^2}-i\frac{2\beta\sqrt{c_\flat}}{t^2},\quad \beta\in\mathbb{R}.
\end{equation}
So,
\begin{equation*}
\vert\zeta _t-c_\flat\vert^2 =\left(\frac{\beta ^2 - c_\flat}{t^2}-c_\flat\right)^2
+\frac{4\beta ^2c_\flat}{t^4}.
\end{equation*}
After elementary transformations, 
$$
\vert\zeta _t-c_\flat\vert^2 = t^{-4}\left(
(\beta ^2-c_\flat t^2 )^2
+2\beta ^2 c_\flat +c_\flat ^2+2 c_\flat ^2 t^2
\right).$$ 
Consequently, 
$
\vert\zeta _t -c_\flat\vert ^{-2}\leqslant(2c_\flat ^2)^{-1}t^2$.
For $\zeta _t\in \Pi _t$ with $\mathrm{Re}\,\zeta _t\leqslant c_\flat$, we use the estimate
\begin{equation}
\label{rho<=c_1}
\varrho _\flat (\zeta _t)
\leqslant\max\lbrace 1;\vert\zeta _t-c_\flat\vert ^{-2}\rbrace
\leqslant\max\lbrace 1;(2c_\flat ^2)^{-1}t^2\rbrace
\leqslant\mathfrak{c}_1(t^2+1);\quad\mathfrak{c}_1:=\max \lbrace 1;(2c_\flat ^2)^{-1}\rbrace .
\end{equation}
Let $\psi _t=\mathrm{arg}\,(\zeta _t-c_\flat)$. 
For $\zeta _t\in\Pi _t$ with $c_\flat <\mathrm{Re}\,\zeta _t\leqslant c_\flat +1$, the value of $\varrho _\flat (\zeta _t)$ can be estimated as follows:
\begin{equation}
\label{3.11b}
\varrho _\flat (\zeta _t)\leqslant \max\lbrace c(\psi _t)^2\vert \zeta _t -c_\flat\vert ^{-2};c(\psi _t)^2\rbrace
=\max\lbrace \vert \mathrm{Im}\,(\zeta _t-c_\flat)\vert ^{-2}; c(\psi _t)^2\rbrace .
\end{equation}
We have
\begin{equation}
\label{3.11a}
\vert\mathrm{Im}\,(\zeta _t-c_\flat)\vert ^{-2}=\vert\mathrm{Im}\,\zeta _t\vert ^{-2}
\leqslant\vert\mathrm{Im}\widehat{\zeta}_t\vert ^{-2},\quad\zeta _t\in\Pi _t,\; c_\flat <\mathrm{Re}\,\zeta _t\leqslant c_\flat +1.
\end{equation}
Here $\widehat{\zeta}_t$ is the point at the contour $\Pi _t$ such that $\mathrm{Re}\,\widehat{\zeta}_t =c_\flat$. (There are two such points, one can choose any.) Let $\widehat{\beta}\in\mathbb{R}$ be the corresponding value of the parameter $\beta$. Then 
$
\mathrm{Re}\,\widehat{\zeta}_t=c_\flat =t^{-2}(\widehat{\beta}^2-c_\flat)$. 
So, $\widehat{\beta}^2=c_\flat(1+t^2)$. And, by \eqref{3.11a}, for $\zeta _t\in\Pi _t$ under consideration we have
\begin{equation}
\label{zvezda}
\vert \mathrm{Im}\,(\zeta _t -c_\flat )\vert ^{-2}=
\vert\mathrm{Im}\,\zeta _t\vert ^{-2}
\leqslant\frac{t^4}{4c_\flat}\widehat{\beta}^{-2}
=\frac{ t^4}{4c_\flat ^2(1+t^2)}\leqslant\frac{t^2}{4c_\flat ^2}.
\end{equation}

Now we want to estimate $c(\psi _t)$. Obviously, for $\zeta _t\in\Pi _t$ with $c_\flat <\mathrm{Re}\,\zeta _t\leqslant c_\flat +1$ we have $c(\psi _t)\leqslant c(\widetilde{\psi}_t)$, where $\widetilde{\psi}_t=\mathrm{arg}\,\widetilde{\zeta}_t$, $\widetilde{\zeta}_t\in\Pi _t$, $\mathrm{Re}\,\widetilde{\zeta}_t=c_\flat +1$. (There are two such points on the contour.) Assume that the point $\widetilde{\zeta}_t\in\Pi _t$ corresponds to the parameter $\widetilde{\beta}>0$. Then 
$
c_\flat +1 =\mathrm{Re}\,\widetilde{\zeta}_t=t^{-2}(\widetilde{\beta}^2-c_\flat)$. 
So,
\begin{equation}
\label{tilde beta}
\widetilde{\beta}^2=c_\flat +t^2(c_\flat +1). 
\end{equation}
Thus,
\begin{equation*}
\left(
\mathrm{Im}\,\widetilde{\zeta}_t\right)^2
=\frac{4\widetilde{\beta}^2c_\flat}{t^4}
=
\frac{4 c_\flat ^2+4c_\flat t^2(c_\flat +1)}{t^4}.
\end{equation*}
Next,
\begin{equation*}
c(\widetilde{\psi}_t)^2
=\frac{\vert\widetilde{\zeta}_t-c_\flat\vert ^2}{\left(
\mathrm{Im}\,\widetilde{\zeta}_t\right)^2}
=\frac{\left(\mathrm{Re}\,(\widetilde{\zeta}_t-c_\flat)\right)^2+\left(
\mathrm{Im}\,\widetilde{\zeta}_t\right)^2}{\left(
\mathrm{Im}\,\widetilde{\zeta}_t\right)^2}
=\frac{t^4+4c_\flat ^2 +4c_\flat t^2(c_\flat +1)}{4 c_\flat ^2 +4c_\flat t^2(c_\flat +1)}.
\end{equation*}
By the elementary inequality $1\leqslant (c_\flat +1)^2$, 
\begin{equation*}
c(\widetilde{\psi}_t)^2
\leqslant
\frac{t^4(c_\flat +1)^2+4c_\flat ^2 +4c_\flat t^2(c_\flat +1)}{4 c_\flat ^2 +4c_\flat t^2(c_\flat +1)}=
\frac{\left(t^2(c_\flat +1)+2c_\flat\right)^2}{4 c_\flat ^2 +4c_\flat t^2(c_\flat +1)}.
\end{equation*}
By decreasing the denominator, we obtain
\begin{equation*}
c(\widetilde{\psi}_t)^2
\leqslant\frac{\left(t^2(c_\flat +1)+2c_\flat\right)^2}{2 c_\flat ^2 +c_\flat t^2(c_\flat +1)}
=\frac{t^2(c_\flat +1)+2c_\flat}{c_\flat}=2+(1+c_\flat ^{-1})t^2.
\end{equation*}
Thus, by \eqref{3.11b} and \eqref{zvezda}, for $\zeta _t\in \Pi _t$ with $c_\flat <\mathrm{Re}\,\zeta _t\leqslant c_\flat +1$ we have
\begin{equation}
\label{rho<=c_2}
\varrho _\flat (\zeta _t)
\leqslant \max\lbrace(2 c_\flat)^{-2}t^2;2+(1+c_\flat ^{-1})t^2\rbrace
\leqslant\mathfrak{c}_2(t^2+1);\quad\mathfrak{c}_2:=\max\lbrace (2c_\flat)^{-2};2;1+c_\flat ^{-1}\rbrace .
\end{equation}
Bringing together \eqref{rho<=c_1} and \eqref{rho<=c_2}, we arrive at the estimate
\begin{equation}
\label{rho <=c_3}
\varrho _\flat (\zeta _t)\leqslant\mathfrak{c}_3(1+t^2),\quad\zeta _t\in\Pi _t,\;\mathrm{Re}\,\zeta _t\leqslant c_\flat +1;\quad\mathfrak{c}_3:=\max\lbrace\mathfrak{c}_1;\mathfrak{c}_2\rbrace .
\end{equation}

Let now $\zeta _t\in\Pi _t$ with $\mathrm{Re}\,\zeta _t >c_\flat +1$. For this part of the contour, we will use \eqref{sem'.a} to estimate the integrand in \eqref{mathfrak I=}.   
Let $\phi _t=\mathrm{arg}\,\zeta _t$. 
By \eqref{zeta_t(beta)=}, 
\begin{equation}
\label{estimates large beta's}
t^2\vert\mu ^{-3}\vert\vert\zeta _t\vert ^{-1/2}c(\phi_t)^2=t^{-1}\vert\zeta _t\vert ^{-2}c(\phi_t)^2=t^{-1}\vert\mathrm{Im}\,\zeta _t\vert ^{-2}=(4c_\flat)^{-1}t^3\beta ^{-2}.
\end{equation} 

Now, we can estimate the integral \eqref{mathfrak I=}:
\begin{equation*}
\mathfrak{I}(\varepsilon ;t)
=\frac{t^2}{2\pi}\int _{-\infty}^\infty e^{\sqrt{c_\flat}}e^{i\beta}( \mu (\beta)) ^{-3}
\left(
(B_{D,\varepsilon}+\frac{\mu (\beta) ^2}{t^2} I)^{-1}-(B_D^0+\frac{\mu (\beta) ^2}{t ^2} I)^{-1}\right)\,d\beta .
\end{equation*}
Here $\mu (\beta)=c_\flat ^{1/2}+i\beta$. 
Combining this with \eqref{sem'.a}, \eqref{Th dr appr 1}, \eqref{rho <=c_3}, and \eqref{estimates large beta's}, we obtain
\begin{equation}
\label{I<= int1 +int 2}
\begin{split}
&\Vert\mathfrak{I}(\varepsilon ;t)\Vert _{L_2(\mathcal{O})\rightarrow L_2(\mathcal{O})}
\\
&\leqslant\frac{e^{\sqrt{c_\flat}}t^2}{2\pi}
\left(\varepsilon C_2\mathfrak{c}_3(1+t^2)\int _{-\widetilde{\beta}}^{\widetilde{\beta}}\vert\mu (\beta)\vert ^{-3} \,d\beta
+\varepsilon C_1 (2c_\flat)^{-1} t \int _{\widetilde{\beta}}^\infty \beta ^{-2}\,d\beta\right) .
\end{split}
\end{equation}
(Recall that $\beta=\widetilde{\beta}$ (see \eqref{tilde beta}) corresponds to the point $\widetilde{\zeta}_t\in \Pi _t$ with $\mathrm{Re}\,\widetilde{\zeta}_t=c_\flat +1$.) 
Note that $\vert \mu (\beta)\vert\geqslant \mathrm{Re}\,\mu (\beta)=c_\flat ^{1/2}$. So, by \eqref{tilde beta},
\begin{equation}
\int _{-\widetilde{\beta}}^{\widetilde{\beta}}\vert\mu (\beta)\vert ^{-3} \,d\beta
\leqslant 2c_\flat ^{-3/2}\widetilde{\beta}=2c_\flat ^{-3/2}\sqrt{c_\flat +t^2(c_\flat+1)}
\leqslant\mathfrak{c}_4 (t^2+1)^{1/2},
\end{equation}
where $\mathfrak{c}_4:=2c_\flat ^{-3/2}(c_\flat +1)^{1/2}$. 
Next, according to \eqref{tilde beta},
\begin{equation}
\label{int 2<=}
\int _{\widetilde{\beta}}^\infty \beta ^{-2}\,d\beta =\widetilde{\beta}^{-1}=(c_\flat +t^2(c_\flat +1))^{-1/2}
\leqslant t^{-1}(c_\flat +1)^{-1/2}.
\end{equation}

Bringing \eqref{I<= int1 +int 2}--\eqref{int 2<=} together, we obtain
\begin{equation}
\label{I itog estimate}
\Vert \mathfrak{I}(\varepsilon ;t)\Vert _{L_2(\mathcal{O})\rightarrow L_2(\mathcal{O})}
\leqslant\mathfrak{c}_5\varepsilon t^2 \left( (1+t^2)^{3/2}+1\right),
\end{equation}
where 
$
\mathfrak{c}_5:=(2\pi)^{-1}e^{\sqrt{c_\flat}}\max\lbrace \mathfrak{c}_3\mathfrak{c}_4 C_2;(2c_\flat )^{-1}(c_\flat +1)^{-1/2}C_1\rbrace $.

Combining
\eqref{cos-cos Res AD0}, \eqref{cos-cos=something + int norm estimate}, \eqref{mathfrak I=}, and \eqref{I itog estimate}, we arrive at the estimate
\begin{equation}
\label{Th cos B_D,eps different terms t}
\Bigl\Vert
\left(
\cos (t B_{D,\varepsilon}^{1/2})-\cos (t (B_D^0)^{1/2})
\right)(B_D^0)^{-2}\Bigr\Vert _{L_2(\mathcal{O})\rightarrow L_2(\mathcal{O})}
\leqslant \widehat{C}_7\varepsilon \left(1+t^2 +t^2(1+t^2)^{3/2}\right) 
\end{equation} 
with the constant 
$
\widehat{C}_7:=\max\left\lbrace 4\mathcal{C}_1C_{16};2^{-1}C_{16}+\mathfrak{c}_5\right\rbrace 
$. 
Note that for $\vert t\vert \leqslant 1$ the leading degree of $t$ in the right-hand side of  \eqref{Th cos B_D,eps different terms t} is $t^0$, but for $\vert t\vert >1$ the leading degree is $t^5$. So, \eqref{Th cos B_D,eps different terms t} implies the required estimate \eqref{Th cos B_D,eps} with the constant $C_7:=2(1+\sqrt{2})\widehat{C}_7$.

Combining the identity 
\begin{equation}
\label{sin tozd}
B_{D,\varepsilon}^{-1/2}\sin (tB_{D,\varepsilon }^{1/2})
=\int _0^t \cos (\tau B_{D,\varepsilon}^{1/2})\,d\tau,
\end{equation}
the similar identity for the effective operator, and \eqref{Th cos B_D,eps}, we arrive at estimate \eqref{Th sin 1}.
\end{proof}

\subsection{Proof of Theorem~\textnormal{\ref{Theorem cos corrector}}}

\begin{proof}[Proof of Theorem~\textnormal{\ref{Theorem cos corrector}}]
Without loss of generality, let $t> 0$. 
Similarly to \eqref{basic identity},
\begin{equation*}
\cos \bigl(t(B_D^0)^{1/2}\bigr)(B_D^0)^{-2}
=-\frac{t^2}{2}(B_D^0)^{-1}+(B_D^0)^{-2}
+\frac{1}{2\pi i}\int _{\mathrm{Re}\,\lambda =\sqrt{c_\flat} /t}\lambda ^{-3}(B_D^0+\lambda ^2 I)^{-1}e^{\lambda t}\,d\lambda .
\end{equation*}
This implies that
\begin{equation*}
\begin{split}
\varepsilon \bigl(\Lambda ^\varepsilon b(\mathbf{D})+\widetilde{\Lambda}^\varepsilon\bigr)S_\varepsilon P_\mathcal{O}\cos \bigl(t (B_D^0)^{1/2}\bigr)(B_D^0)^{-2}
&=-\frac{\varepsilon t^2}{2}K_D(\varepsilon ;0)+\varepsilon K_D(\varepsilon ;0)(B_D^0)^{-1}\\
&+\frac{\varepsilon}{2\pi i}\int _{\mathrm{Re}\,\lambda =\sqrt{c_\flat} /t}\lambda ^{-3}K_D(\varepsilon ;-\lambda ^2)e^{\lambda t}\,d\lambda .
\end{split}
\end{equation*}
Here $K_D(\varepsilon;\cdot)$ is the operator \eqref{K_D(eps,zeta)}. 
Therefore, by \eqref{14},
\begin{equation}
\label{33}
\begin{split}
&\cos (tB_{D,\varepsilon}^{1/2})B_{D,\varepsilon}^{-2}-\cos \bigl(t (B_D^0)^{1/2}\bigr)(B_D^0)^{-2}
-\varepsilon \bigl(\Lambda ^\varepsilon b(\mathbf{D})+\widetilde{\Lambda}^\varepsilon\bigr)S_\varepsilon P_\mathcal{O}\cos \bigl(t (B_D^0)^{1/2}\bigr)(B_D^0)^{-2}
\\
&=-\frac{t^2}{2}\left( B_{D,\varepsilon}^{-1}-(B_D^0)^{-1}-\varepsilon K_D(\varepsilon ;0)\right)
+\left(B_{D,\varepsilon}^{-2}-(B_D^0)^{-2}-\varepsilon K_D(\varepsilon ;0)(B_D^0)^{-1}\right)
\\
&+\frac{1}{2\pi i}\int _{\mathrm{Re}\,\lambda =\sqrt{c_\flat} /t}\lambda ^{-3}
\left(
(B_{D,\varepsilon }+\lambda ^2 I)^{-1}-(B_D^0+\lambda ^2 I)^{-1}-\varepsilon K_D(\varepsilon ;-\lambda ^2)\right)
e^{\lambda t}\,d\lambda.
\end{split}
\end{equation}
Denote the last summand in the right-hand side of \eqref{33} by $\mathcal{I}(\varepsilon ;t)$.

Combining \eqref{B_D,eps L2 ->H^1}, \eqref{B_D^0 L2 ->L2}, \eqref{Th dr appr 1}, and \eqref{31}, we obtain
\begin{equation}
\label{34}
\begin{split}
\Vert& B_{D,\varepsilon}^{-2}-(B_D^0)^{-2}-\varepsilon K_D(\varepsilon ;0)(B_D^0)^{-1}\Vert _{L_2(\mathcal{O})\rightarrow H^1(\mathcal{O})}
\\
&\leqslant \Vert B_{D,\varepsilon}^{-1}\bigl( B_{D,\varepsilon}^{-1}-(B_D^0)^{-1}\bigr)\Vert _{L_2(\mathcal{O})\rightarrow H^1(\mathcal{O})}
\\
&+\Vert \left(B_{D,\varepsilon}^{-1}-(B_D^0)^{-1}-\varepsilon K_D(\varepsilon ;0)\right)(B_D^0)^{-1}\Vert _{L_2(\mathcal{O})\rightarrow H^1(\mathcal{O})}
\\
&\leqslant
\mathcal{C}_2C_2\max\lbrace 1; c_\flat ^{-2}\rbrace\varepsilon +2\max\lbrace 1; c_\flat ^{-2}\rbrace\mathcal{C}_1C_4\varepsilon ^{1/2} \leqslant\mathfrak{c}_6\varepsilon ^{1/2},
\end{split}
\end{equation}
where $\mathfrak{c}_6:=\max\lbrace 1; c_\flat ^{-2}\rbrace(\mathcal{C}_2C_2+2\mathcal{C}_1C_4)$.

By \eqref{31}, \eqref{33}, and \eqref{34},
\begin{equation}
\label{35}
\begin{split}
\Vert &\cos (tB_{D,\varepsilon}^{1/2})(B_{D,\varepsilon})^{-2}-\cos \bigl(t (B_D^0)^{1/2}\bigr)(B_D^0)^{-2}
\\
&-\varepsilon \bigl(\Lambda ^\varepsilon b(\mathbf{D})+\widetilde{\Lambda}^\varepsilon\bigr)S_\varepsilon P_\mathcal{O}\cos \bigl(t (B_D^0)^{1/2}\bigr)(B_D^0)^{-2}\Vert _{L_2(\mathcal{O})\rightarrow H^1(\mathcal{O})}
\\
&\leqslant \max\lbrace 1; c_\flat ^{-2}\rbrace C_4\varepsilon ^{1/2} t^2 +\mathfrak{c}_6\varepsilon ^{1/2}+\Vert \mathcal{I}(\varepsilon ;t)\Vert _{L_2(\mathcal{O})\rightarrow H^1(\mathcal{O})}.
\end{split}
\end{equation}
Changing the variable  $\lambda t=\mu$ in the integral  $\mathcal{I}(\varepsilon ;t)$, we get
\begin{equation}
\label{36}
\mathcal{I}(\varepsilon ;t)
=\frac{t^2}{2\pi i}\int _{\mathrm{Re}\,\mu =\sqrt{c_\flat}}e^\mu \mu ^{-3}
\left(
(B_{D,\varepsilon}+\frac{\mu ^2}{t^2} I)^{-1}-(B_D^0+\frac{\mu ^2}{t^2}I)^{-1}-\varepsilon K_D\bigl(\varepsilon;-\frac{\mu ^2}{t^2}\bigr)\right)\,d\mu .
\end{equation}
Let $\mu =\mu (\beta)= c_\flat ^{1/2}+i\beta$, $\beta\in\mathbb{R}$. Then $\zeta _t (\beta)=-\mu (\beta)^2 /t^2$ lies on the parabola $\Pi _t$ (see \eqref{Pi}). Recall that $\widetilde{\beta}$ is defined by  \eqref{tilde beta}. 
For $-\widetilde{\beta}\leqslant\beta\leqslant\widetilde{\beta}$ we use estimate \eqref{31}. By \eqref{zeta_t(beta)=},
\begin{equation*}
\begin{split}
\varepsilon ^{1/2}\varrho _\flat (\zeta _t)^{1/2}+\varepsilon \vert 1+\zeta _t\vert ^{1/2}\varrho _\flat (\zeta _t)
&\leqslant
\varepsilon ^{1/2}\varrho _\flat (\zeta _t)(1+(1+\vert \zeta _t\vert )^{1/2})
\\
&\leqslant
\varepsilon ^{1/2}\varrho _\flat (\zeta _t)\left(1+\left(1+t^{-2}(c_\flat +\beta ^2)\right)^{1/2}\right).
\end{split}
\end{equation*}
Together with \eqref{31} and \eqref{rho <=c_3}, this implies
\begin{equation}
\label{37}
\begin{split}
\Vert &(B_{D,\varepsilon}-\zeta _t(\beta)I)^{-1}-(B_D^0-\zeta _t(\beta)I)^{-1}-\varepsilon K_D(\varepsilon ;\zeta _t(\beta))\Vert _{L_2(\mathcal{O})\rightarrow H^1(\mathcal{O})}
\\
&\leqslant 
\varepsilon ^{1/2}(1+t^2)\mathfrak{c}_3C_4\left(2+t^{-1}(c_\flat ^{1/2} +\widetilde{\beta})\right)
,\quad -\widetilde{\beta}\leqslant\beta\leqslant\widetilde{\beta}.
\end{split}
\end{equation}
For $\vert\beta\vert >\widetilde{\beta}$, we use \eqref{30}, the estimate
\begin{equation*}
\begin{split}
\vert \mu (\beta)\vert ^{-3} \frac{c(\phi _t)^2}{\vert \zeta _t (\beta)\vert ^{1/4}}
&=
t^{-3}\vert \zeta _t (\beta)\vert ^{-3/2}\frac{c(\phi _t)^2}{\vert \zeta _t (\beta)\vert ^{1/4}}
=
t^{-3}\frac{\vert\zeta _t (\beta)\vert ^{1/4}}{\vert\mathrm{Im}\,\zeta _t (\beta)\vert ^2}
\\
&=t^{-3}\frac{\left(t^{-2}(\beta ^2 +c_\flat)\right)^{1/4}}{t^{-4}4\beta^2 c_\flat}
\leqslant (4c_\flat)^{-1}t^{1/2}(\vert\beta\vert ^{-3/2}+c_\flat ^{1/4}\beta ^{-2}),
\end{split}
\end{equation*}
and the identity
\begin{equation*}
c(\phi_t)^{3/2}\vert\mu (\beta )\vert^{-3}=\frac{1}{t^3\vert\mathrm{Im}\,\zeta _t(\beta)\vert ^{3/2}}= 2^{-3/2}c_\flat ^{-3/4}\vert\beta\vert ^{-3/2}.
\end{equation*}
We obtain that
\begin{equation}
\label{38}
\begin{split}
\vert \mu (\beta)\vert ^{-3}\Vert (B_{D,\varepsilon}-\zeta _t(\beta )I)^{-1}-(B_D^0-\zeta _t(\beta) I)^{-1}-\varepsilon K_D(\varepsilon ;\zeta _t(\beta))\Vert _{L_2(\mathcal{O})\rightarrow H^1(\mathcal{O})}
\\
\leqslant C_5 \varepsilon ^{1/2}\left(t^{1/2}(4c_\flat)^{-1}(\vert\beta\vert ^{-3/2}+c_\flat ^{1/4}\beta ^{-2})
+2^{-3/2}c_\flat ^{-3/4}\vert\beta\vert ^{-3/2}\right),\quad \vert \beta\vert >\widetilde{\beta}.
\end{split}
\end{equation}
From \eqref{36}--\eqref{38} and the estimate $\vert \mu (\beta)\vert \geqslant c_\flat ^{1/2}$ it follows that
\begin{equation*}
\begin{split}
\Vert & \mathcal{I}(\varepsilon ;\zeta )\Vert _{L_2(\mathcal{O})\rightarrow H^1(\mathcal{O})}
\leqslant
\frac{t^2 e^{\sqrt{c_\flat}}}{2\pi}
\Bigl(
\mathfrak{c}_3C_4(1+t^2)\varepsilon ^{1/2}\left(2+t^{-1}(c_\flat ^{1/2} +\widetilde{\beta})\right)\int _{-\widetilde{\beta}}^{\widetilde{\beta}}\vert\mu (\beta)\vert ^{-3}\,d\beta
\\
&+2\varepsilon ^{1/2}t^{1/2}(4c_\flat )^{-1}C_5 \int _{\widetilde{\beta}}^\infty  (\beta ^{-3/2}+c_\flat ^{1/4}\beta ^{-2})\,d\beta
+2\varepsilon ^{1/2}C_5 2^{-3/2}c_\flat ^{-3/4}\int _{\widetilde{\beta}}^\infty \beta ^{-3/2}\,d\beta\Bigr)
\\
&\leqslant \frac{t^2 e^{\sqrt{c_\flat}}}{2\pi}
\Bigl(
\mathfrak{c}_3C_4(1+t^2)\varepsilon ^{1/2}c_\flat ^{-3/2}\left(2+t^{-1}(c_\flat ^{1/2} +\widetilde{\beta})\right)2\widetilde{\beta}
+\varepsilon ^{1/2}t^{1/2}c_\flat ^{-1}C_5\widetilde{\beta}^{-1/2}
\\
&+\varepsilon ^{1/2}t^{1/2}(2c_\flat )^{-1}c_\flat ^{1/4}C_5\widetilde{\beta}^{-1}
+2^{1/2}\varepsilon ^{1/2} C_5 c_\flat ^{-3/4}\widetilde{\beta}^{-1/2}\Bigr).
\end{split}
\end{equation*}
By \eqref{tilde beta}, 
$(1+t^2)\widetilde{\beta}\leqslant (1+t^2)^{3/2}(c_\flat +1)^{1/2}$ and 
$\widetilde{\beta}^{-1/2}\leqslant t^{-1/2}(c_\flat +1)^{-1/4}$. 
Then 
$t^{1/2}\widetilde{\beta}^{-1/2}\leqslant  (c_\flat +1)^{-1/4}$ and 
$t^{1/2}\widetilde{\beta}^{-1}\leqslant t^{-1/2}(c_\flat +1)^{-1/2}$. 
Therefore,
\begin{equation}
\label{39}
\Vert \mathcal{I}(\varepsilon ;t)\Vert _{L_2(\mathcal{O})\rightarrow H^1(\mathcal{O})}
\leqslant \mathfrak{c}_7\varepsilon ^{1/2}t^2\left((1+t^2)^{3/2}(1+t^{-1})+t^{-1/2}+1\right),
\end{equation}
where
\begin{equation*}
\begin{split}
\mathfrak{c}_7&:=(2\pi )^{-1}e^{\sqrt{c_\flat}}\max
\Bigl\lbrace
2\mathfrak{c}_3C_4c_\flat ^{-3/2}(c_\flat +1)^{1/2}\max\lbrace 2+(c_\flat +1)^{1/2};2c_\flat ^{1/2}\rbrace
;
\\
&(c_\flat +1)^{-1/4}c_\flat ^{-1}C_5;
2^{-1}c_\flat ^{-3/4}(c_\flat +1)^{-1/2}C_5
+2^{1/2} C_5 c_\flat ^{-3/4}(c_\flat +1)^{-1/4}\Bigr\rbrace .
\end{split}
\end{equation*}
Combining \eqref{35} and \eqref{39}, we arrive at the estimate
\begin{equation}
\label{Th cos with correction term t different terms}
\begin{split}
\bigl\Vert &\cos (tB_{D,\varepsilon}^{1/2})B_{D,\varepsilon}^{-2}-\cos \bigl(t (B_D^0)^{1/2}\bigr)(B_D^0)^{-2}
\\
&-\varepsilon \bigl(\Lambda ^\varepsilon b(\mathbf{D})+\widetilde{\Lambda}^\varepsilon\bigr)S_\varepsilon P_\mathcal{O}\cos \bigl(t (B_D^0)^{1/2}\bigr)(B_D^0)^{-2}\bigr\Vert _{L_2(\mathcal{O})\rightarrow H^1(\mathcal{O})}
\\
&\leqslant \widehat{C}_{10}\varepsilon ^{1/2}\left(1+t^2\bigl(t^{-1/2}+1+(1+t^2)^{3/2}(1+t^{-1})\bigr)\right).
\end{split}
\end{equation}
Here $\widehat{C}_{10}:=\max \bigl\lbrace C_4 \max\lbrace 1; c_\flat ^{-2}\rbrace  +\mathfrak{c}_7;\mathfrak{c}_6\bigr\rbrace$. Finally, by \eqref{H^1-norm <= BDeps^1/2}, \eqref{B_D,eps L2 ->L2}, and \eqref{AD,eps -1 -...},
\begin{equation*}
\begin{split}
\Vert &\cos (tB_{D,\varepsilon}^{1/2})B_{D,\varepsilon}^{-1}(B_{D,\varepsilon}^{-1}-(B_D^0)^{-1})\Vert _{L_2(\mathcal{O})\rightarrow H^1(\mathcal{O})}\\
&\leqslant c_3
\Vert B_{D,\varepsilon}^{-1/2}(B_{D,\varepsilon}^{-1}-(B_D^0)^{-1})\Vert _{L_2(\mathcal{O})\rightarrow L_2(\mathcal{O})}
\leqslant c_3\mathcal{C}_1^{1/2}C_{16} \varepsilon .
\end{split}
\end{equation*}
Together with \eqref{Th cos with correction term t different terms} this implies
\begin{equation}
\label{4.11}
\begin{split}
\Bigl\Vert &\Bigl(\cos (tB_{D,\varepsilon}^{1/2})B_{D,\varepsilon}^{-1}-\cos \bigl(t (B_D^0)^{1/2}\bigr)(B_D^0)^{-1}
\\
&-\varepsilon \bigl(\Lambda ^\varepsilon b(\mathbf{D})+\widetilde{\Lambda}^\varepsilon\bigr)S_\varepsilon P_\mathcal{O}\cos \bigl(t (B_D^0)^{1/2}\bigr)(B_D^0)^{-1}\Bigr)(B_D^0)^{-1}\Bigr\Vert _{L_2(\mathcal{O})\rightarrow H^1(\mathcal{O})}
\\
&\leqslant \widetilde{C}_{10}\varepsilon ^{1/2}\left(1+t^2\bigl(1+t^{-1/2}+(1+t^2)^{3/2}(1+t^{-1})\bigr)\right).
\end{split}
\end{equation}
Here $\widetilde{C}_{10}:=\widehat{C}_{10} +c_3\mathcal{C}_1^{1/2}C_{16} $. 
Finally, note that for $\vert t\vert <1$ the leading degree of $t$ in the right-hand side of \eqref{4.11} is $t^0$, and for $\vert t\vert \geqslant 1$ the leading degree is  $t^5$. Using this argument, from \eqref{4.11} we derive estimate \eqref{Th cos with correction term} with $C_{10}:=(3+2^{5/2})\widetilde{C}_{10}$.
\end{proof}

\subsection{Proof of Theorem \ref{Theorem sin corrector}}

\begin{proof}[Proof of Theorem \textnormal{\ref{Theorem sin corrector}}.]

By using Theorem \ref{Theorem cos corrector}, identity \eqref{sin tozd}, and the similar identity for the effective operator, 
we obtain
\begin{equation}
\label{int Th cos corrector}
\begin{split}
\Bigl\Vert &
\bigl( B_{D,\varepsilon}^{-1/2}\sin (tB_{D,\varepsilon }^{1/2})B_{D,\varepsilon}^{-1}-(B_D^0)^{-1/2}\sin (t (B_D^0)^{1/2})(B_D^0)^{-1}
\\
&-\varepsilon (\Lambda ^\varepsilon b(\mathbf{D})+\widetilde{\Lambda}^\varepsilon )S_\varepsilon P_\mathcal{O} (B_D^0)^{-1/2}\sin (t (B_D^0)^{1/2})(B_D^0)^{-1}\bigr)
(B_D^0)^{-1}\Vert _{L_2(\mathcal{O})\rightarrow H^1(\mathcal{O})}
\\
&\leqslant C_{10}\varepsilon ^{1/2} \vert t\vert (1+\vert t\vert ^5),\quad t\in\mathbb{R},\quad 0<\varepsilon\leqslant\varepsilon _1.
\end{split}
\end{equation}
Next, by \eqref{H^1-norm <= BDeps^1/2},  \eqref{B_D^0 L2 ->L2}, and  \eqref{AD,eps -1 -...},
\begin{equation}
\label{sin simple}
\Vert B_{D,\varepsilon}^{-1/2}\sin (t B_{D,\varepsilon}^{1/2})(B_{D,\varepsilon}^{-1}-(B_D^0)^{-1})(B_D^0)^{-1}\Vert _{L_2(\mathcal{O})\rightarrow H^1(\mathcal{O})}
\leqslant c_3C_{16}\mathcal{C}_1\varepsilon .
\end{equation}
Combining \eqref{int Th cos corrector} and \eqref{sin simple}, we arrive at estimate \eqref{Th sin 2} with the constant $$C_8:=2\left(C_{10}+c_3C_{16}\mathcal{C}_1\right).$$

Let us check inequality \eqref{Th fluxes operator terms}. By \eqref{b_l <=} and \eqref{Th sin 2}, for $t\in\mathbb{R}$ and $0<\varepsilon\leqslant\varepsilon _1$ we have
\begin{equation}
\label{flux proof start}
\begin{split}
\Bigl\Vert & \Bigl(
g^\varepsilon b(\mathbf{D})B_{D,\varepsilon}^{-1/2}\sin (t B_{D,\varepsilon}^{1/2})
\\
&
-g^\varepsilon b(\mathbf{D})\bigl( I+\varepsilon \Lambda ^\varepsilon b(\mathbf{D})S_\varepsilon P_\mathcal{O}+\varepsilon \widetilde{\Lambda}^\varepsilon S_\varepsilon P_\mathcal{O}\bigr)
(B_D^0)^{-1/2}\sin (t (B_D^0)^{1/2})
\Bigr)
\\
&\times(B_D^0)^{-2}
\Bigr\Vert _{L_2(\mathcal{O})\rightarrow L_2(\mathcal{O})}
\leqslant(d\alpha _1)^{1/2}C_{8}\Vert g\Vert _{L_\infty}
\varepsilon ^{1/2}
 (1+t ^6).
\end{split}
\end{equation}
Obviously,
\begin{equation}
\label{flux tozd}
\begin{split}
g^\varepsilon & b(\mathbf{D})\bigl( I+\varepsilon \Lambda ^\varepsilon b(\mathbf{D})S_\varepsilon P_\mathcal{O}+\varepsilon \widetilde{\Lambda}^\varepsilon S_\varepsilon P_\mathcal{O}\bigr)
(B_D^0)^{-1/2}\sin (t (B_D^0)^{1/2})
(B_D^0)^{-2}
\\
&=
g^\varepsilon b(\mathbf{D})(B_D^0)^{-5/2}\sin (t (B_D^0)^{1/2})
+
g^\varepsilon \left(b(\mathbf{D})\Lambda\right)^\varepsilon S_\varepsilon b(\mathbf{D})P_\mathcal{O}
(B_D^0)^{-5/2}\sin (t (B_D^0)^{1/2})
\\
&+g^\varepsilon \left(b(\mathbf{D})\widetilde{\Lambda}\right)^\varepsilon S_\varepsilon P_\mathcal{O}
(B_D^0)^{-5/2}\sin (t (B_D^0)^{1/2})
\\
&+\varepsilon \sum _{l=1}^d g^\varepsilon b_l\left(\Lambda ^\varepsilon S_\varepsilon b(\mathbf{D})D_l+\widetilde{\Lambda}^\varepsilon S_\varepsilon D_l\right)
P_\mathcal{O}(B_D^0)^{-5/2}\sin (t (B_D^0)^{1/2}).
\end{split}
\end{equation}
The fourth summand in the right-hand side of \eqref{flux tozd} can be estimated with the help of \eqref{b_l <=}, \eqref{Lambda S_eps<=}, and \eqref{tildeLambda S_eps<=}:
\begin{equation}
\label{4 summand flux}
\begin{split}
\Biggl\Vert &
\varepsilon \sum _{l=1}^d g^\varepsilon b_l\left(\Lambda ^\varepsilon S_\varepsilon b(\mathbf{D})D_l+\widetilde{\Lambda}^\varepsilon S_\varepsilon D_l\right)
P_\mathcal{O}(B_D^0)^{-5/2}\sin (t (B_D^0)^{1/2})\Biggr\Vert _{L_2(\mathcal{O})\rightarrow L_2(\mathbb{R}^d)}
\\
&\leqslant
\varepsilon (d\alpha _1)^{1/2} \Vert g\Vert _{L_\infty}
M_1\Vert b(\mathbf{D})\mathbf{D}P_\mathcal{O}(B_D^0)^{-5/2}\sin (t (B_D^0)^{1/2})\Vert _{L_2(\mathcal{O})\rightarrow L_2(\mathbb{R}^d)}
\\
&+\varepsilon (d\alpha _1)^{1/2} \Vert g\Vert _{L_\infty}\widetilde{M}_1  \Vert \mathbf{D}P_\mathcal{O}(B_D^0)^{-5/2}\sin (t (B_D^0)^{1/2})\Vert _{L_2(\mathcal{O})\rightarrow L_2(\mathbb{R}^d)}.
\end{split}
\end{equation}
Combining \eqref{<b^*b<}, \eqref{PO}, \eqref{proof no S_eps-2}, and
\eqref{4 summand flux}, we get
\begin{equation}
\label{4 summand flux itog}
\begin{split}
\Biggl\Vert &
\varepsilon \sum _{l=1}^d g^\varepsilon b_l\left(\Lambda ^\varepsilon S_\varepsilon b(\mathbf{D})D_l+\widetilde{\Lambda}^\varepsilon S_\varepsilon D_l\right)
P_\mathcal{O}(B_D^0)^{-5/2}\sin (t (B_D^0)^{1/2})\Biggr\Vert _{L_2(\mathcal{O})\rightarrow L_2(\mathbb{R}^d)}
\\
&\leqslant
\varepsilon \vert t\vert \widehat{C}_9,\quad t\in\mathbb{R},
\quad 0<\varepsilon\leqslant 1
,
\end{split}
\end{equation}
where $\widehat{C}_9:=(d\alpha _1)^{1/2}\Vert g\Vert _{L_\infty}\bigl(M_1\alpha _1^{1/2}C_\mathcal{O}^{(2)}
+\widetilde{M}_1C_\mathcal{O}^{(1)}\bigr)\mathcal{C}_1\mathcal{C}_3 $.

By Proposition \ref{Proposition S__eps - I} and \eqref{<b^*b<}, \eqref{PO}, \eqref{proof no S_eps-2},
\begin{equation}
\label{flux proof final}
\begin{split}
\Vert & g^\varepsilon b(\mathbf{D})(S_\varepsilon -I)P_\mathcal{O}(B_D^0)^{-5/2}\sin (t (B_D^0)^{1/2})\Vert _{L_2(\mathcal{O})\rightarrow L_2(\mathbb{R}^d)}
\\
&\leqslant 
\varepsilon r_1\Vert g\Vert _{L_\infty}\Vert \mathbf{D} b(\mathbf{D})P_\mathcal{O}(B_D^0)^{-5/2}\sin (t (B_D^0)^{1/2})\Vert _{L_2(\mathcal{O})\rightarrow L_2(\mathbb{R}^d)}\\
&\leqslant
\varepsilon \vert t\vert r_1\Vert g\Vert _{L_\infty}\alpha _1^{1/2}C_\mathcal{O}^{(2)}\mathcal{C}_1\mathcal{C}_3.
\end{split}
\end{equation}
Combining \eqref{tilde g}, \eqref{flux proof start}, \eqref{flux tozd}, \eqref{4 summand flux itog}, and \eqref{flux proof final}, we arrive at the required  inequality \eqref{Th fluxes operator terms} with the constant $C_9:=(d\alpha _1)^{1/2}C_8\Vert g\Vert _{L_\infty}+\widehat{C}_9+r_1\Vert g\Vert _{L_\infty}\alpha _1^{1/2}C_\mathcal{O}^{(2)}\mathcal{C}_1\mathcal{C}_3$.
\end{proof}

\end{document}